\documentclass[11pt,english]{article}
\usepackage[T1]{fontenc}
\usepackage[latin9]{inputenc}
\usepackage{geometry}
\usepackage{color}
\usepackage{babel}
\usepackage{mathtools}
\usepackage{enumitem}
\usepackage{amsmath}
\usepackage{amsthm}
\usepackage{amssymb}
\usepackage{xargs}[2008/03/08]
\usepackage[unicode=true,pdfusetitle,
 bookmarks=true,bookmarksnumbered=false,bookmarksopen=false,
 breaklinks=false,pdfborder={0 0 1},backref=false,colorlinks=true]
 {hyperref}
\hypersetup{
 citecolor=blue, linkcolor=red}

\makeatletter
\theoremstyle{plain}
\newtheorem{thm}{\protect\theoremname}
  \theoremstyle{remark}
  \newtheorem{rem}{\protect\remarkname}
  \theoremstyle{plain}
  \newtheorem{cor}{\protect\corollaryname}
  \theoremstyle{plain}
  \newtheorem{fact}{\protect\factname}
  \theoremstyle{plain}
  \newtheorem{lem}{\protect\lemmaname}
  \theoremstyle{plain}
  \newtheorem{prop}{\protect\propositionname}

\usepackage[mathscr]{euscript}
\usepackage{amsfonts}
\usepackage{dsfont}
\usepackage{times}
\usepackage{amsmath}
\usepackage{bbm}

\pdfmapfile{+txfonts.map}
\usepackage{txfonts}
\usepackage{appendix}
\allowdisplaybreaks

\theoremstyle{definition}
\newtheorem{mdl}{Model}

\makeatother

  \providecommand{\factname}{Fact}
  \providecommand{\lemmaname}{Lemma}
  \providecommand{\propositionname}{Proposition}
  \providecommand{\remarkname}{Remark}
\providecommand{\corollaryname}{Corollary}
\providecommand{\theoremname}{Theorem}

\begin{document}
\global\long\def\E{\mathbb{E}}
\global\long\def\P{\mathbb{P}}
\global\long\def\Var{\operatorname*{Var}}
\global\long\def\Cov{\operatorname*{Cov}}
\global\long\def\Tr{\operatorname*{Tr}}
\global\long\def\diag{\operatorname*{diag}}
\newcommandx\norm[2][usedefault, addprefix=\global, 1=\#1]{\Vert#1\|_{#2}}
\global\long\def\opnorm#1{\norm[#1]{\textup{op}}}
\global\long\def\rank#1{\operatorname*{rank}(#1)}
\global\long\def\md#1{\operatorname*{maxdiag}(#1)}
\global\long\def\indic{\operatorname*{\mathbb{I}}}
\global\long\def\diff{\operatorname{d}\!}
\global\long\def\argmax{\operatorname*{arg\,max}}
\global\long\def\argmin{\operatorname*{arg\,min}}
\global\long\def\Bern{\operatorname{Ber}}
\global\long\def\pos{\operatorname{pos}}
\global\long\def\round{\operatorname{round}}
\global\long\def\sign{\operatorname{sign}}

\global\long\def\mtx#1{\bm{#1}}
\global\long\def\vct#1{\bm{#1}}
\global\long\def\real{\mathbb{R}}

\global\long\def\A{\mathbf{A}}
\global\long\def\Atilde{\widetilde{\A}}
\global\long\def\B{\mathbf{B}}
\global\long\def\D{\mathbf{D}}
\global\long\def\G{\mathbf{G}}
\global\long\def\H{\mathbf{H}}
\global\long\def\I{\mathbf{I}}
\global\long\def\J{\mathbf{J}}
\global\long\def\K{\mathbf{K}}
\global\long\def\L{\mathbf{L}}
\global\long\def\M{\mathbf{M}}
\global\long\def\P{\mathbb{P}}
\global\long\def\O{\mathbf{O}}
\global\long\def\Q{\mathbf{Q}}
\global\long\def\R{\mathbf{R}}
\global\long\def\U{\mathbf{U}}
\global\long\def\Uperp{\mathbf{U}_{\perp}}
\global\long\def\V{\mathbf{V}}
\global\long\def\Vtilde{\widetilde{\mathbf{V}}}
\global\long\def\Vprime{\mathbf{V}'}
\global\long\def\W{\mathbf{W}}
\global\long\def\Wtilde{\widetilde{\W}}
\global\long\def\X{\mathbf{X}}
\global\long\def\Y{\mathbf{Y}}
\global\long\def\Z{\mathbf{Z}}
\global\long\def\one{\mathbf{1}}
\global\long\def\calV{\mathcal{V}}
\global\long\def\calR{\mathcal{R}}
\global\long\def\calS{\mathcal{S}}
\global\long\def\calT{\mathcal{T}}
\global\long\def\calVrow{\mathcal{V}_{\mbox{row}}}
\global\long\def\calVcol{\mathcal{V}_{\mbox{col}}}
\global\long\def\d{\mathbf{d}}
\global\long\def\u{\mathbf{u}}
\global\long\def\v{\mathbf{v}}
\global\long\def\x{\mathbf{x}}
\global\long\def\z{\mathbf{z}}

\global\long\def\Yhat{\widehat{\Y}}
\global\long\def\Ystar{\Y^{*}}
\global\long\def\Adj{\A}
\global\long\def\ObsAdj{\B}
\global\long\def\Noise{\W}
\global\long\def\HalfAdj{\boldsymbol{\Lambda}}
\global\long\def\HalfNoise{\boldsymbol{\Psi}}
\global\long\def\Yhp{\widehat{\Y}_{\perp}}
\global\long\def\Yround{\widehat{\Y}^{\text{R}}}
\global\long\def\CensorMat{\Z}
\global\long\def\diffmat{\Delta}
\global\long\def\Diffmat{\boldsymbol{\diffmat}}
\global\long\def\AdjH{\Adj^{\text{H}}}
\global\long\def\AdjSR{\Adj^{\text{SR}}}

\global\long\def\yhat{\widehat{Y}}
\global\long\def\ystar{Y^{*}}
\global\long\def\adj{A}
\global\long\def\adjh{A^{\text{H}}}
\global\long\def\adjsr{A^{\text{SR}}}
\global\long\def\obsadj{B}
\global\long\def\noise{W}
\global\long\def\halfadj{\Lambda}
\global\long\def\halfnoise{\Psi}
\global\long\def\yround{\widehat{Y}^{\text{R}}}
\global\long\def\censormat{Z}

\global\long\def\OneMat{\J}
\global\long\def\onemat{J}
\global\long\def\onevec{\mathbf{1}}
\global\long\def\indic{\mathbb{I}}
\global\long\def\IdMat{\I}

\global\long\def\num{n}
\global\long\def\size{\ell}
\global\long\def\numclust{k}
\global\long\def\snr{{I^{*}}}
\global\long\def\inprob{p}
\global\long\def\outprob{q}
\global\long\def\flipprob{\epsilon}
\global\long\def\obsprob{\alpha}
\global\long\def\apxconst{\rho}
\global\long\def\error{\gamma}
\global\long\def\LabelStar{\boldsymbol{\sigma}^{*}}
\global\long\def\labelstar{\sigma^{*}}
\global\long\def\LabelHat{\widehat{\boldsymbol{\sigma}}}
\global\long\def\labelhat{\widehat{\sigma}}
\global\long\def\LabelSDP{\LabelHat^{\textup{sdp}}}
\global\long\def\labelsdp{\labelhat^{\text{sdp}}}
\global\long\def\LabelMLE{\LabelHat^{\textup{mle}}}
\global\long\def\labelmle{\labelhat^{\text{mle}}}
\global\long\def\proxnum{\bar{\num}}

\global\long\def\z{\mathbf{z}}
\global\long\def\labels{\sigma}
\global\long\def\labeltilde{\tilde{\sigma}}
\global\long\def\Label{\boldsymbol{\sigma}}
\global\long\def\Labeltilde{\tilde{\Label}}
\global\long\def\misrate{\operatorname*{\texttt{err}}}
\global\long\def\apkmedian{\texttt{\ensuremath{\apxconst}-kmed}}
\global\long\def\std{\tau}
\global\long\def\renyi{I}
\global\long\def\tune{\lambda^{*}}

\global\long\def\PT{\mathcal{P}_{T}}
\global\long\def\PTperp{\mathcal{P}_{T^{\perp}}}
\global\long\def\positify{{\cal T}}
\global\long\def\calJ{{\cal J}}
\global\long\def\calL{\mathcal{L}}
\global\long\def\calM{{\cal M}}
\global\long\def\calY{{\cal Y}}
\global\long\def\paramset{\Theta}
\global\long\def\prior{\phi}
\global\long\def\bayrisk{B}

\global\long\def\constsumthreeterms{72}
\global\long\def\constoperatornormW{182}
\global\long\def\t{{\displaystyle ^{\top}}}
\global\long\def\Abar{\bar{A}}
\global\long\def\bbar{\bar{b}}
\global\long\def\ybar{\bar{y}}
\global\long\def\xbar{\bar{x}}
\global\long\def\xtilde{\tilde{x}}
\global\long\def\constp{C_{\inprob}}
\global\long\def\consts{C_{\snr}}
\global\long\def\conste{C_{e}}
\global\long\def\constu{c_{u}}
\global\long\def\constdense{c_{d}}
\global\long\def\constgamma{C_{g}}
\global\long\def\constmis{C_{m}}
\global\long\def\constStwo{C_{S_{2}}}
\global\long\def\constpilot{C_{\text{pilot}}}
\global\long\def\consteta{C_{\text{pilot}}'}

\global\long\def\pairset{\mathcal{L}}
 \global\long\def\betterset{\mathcal{Y}}

\global\long\def\var{H}

\title{Achieving the Bayes Error Rate in Synchronization and Block Models
by SDP, Robustly}

\author{Yingjie Fei and Yudong Chen\\
School of Operations Research and Information Engineering\\
Cornell University\\
\{yf275,yudong.chen\}@cornell.edu}

\date{}
\maketitle
\begin{abstract}
We study the statistical performance of semidefinite programming (SDP)
relaxations for clustering under random graph models. Under the $\mathbb{Z}_{2}$
Synchronization model, Censored Block Model and Stochastic Block Model,
we show that SDP achieves an error rate of the form 
\[
\exp\Big[-\big(1-o(1)\big)\proxnum\snr\Big].
\]
Here $\proxnum$ is an appropriate multiple of the number of nodes
and $\snr$ is an information-theoretic measure of the signal-to-noise
ratio. We provide matching lower bounds on the Bayes error for each
model and therefore demonstrate that the SDP approach is Bayes optimal.
As a corollary, our results imply that SDP achieves the optimal exact
recovery threshold under each model. Furthermore, we show that SDP
is robust: the above bound remains valid under semirandom versions
of the models in which the observed graph is modified by a monotone
adversary. Our proof is based on a novel primal-dual analysis of SDP
under a unified framework for all three models, and the analysis shows
that SDP tightly approximates a joint majority voting procedure.
\end{abstract}

\section{Introduction\label{sec:intro}}

Clustering and community detection in graphs is an important problem
lying at the intersection of computer science, optimization, statistics
and information theory. Random graph models provide a venue for studying
the average-case behavior of these problems. In these models, noisy
pairwise observations are generated randomly according to the unknown
clustering structure of the nodes. In its basic form, such a model
involves $\num$ nodes divided into two clusters, which can be represented
by a vector $\LabelStar\in\left\{ \pm1\right\} ^{\num}$. For each
pair of nodes $i$ and $j$, one observes a number $\adj_{ij}\in\real$
generated independently based on the sign of $\labelstar_{i}\labelstar_{j}$,
that is, whether the two nodes are in the same cluster or not. Given
one realization of the random graph $\Adj=(\adj_{ij})\in\real^{\num\times\num}$,
the goal is to estimate the vector $\LabelStar$, or equivalently,
the matrix $\Ystar\coloneqq(\labelstar_{i}\labelstar_{j})\in\{\pm1\}^{\num\times\num}$.
Among the most popular random graph models are the $\mathbb{Z}_{2}$
Synchronization (Z2) model, Censored Block Model (CBM) and Stochastic
Block Model (SBM), where $\adj_{ij}$ follows the Gaussian, censored
$\pm1$ and Bernoulli distributions, respectively (see Section~\ref{sec:setup}
for the details). We consider these three models in this paper. 

Clustering is a challenging problem involving discrete and hence non-convex
optimization. SDP relaxations have emerged as an efficient and robust
approach to this problem, and recent work has witnessed the advances
in establishing rigorous performance guarantees for SDP (see Section~\ref{sec:related}
for a review of this literature). Such guarantees are typically stated
in terms of a signal-to-noise ratio (SNR) measure $\snr$ that depends
on the specific random model (see Equation~(\ref{eq:snr})). In terms
of controlling the \emph{estimation error} of SDP, the best and most
general results to date are given in the line of work in \cite{Guedon2015,fei2019exponential},
which proves that the optimal SDP solution $\Yhat$ satisfies the
bound
\begin{equation}
\misrate(\LabelSDP,\LabelStar)\lesssim\frac{1}{\num^{2}}\norm[\Yhat-\Ystar]1\lesssim\exp\left[-\frac{\num\snr}{C}\right],\label{eq:suboptimal_bound}
\end{equation}
where $C>0$ is a large constant, $\norm[\cdot]1$ denotes the entrywise
$\ell_{1}$ norm, and $\misrate(\LabelSDP,\LabelStar)$ denotes the
fraction of nodes mis-clustered by an estimate $\LabelSDP\in\{\pm1\}^{\num}$,
extracted from $\Yhat$, of the ground-truth cluster labels $\LabelStar$.
The above result is, however, unsatisfactory due to the presence of
a large multiplicative constant $C$ in the exponent, rendering the
bound fundamentally sub-optimal. In particular, the interesting regime
for proving an error bound is when $n\snr\le2\log\num$, as otherwise
SDP is already known to attain zero error. With a large $C$ in the
exponent, the result in (\ref{eq:suboptimal_bound}) provides a rather
loose, sometimes even uninformative,\footnote{Note that $\frac{1}{\num^{2}}\norm[\Yhat-\Ystar]1$ is trivially upper
bounded by $2$ since $\Yhat,\Ystar\in[-1,1]^{\num\times\num}$.} bound in this regime. Moreover, this sub-optimality is intrinsic
to the proof techniques used and cannot be avoided simply by more
careful calculations. 

In this paper, we establish a strictly tighter, and essentially optimal,
error bound on SDP. Let $\proxnum=\num$ for Z2 and CBM, and $\proxnum=\frac{\num}{2}$
for SBM. 
\begin{thm}[Informal]
\label{thm:informal}As $\num\to\infty$, with probability tending
to one, the optimal solution $\Yhat$ of the SDP relaxation satisfies
\begin{equation}
\frac{1}{\num^{2}}\norm[\Yhat-\Ystar]1\le\exp\Big[-\big(1-o(1)\big)\proxnum\snr\Big],\label{eq:informal_SDP_error}
\end{equation}
Moreover, the explicit label estimate $\LabelSDP$ computed by taking
entrywise signs of the top eigenvector of $\Yhat$ satisfies
\begin{equation}
\misrate(\LabelSDP,\LabelStar)\le\exp\Big[-\big(1-o(1)\big)\proxnum\snr\Big].\label{eq:informal_misrate}
\end{equation}
\end{thm}
In all three models, the error exponent $\snr$ is a form of Renyi
divergence. See Theorem \ref{thm:SDP_error} for the precise statement
of our results as well as an explicit, non-asymptotic estimate of
the $o(1)$ term. One should compare this result with the following
minimax lower bound, which shows that any estimator $\LabelHat$ must
incur an error\emph{
\begin{equation}
\misrate(\LabelHat,\LabelStar)\ge\exp\Big[-\big(1+o(1)\big)\proxnum\snr\Big].\label{eq:informal_lower}
\end{equation}
}as the latter represents the best achievable \emph{Bayes risk} of
the problem. For SBM, this bound is established in \cite{zhang2016minimax};
for Z2 and CBM, the above lower bound is new and formally established
in Theorem~\ref{thm:minimax_lower}. In view of the above upper and
lower bound, we see that SDP achieves the optimal Bayes error under
all three models. 

\paragraph{Optimality as a surprise?}

The result above has come as unexpected to us, as it shows that relaxing
the original discrete clustering problem via SDP incurs essentially
no loss in terms of statistical accuracy. As we discuss in Section~\ref{sec:highlight}
and further elaborate in Section~\ref{sec:proof_sketch}, we prove
this result by showing, via a novel primal-dual analysis, that SDP
tightly approximates a majority voting procedure, and this procedure
leads to the optimal error exponent $\snr$.  Interestingly, our
analysis is not tethered to the optimality of $\Yhat$ to the SDP;
rather, it only relies on the fact that $\Yhat$ is feasible and no
worse in objective value than $\Ystar$, and thus the bounds (\ref{eq:informal_SDP_error})
and (\ref{eq:informal_misrate}) in fact hold for any matrix $\Y$
with these two properties. This kind of leeway in the analysis makes
the bounds robust, as we elaborate next.

\paragraph{Robustness.}

We show that the bounds in Theorem \ref{thm:informal} continue to
hold under the so-called \emph{monotone semirandom model}~\cite{feige2001semirandom},
where an adversary is allowed to make arbitrary changes to the graph
in a way that apparently strengthens connections within each cluster
and weakens connections between clusters. While this model seemingly
makes the clustering problem easier, they in fact foil, provably,
many existing algorithms, particularly those that over-exploit the
specific structures of standard SBM in order to achieve tight recovery
guarantees~\cite{feige2001semirandom,moitra2016robust}. In contrast,
our results show that SDP relaxations enjoy a robustness property
that is possessed by few other algorithms. Importantly, this generalization
can be achieved with little extra effort from our main result (see
Theorem~\ref{thm:SBM_semirandom} and its proof).

\paragraph{Exact recovery.}

As another illustration of the strength of Theorem \ref{thm:informal},
we note that it implies sharp condition for SDP to recover $\LabelStar$
\emph{exactly}. In particular, when $\proxnum\snr>(1+\delta)\log n$
for any positive constant $\delta$, the bound (\ref{eq:informal_misrate})
ensures that $\misrate(\LabelSDP,\LabelStar)<\frac{1}{\num}$ and
hence $\misrate(\LabelSDP,\LabelStar)=0$. Moreover, the lower bound~(\ref{eq:informal_lower})
shows that exact recovery is information-theoretically impossible
when $\proxnum\snr<\log n$. In the literature, establishing such
tight exact recovery thresholds often involves specialized and sophisticated
arguments, and has been  the milestones in the remarkable recent development
on community detection~(see Section~\ref{sec:related} for a discussion
of related work). We recover these results, for all three models,
as a corollary of our main theorem by plugging in the corresponding
expressions of $\snr$ and $\proxnum$.  In fact, the non-asymptotic
version of Theorem~\ref{thm:informal} guarantees exact recovery
via SDP with an explicit second-order term $\delta=O\big(1/\sqrt{\log\num}\big)$,
which is a refinement of existing results.

\subsection{Primal-dual analysis\label{sec:highlight}}

Key to the establishment of our results is a novel analysis that exploits
both primal and dual characterizations of the SDP. To set the context,
we note that the sub-optimal bound (\ref{eq:suboptimal_bound}) in
\cite{fei2019exponential,giraud2018partial} is established by utilizing
the primal optimality of the SDP solution $\Yhat$. Their arguments,
however, are too crude to provide a tight estimate of the multiplicative
constant $C$ in the exponent. On the other hand, work on exact recovery
for SDP typically makes use of a dual analysis \cite{hajek2016achieving,bandeira2015convex};
in particular, the optimality of $\Ystar$ is certified by showing
the existence of a corresponding dual optimal solution, often explicitly
in the form of a diagonal matrix $\D$ with $D_{ii}=\labelstar_{i}\sum_{j}\adj_{ij}\labelstar_{j}$.
As this ``dual certificate'' $\D$ is tied to (and constructed using)
$\Ystar$, such a certification approach would only succeed when the
SDP indeed admits $\Ystar$ as an optimal solution.

Here we are concerned with the setting where the optimal solution
$\Yhat$ is different from $\Ystar$, and our goal is to bound their
difference. As it is a priori unknown what $\Yhat$ should look like,
we do not know which matrix to certify or how to construct its associated
dual solution, rendering the above dual certification argument inapplicable.
Instead, we make use of the fact that $\Yhat$ is feasible to the
SDP and has a better primal objective value than $\Ystar$, that is,
$\Yhat$ lies in the sublevel set defined by $\Ystar$ and the constraints
of the SDP. We then characterize the diameter of this sublevel set
by using, perhaps surprisingly, the dual certificate $\D$ of $\Ystar$.
Our analysis is thus fundamentally different from the dual certification
analysis in existing work, which only applies when the sublevel set
consists of a single element $\Ystar$. At the same time, we make
use of $\D$ in a crucial way to achieve an exponential improvement
over previous primal analysis.

Note that our analysis, and hence our error bounds as well, actually
apply to \emph{every element }of this sublevel set, not just the optimal
solution $\Yhat$. As can be seen in our proof, this flexibility plays
an important role in establishing the aforementioned robustness results
under semirandom and heterogeneous SBMs. On the other hand, however,
with this level of generality we probably should not expect the second-order
$o(1)$ term in our bounds to be optimal.

Finally, we emphasize that our results for Z2, CBM and SBM are proved
under a unified framework. The main proof steps are deterministic
and hold for the three models at once; only certain probabilistic
arguments are model-specific.  In  Section~\ref{sec:proof_sketch}
we outline this proof framework, and provide intuitions on the majority
voting mechanism that drives the error rate $e^{-\proxnum\snr}$.
We believe that this unified framework may be broadly useful in studying
SDP relaxations for other discrete problems under average-case/probabilistic
settings.

\subsection{Paper organization}

In Section \ref{sec:related}, we review related work on Z2, CBM and
SBM. In Section \ref{sec:setup}, we formally introduce the models
and the SDP relaxation approach. In Section~\ref{sec:main}, we present
our main results, with a discussion on their consequences and comparison
with existing work. We outline the main steps of the proofs and discuss
the intuitions in Section~\ref{sec:proof_sketch}, with the complete
proofs deferred to the appendix. The paper is concluded in Section~\ref{sec:discuss}
with a discussion on future directions. 

\section{Related work\label{sec:related}}

There is a large array of recent results on community detection and
graph clustering, in particular, under Z2, CBM and SBM. The readers
are referred to the surveys \cite{abbe2016recent,moore2017landscape,li2018survey}
for comprehensive reviews. Without trying to enumerate this body of
work, here we restrict attention to those that study sharp performance
bounds, with a particular focus on work on the SDP relaxation approach.
A more detailed, quantitative comparison with our results is provided
in Section~\ref{sec:main} after our main theorems.

To begin, we note that existing work has considered several recovery
criteria for an estimator $\LabelHat$ of $\LabelStar$: \emph{weak
recovery} means $\LabelHat$ is better than random guess, that is,
$\misrate(\LabelHat,\LabelStar)<\frac{1}{2}$; \emph{partial recovery}
means $\misrate(\LabelHat,\LabelStar)\le\delta$ for a given $\delta\in(0,\frac{1}{2})$;
\emph{exact recovery} means $\misrate(\LabelHat,\LabelStar)=0$ \cite{abbe2016recent}.

\subsection{$\mathbb{Z}_{2}$ Synchronization and Censored Block Model\label{sec:related_Z2_CBM}}

The Z2 model, being a simplified version of the angular/phase synchronization
problem, is studied in \cite{bandeira2017tightness}, which argues
that exact recovery is possible if and only if $\snr>\frac{\log\num}{\num}$.
The work in~\cite{bandeira2015convex} and \cite{abbe2017entrywise}
shows this optimal exact recovery threshold is achieved by SDP and
a spectral algorithm, respectively.    The work in \cite{perry2016optimality,lelarge2017fundamental}
considers low-rank matrix estimation under a spiked Wigner model\textemdash of
which Z2 is a special case\textemdash and identifies the weak recovery
threshold.   

CBM is considered in \cite{abbe2014censored,hajek2015censor}, which
identifies sufficient and necessary conditions for exact recovery.
They also show that SDP achieves a sub-optimal exact recovery threshold,
which is further improved to be optimal in \cite{bandeira2015convex,hajek2015censor}.
 CBM is a special case of the so-called Labelled SBM, whose weak
recovery threshold is studied in \cite{heimlicher2012label,lelarge2015reconstruction}.
Achieving tight partial/weak recovery guarantees in CBM is challenging
due to the sparsity of the observations. A sub-optimal partial recovery
error bound can be achieved by a spectral algorithm with trimming
\cite{ChinRaoVu15}. The work of \cite{saade2015spectral} studies
a sophisticated spectral algorithm based on the non-backtracking operator
or Bethe Hessian, and shows that it achieves the optimal weak recovery
threshold. 

For both Z2 and CBM, we establish for the first time that SDP has
the optimal error rates for partial recovery. Our results also imply,
as an immediate corollary, that SDP achieves the optimal exact recovery
threshold as well as a sub-optimal weak recovery threshold. 

\subsection{Stochastic Block Model\label{sec:related_SBM}}

SBM is arguably the most studied out of these three models. Most related
to us is a line of work that characterizes minimax optimal error rates
for partial recovery. For the binary symmetric SBM, the work \cite{zhang2016minimax}
establishes the aforementioned minimax lower bound (\ref{eq:informal_lower}).
They also provide an exponential-time algorithm that achieves a matching
upper bound (up to an $o(1)$ factor in the exponent). Much research
effort focuses on developing computationally feasible algorithms,
and identifying the minimax rates in more general settings \cite{gao2017achieving,gao2018DCBM,xu2017weighted,yun2014accurate,yun2016optimal,zhang2017theoretical,zhou2018nonasymptotic}.
The monograph \cite{gao2018minimax} provides a review on recent work
on this front. We note that this line of work does not consider the
SDP relaxation approach nor deliver robustness guarantees as we do.
Nevertheless, we will compare our results with theirs after stating
our main theorems.

For exact recovery under binary symmetric SBM with $\inprob,\outprob\asymp\frac{\log\num}{\num}$,
the work in \cite{abbe2016exact,mossel2016bisection} establishes
the sufficient and necessary condition $(\!\sqrt{\inprob}-\sqrt{\outprob})^{2}>\frac{2\log\num}{\num}$.
Follow-up work develops efficient algorithms for exact recovery and
considers extensions to more general SBMs; see, e.g., \cite{abbe2015general,abbe2015recovering,abbe2017entrywise,jog2015information,perry2015semidefinite}.
As mentioned, our results imply sharp bounds for exact recovery.

Weak recovery under the binary symmetric SBM is most relevant in the
sparse regime $\inprob,\outprob\asymp\frac{1}{\num}$. Work of \cite{lelarge2015reconstruction,massoulie2014ramanujan,mossel2015reconstruction}
establishes that the necessary and sufficient condition of weak recovery
is $\frac{\num(\inprob-\outprob)^{2}}{\inprob+\outprob}>2$. Subsequent
work proves similar phase transitions and shows that various algorithms
achieve weak recovery above the optimal threshold for the SBM with
$\numclust\ge2$ and possibly unbalanced clusters; see, e.g., \cite{abbe2015multiple,abbe2018graph,banerjee2018contiguity,bordenave2018nonbacktracking,caltagirone2018recovering,coja2018information,mossel2018proof,stephan2018robustness}.
As discussed later, our results also imply weak recovery guarantees
with a sub-optimal constant.

\subsection{Optimality and robustness of SDP}

For SBM, SDP has been proven to succeed in exact and weak recovery
above the corresponding optimal thresholds (sometimes under additional
assumptions). In particular, see \cite{agarwal2017multisection,bandeira2015convex,hajek2016achieving}
for exact recovery, and \cite{MontanariSen16} for weak recovery.
Prior to our work, SDP was not known to achieve the optimal error
rate between the exact and weak recovery regimes. Sub-optimal polynomial
rates are first proved in \cite{Guedon2015}, later improved to exponential
in \cite{fei2019exponential}, and further generalized in \cite{fei2018hidden,giraud2018partial}.

Robustness has been recognized as a distinct feature of the SDP approach
as compared to other more specialized algorithms for SBMs. Work in
this direction has established robustness of SDP against random erasures
\cite{hajek2015censor,hajek2016extension}, atypical node degrees
\cite{Guedon2015} and adversarial corruptions \cite{hajek2016achieving,MontanariSen16,Cai2014robust,makarychev2016learning}.
The work in \cite{moitra2016robust} investigates the relationship
between statistical optimality and robustness under monotone semirandom
models; we revisit this result in more details later.

\section{Problem Set-up\label{sec:setup}}

In this section, we formally define the models and introduce the SDP
relaxation approach.

\subsection{Notations}

Vectors and matrices are denoted by bold letters. For a vector $\u$,
$u_{i}$ and $u(i)$ both denote its $i$-th entry.   For a matrix
$\M$, we let $M_{ij}$ denote its $(i,j)$-th entry, $\Tr(\M)$ its
trace, and $\norm[\M]1\coloneqq\sum_{i,j}\left|M_{ij}\right|$ its
entry-wise $\ell_{1}$ norm.  We write $\M\succeq0$ if $\M$ is
symmetric positive semidefinite. The trace inner product between two
matrices is $\left\langle \M,\G\right\rangle \coloneqq\Tr(\M^{\top}\G)=\sum_{i,j}M_{ij}G_{ij}$.
Denote by $\I$ and $\OneMat$ the $\num\times\num$ identity matrix
and all-one matrix, respectively, and denote by $\onevec$ the all-one
column vector of length $\num$. 

$\Bern(\mu)$ denotes the Bernoulli distribution with mean $\mu\in[0,1]$.
For a positive integer~$i$, let $[i]\coloneqq\{1,2,\ldots,i\}$.
For a real number $x$, $\left\lceil x\right\rceil $ denotes its
ceiling and $\left\lfloor x\right\rfloor $ denotes its floor. $\indic\{\cdot\}$
is the indicator function. For two non-negative sequences $\{a_{\num}\}$
and $\{b_{\num}\}$, we write $a_{\num}=O(b_{\num})$, $b_{\num}=\Omega(a_{\num})$
or $a_{\num}\lesssim b_{\num}$ if there exists a universal constant
$C>0$ such that $a_{\num}\le Cb_{\num}$ for all $\num$. We write
$a_{\num}\asymp b_{\num}$ if both $a_{n}=O(b_{\num})$ and $a_{\num}=\Omega(b_{\num})$
hold. Asymptotic statements are with respect to the regime $n\to\infty$,
in which case we write $a_{\num}=o(b_{\num})$ and $b_{\num}=\omega(a_{\num})$
if $\lim_{n\to\infty}a_{\num}/b_{\num}=0$. 

\subsection{Models \label{sec:setup_model}}

In this section, we formally describe four models for generating the
observed matrix $\A$ from the unknown ground-truth label vector $\LabelStar\in\left\{ \pm1\right\} ^{\num}$.

In Z2~\cite{bandeira2015convex}, each $\adj_{ij}$ is generated
by adding Gaussian noise to $\labelstar_{i}\labelstar_{j}$. Therefore,
the matrix $\Adj$ contains noisy observations of the true relative
signs between each pair of nodes.

\begin{mdl}[$\mathbb{Z}_2$ Synchronization]\label{mdl:Z2} The observed
matrix $\Adj\in\real^{\num\times\num}$ is symmetric with its entries
$\{\adj_{ij},i\le j\}$ generated independently by 
\[
\adj_{ij}\sim N(\labelstar_{i}\labelstar_{j},\std^{2}),
\]
where $\std>0$ is allowed to scale with $\num$.\footnote{In this and the next three models, we assume that the diagonal entries
of $\Adj$ are random, which is inconsequential: these entries are
independent of the ground-truth $\LabelStar$, and they have no effect
on the solutions of the SDP relaxations~(\ref{eq:CBM_Z2_SDP}) or~(\ref{eq:SBM_SDP})
due to the diagonal constraints therein.} \end{mdl}

In CBM, each $\adj_{ij}$ is generated by flipping $\labelstar_{i}\labelstar_{j}$
with probability $\flipprob$ and then erasing it with probability
$1-\obsprob$. One may interpret $\Adj$ as the edge-censored version
of a noisy signed network~\cite{abbe2014censored}. 

\begin{mdl}[Censored Block Model]\label{mdl:CBM} The observed matrix
$\Adj\in\{0,\pm1\}^{\num\times\num}$ is symmetric with its entries
$\{\adj_{ij},i\le j\}$ generated independently by 
\[
\adj_{ij}=\begin{cases}
\labelstar_{i}\labelstar_{j} & \text{with probability (w.p.) }\obsprob(1-\flipprob),\\
-\labelstar_{i}\labelstar_{j} & \text{w.p. }\obsprob\flipprob,\\
0 & \text{w.p. }1-\obsprob,
\end{cases}
\]
where $\obsprob\in(0,1]$ is allowed to scale with $\num$, and $\flipprob\in(0,\frac{1}{2})$
is a constant. \end{mdl}

In SBM, each $\adj_{ij}$ is a Bernoulli random variable, whose mean
is higher if $\labelstar_{i}\labelstar_{j}=1$. Therefore, $\Adj$
is the adjacency matrix of a random graph in which nodes in the same
cluster are more likely to be connected than those in different clusters~\cite{holland83}.

\begin{mdl}[Binary symmetric SBM]\label{mdl:SBM} Suppose that the
ground-truth $\LabelStar\in\left\{ \pm1\right\} ^{\num}$ satisfies
$\left\langle \LabelStar,\one\right\rangle =0$. The observed matrix
$\Adj\in\{0,1\}^{\num\times\num}$ is symmetric with its entries $\{\adj_{ij},i\le j\}$
generated independently by 
\[
\adj_{ij}\sim\begin{cases}
\Bern(\inprob) & \text{if }\labelstar_{i}\labelstar_{j}=1,\\
\Bern(\outprob) & \text{if }\labelstar_{i}\labelstar_{j}=-1,
\end{cases}
\]
where $0<\outprob<\inprob<1$ are allowed to scale with $\num$. \end{mdl}

In both Model \ref{mdl:Z2} (Z2) and Model \ref{mdl:CBM} (CBM), there
can be any number of $\pm1$'s in the ground-truth label vector $\LabelStar\in\{\pm1\}^{\num}$.
 In Model~\ref{mdl:SBM} (binary symmetric SBM), the cluster labels
$\LabelStar$ are assumed to contain the same number of $1$'s and
$-1$'s, so the two clusters have equal size. Despite their simple
forms, the above models have been of central importance in studying
fundamental limits of clustering problems \cite{abbe2016exact,hajek2016achieving,lelarge2015reconstruction,massoulie2014ramanujan,MontanariSen16,mossel2015reconstruction,mossel2016bisection,abbe2014censored,bandeira2015convex}. 

For the purpose of studying the robustness properties of SDP relaxation,
we consider a semirandom generalization of the binary symmetric SBM.
In this model, a so-called monotone adversary, upon observing the
random adjacency matrix $\Adj$ generated from SBM and the ground-truth
clustering $\LabelStar$, modifies $\Adj$ by \emph{arbitrarily }adding
edges between nodes of the same cluster and deleting edges between
nodes of different clusters. 

\begin{mdl}[Semirandom SBM]\label{mdl:semirandom} A monotone adversary
observes $\Adj$ and $\LabelStar$ from Model \ref{mdl:SBM}, picks
an arbitrary set of pairs of nodes $\pairset\subset\{(i,j)\in[\num]\times[\num]:i<j\}$,
and outputs a symmetric matrix $\Adj^{\text{SR}}\in\{0,1\}^{\num\times\num}$
such that for each $i<j$, 
\[
\adj_{ij}^{\text{SR}}=\begin{cases}
1 & \text{if }(i,j)\in\pairset,\ \labelstar_{i}\labelstar_{j}=1,\\
0 & \text{if }(i,j)\in\pairset,\ \labelstar_{i}\labelstar_{j}=-1,\\
\adj_{ij}, & \text{if }(i,j)\notin\pairset.
\end{cases}
\]
Note that the set $\pairset$ is allowed to depend on the realization
of $\Adj$. \end{mdl}

Semirandom models have a long history with many variants \cite{blum1995coloring}.
Model~\ref{mdl:semirandom} above has been considered in \cite{feige2001semirandom,moitra2016robust}
for SBM. While seemingly revealing more information about the underlying
cluster structure, the semirandom model in fact destroys many local
structures of the basic SBM, thus frustrating many algorithms that
over-exploit such structures. In contrast, SDP is robust against the
monotone adversary under Model \ref{mdl:semirandom}, as we shall
see in Section~\ref{sec:robustness} below.
\begin{rem}
\label{rem:semirandom_Z2_CBM}One may define semirandom versions of
Z2 and CBM in an analogous fashion as above; that is, the adversary
may choose a set $\pairset$ and positive numbers $\{c_{ij},i<j\}$,
and then change $\adj_{ij}$ and $\adj_{ji}$ to $\adj_{ij}+c_{ij}\labelstar_{i}\labelstar_{j}$
for each $(i,j)\in\pairset$. It can be shown that SDP achieves the
same performance guarantees in these semirandom settings of Z2 and
CBM as in the original models. For conciseness we omit such details
and only focus on the semirandom extension of SBM. \\
\end{rem}
For each model discussed above, we define a measure of the signal-to-noise
ratio (SNR): 
\begin{equation}
\snr\coloneqq\begin{cases}
(2\std^{2})^{-1}, & \text{for Model \ref{mdl:Z2}},\\
\left(\!\sqrt{\obsprob(1-\flipprob)}-\sqrt{\obsprob\flipprob}\right)^{2}, & \text{for Model \ref{mdl:CBM}},\\
-2\log\left[\!\sqrt{\inprob\outprob}+\sqrt{(1-\inprob)(1-\outprob)}\right], & \text{for Models \ref{mdl:SBM} and \ref{mdl:semirandom}}.
\end{cases}\label{eq:snr}
\end{equation}
In each case, $\snr$ is a form of Renyi divergence of order $\frac{1}{2}$
\cite{gil2013renyi} between the distributions of $\adj_{ij}$ and
$\adj_{ij'}$ with $\labelstar_{ij}=-\labelstar_{ij'}=1$. In particular,
for Z2, $\snr$ is half of the Renyi divergence (or equivalently,
the Kullback\textendash Leibler divergence) between $N(1,\std^{2})$
and $N(-1,\std^{2})$. For CBM, we have $\snr\approx-\log(1-\snr)$
with the latter being half of the Renyi divergence between two random
variables $\var$ and $-\var$, where $\var$ has probability mass
function $\obsprob(1-\flipprob)\cdot\delta_{1}+\obsprob\flipprob\cdot\delta_{-1}+(1-\alpha)\cdot\delta_{0}$
and $\delta_{a}$ denotes the Dirac delta function centered at $a$.\footnote{In fact, in this case $\snr$ is the squared Hellinger distance between
$\var$ and $-\var$.}  In SBM, $\snr$ is the Renyi divergence between $\Bern(\inprob)$
and $\Bern(\outprob)$. These divergences, and their first-order approximations
(discussed in Section~\ref{sec:preliminary}), are commonly used
as SNR measures in previous work on these models (e.g.,~\cite{abbe2014censored,bandeira2015convex,zhang2016minimax}).

Finally, we define the following distance measure between two vectors
of cluster labels  $\Label,\Label'\in\{\pm1\}^{\num}$: 
\[
\misrate(\Label,\Label')\coloneqq\min_{g\in\{\pm1\}}\frac{1}{\num}\sum_{i\in[\num]}\indic\{g\labels_{i}\ne\labels_{i}'\}.
\]
In words, $\misrate(\Label,\Label')$ is the fraction of nodes that
are assigned a different label under $\Label$ and $\Label'$, modulo
a global flipping of signs. With $\LabelStar$ being the true labels,
$\misrate(\LabelHat,\LabelStar)$ measures the relative error of the
estimator $\LabelHat$.

\subsection{SDP relaxation}

The SDP formulations we consider can be derived as the convex relaxation
of the MLE of $\LabelStar$. Under Models~\ref{mdl:Z2} or~\ref{mdl:CBM},
the MLE  $\LabelMLE$ is given by the solution of the discrete and
non-convex optimization problem
\begin{equation}
\begin{aligned}\max_{\Label\in\{\pm1\}^{\num}}\; & \left\langle \Adj,\Label\Label\t\right\rangle .\end{aligned}
\label{eq:MLE}
\end{equation}
The MLE under Model~\ref{mdl:SBM} includes the extra constraint
$\left\langle \Label,\one\right\rangle =0$ due to the balanced-cluster
assumption. Derivation of the MLE in this form is now standard; see
for example \cite{bandeira2017tightness,bandeira2015convex} for Z2,
\cite{hajek2016extension} for CBM, and \cite{li2018survey} for SBM.
 Now define the lifted variable $\Y=\Label\Label\t$, and observe
that $\Y$ satisfies $\Y\succeq0$, $Y_{ii}=(\labels_{i})^{2}=1$
for $i\in[\num]$. Dropping the constraints that $\Y$ has rank one
and binary entries, we obtain the following SDP relaxation of the
MLE (\ref{eq:MLE}) for Models~\ref{mdl:Z2} or~\ref{mdl:CBM}:
\begin{equation}
\begin{aligned}\Yhat=\argmax_{\mathbf{Y}\in\mathbb{R}^{\num\times\num}}\; & \left\langle \Adj,\Y\right\rangle \\
\mbox{s.t.}\; & \Y\succeq0,\\
 & Y_{ii}=1,\;\;\forall i\in[\num].
\end{aligned}
\label{eq:CBM_Z2_SDP}
\end{equation}
For Model~\ref{mdl:SBM}, using the same reasoning and in addition
replacing the $\left\langle \Label,\one\right\rangle =0$ constraint
by $\left\langle \Y,\OneMat\right\rangle =\left\langle \Label\Label\t,\one\one\t\right\rangle =\left\langle \Label,\one\right\rangle ^{2}=0$,
we arrive at the relaxation:
\begin{equation}
\begin{aligned}\Yhat=\argmax_{\mathbf{Y}\in\mathbb{R}^{\num\times\num}}\; & \left\langle \Adj,\Y\right\rangle \\
\mbox{s.t.}\; & \Y\succeq0,\\
 & Y_{ii}=1,\;\;\forall i\in[\num],\\
 & \left\langle \Y,\OneMat\right\rangle =0.
\end{aligned}
\label{eq:SBM_SDP}
\end{equation}
We also use this SDP for the semirandom Model \ref{mdl:semirandom}.

The optimization problems~(\ref{eq:CBM_Z2_SDP}) and~(\ref{eq:SBM_SDP})
are standard SDPs solvable in polynomial time. We remark that neither
SDP requires knowing the parameters of the data generating processes
(that is, $\std^{2}$, $\obsprob$, $\flipprob$, $\inprob$ and $\outprob$
in Models~\ref{mdl:Z2}\textendash \ref{mdl:SBM}).\footnote{The SDP~(\ref{eq:SBM_SDP}) for SBM does require the knowledge of
two equal-size clusters.} The SDP~(\ref{eq:CBM_Z2_SDP}) was considered in~\cite{bandeira2017tightness,bandeira2015convex}
and~\cite{hajek2016extension} for studying the exact recovery threshold
in Z2 and CBM, respectively, and the SDP~(\ref{eq:SBM_SDP}) was
considered in~\cite{hajek2016achieving} for exact recovery under
the binary symmetric SBM. These formulations can be further traced
back to the work of~\cite{feige2001semirandom} on SDP relaxation
for MIN BISECTION.

We consider the SDP solution $\Yhat$ as an estimate of the ground-truth
matrix $\Ystar:=\LabelStar(\LabelStar)^{\top}$, and seek to characterize
the accuracy of $\Yhat$ in terms of the $\ell_{1}$ error $\norm[\Yhat-\Ystar]1$.
Note that $\Yhat$ is not necessarily a rank-one matrix of the form
$\Yhat=\Label\Label^{\top}$. To extract from $\Yhat$ a vector of
binary estimates of cluster labels, we take the signs of the entries
of the top eigenvector of $\Yhat$ (where the sign of $0$ is $1$,
an arbitrary choice). Letting $\LabelSDP\in\{\pm1\}^{\num}$ be the
vector obtained in this way, we study the error of $\LabelSDP$ as
an estimate of the ground-truth label vector $\LabelStar$, as measured
by $\misrate(\LabelSDP,\LabelStar)$. 

\section{Main results\label{sec:main}}

We present our main results in this section. Henceforth, let $\proxnum:=\num$
in Models~\ref{mdl:Z2} (Z2) and~\ref{mdl:CBM} (CBM), and $\proxnum:=\frac{\num}{2}$
in Models~\ref{mdl:SBM} and~\ref{mdl:semirandom} (SBM and its
semirandom version). To see why this definition of $\proxnum$ is
natural, we note that in Z2 and CBM, the cluster sizes (i.e., numbers
of $1$'s and $-1$'s in $\LabelStar$) do not affect the hardness
of the problem as the distribution is symmetric, and hence $\proxnum$
is simply the number of nodes; in binary symmetric SBMs, recovery
is most difficult when the clusters have equal size\footnote{Otherwise one could recover the large cluster first.}\textemdash which
is the setting we consider\textemdash and accordingly $\proxnum$
is the cluster size. 

\subsection{Minimax lower bounds\label{sec:main_lower}}

Let $\size_{1}(\Label)$ denote the number of 1's in $\Label$. To
state the lower bounds, we consider the following parameter space:
\begin{equation}
\paramset(\num)\coloneqq\begin{cases}
\;\left\{ \pm1\right\} ^{\num}, & \text{for Models \ref{mdl:Z2} and \ref{mdl:CBM}},\\
\;\left\{ \Label\in\left\{ \pm1\right\} ^{\num}:\size_{1}(\Label)\in\left[\frac{\num}{2\beta},\frac{\num\beta}{2}\right]\right\} , & \text{for Model \ref{mdl:SBM}},
\end{cases}\label{eq:param_space}
\end{equation}
where $\beta$ is any number larger than $1+C/\num$ with $C>0$ being
a large enough numerical constant. For Z2 and CBM, $\paramset(\num)$
is the set of all possible cluster label vectors. For SBM, $\paramset(\num)$
consists of label vectors with (roughly) equal-sized clusters; here
we allow for a slight fluctuation in the cluster sizes in SBM following
\cite{zhang2016minimax}.\footnote{This assumption is not essential but makes the proof therein somewhat
simpler.} 

The following theorem gives the minimax lower bound for each model.
\begin{thm}[Lower bound]
\label{thm:minimax_lower}For any constant $c_{0}\in(0,1)$, the
following holds for Model~\ref{mdl:Z2}, Model~\ref{mdl:CBM} with
$\snr=o(1)$, and Model~\ref{mdl:SBM} with $0<\outprob<\inprob<1-c_{0}$.
If $\num\snr\to\infty$ as $\num\to\infty$, then we have
\[
\inf_{\LabelHat}\sup_{\Label\in\paramset(\num)}\E_{\Label}\misrate(\LabelHat,\Label)\ge\exp\left[-\big(1+o(1)\big)\proxnum\snr\right],
\]
where $\E_{\Label}$ denotes expectation under the distribution of
$\Adj$ with $\Label$ being the ground truth, and the infimum is
taken over all estimators of the ground truth (i.e., measurable functions
of $\Adj$).
\end{thm}
For Models \ref{mdl:Z2} and \ref{mdl:CBM}, the proof is given in
Section \ref{sec:proof_minimax_lower_bound}. For Model \ref{mdl:SBM},
the above result is part of \cite[Theorem 2.1]{zhang2016minimax}.

\subsection{Upper bounds on the SDP errors \label{sec:main_SDP}}

We next provide our main results on the error rate of the SDP relaxations
(\ref{eq:CBM_Z2_SDP}) and (\ref{eq:SBM_SDP}). Define the following
sublevel set (or superlevel set to be precise):
\begin{equation}
\betterset(\Adj)\coloneqq\Big\{\Y\in\real^{\num\times\num}:\left\langle \Adj,\Y\right\rangle \ge\left\langle \Adj,\Ystar\right\rangle ,\Y\text{ is feasible to the SDP}\Big\},\label{eq:def_larger_objval_set}
\end{equation}
where feasibility is with respect to the program~(\ref{eq:CBM_Z2_SDP})
for Model \ref{mdl:Z2} or \ref{mdl:CBM}, and to the program~(\ref{eq:SBM_SDP})
for Model \ref{mdl:SBM} or \ref{mdl:semirandom}. In words, $\betterset(\Adj)$
is the set of feasible SDP solutions that attain an objective value
no worse than the ground-truth $\Ystar$. As mentioned, our upper
bounds in fact hold for any solution in $\betterset(\Adj)$.  With
a slight abuse of notation, in the sequel we use $\Yhat$ to denote
an arbitrary matrix in $\betterset(\Adj)$; accordingly, we let $\LabelSDP$
denote the corresponding vector of labels extracted from this $\Yhat$. 

Our main theorem is a non-asymptotic bound on the error rates of the
SDP relaxations.
\begin{thm}[Upper bound]
\emph{\label{thm:SDP_error} }For any constants $c_{0},c_{1}\in(0,1)$,
there exist constants $\consts,\conste,\conste'>0$ such that  the
following holds for Model~\ref{mdl:Z2}, Model~\ref{mdl:CBM}, and
Model~\ref{mdl:SBM} with $0<c_{0}\inprob\le\outprob<\inprob\le1-c_{1}$.
If $\num\snr\ge\consts$, then with probability at least $1-10\exp\left(-\sqrt{\log\num}\right)$,
\[
\begin{aligned}\frac{1}{\num}\norm[\Yhat-\Ystar]1 & \leq\left\lfloor \num\exp\left[-\bigg(1-\conste\sqrt{\frac{1}{\num\snr}}\bigg)\proxnum\snr\right]\right\rfloor ,\\
\misrate(\LabelSDP,\LabelStar) & \leq\exp\left[-\bigg(1-\conste'\sqrt{\frac{1}{\num\snr}}\bigg)\proxnum\snr\right],
\end{aligned}
\qquad\forall\Yhat\in\betterset(\A).
\]
\end{thm}
The proof is given in Appendices~\ref{sec:proof_SDP_error_rate}
and~\ref{sec:proof_cluster_error_rate}. Note the floor operation
in the first inequality above; consequently, we have $\norm[\Yhat-\Ystar]1=0$
whenever the exponent is strictly less than $-\log n$. We later explore
the implication of this fact for exact recovery. 
\begin{rem}
The assumption $c_{0}\inprob\le\outprob$ for Model \ref{mdl:SBM}
is common in the literature on minimax rates \cite{gao2017achieving,gao2018DCBM,zhang2017theoretical,zhou2018nonasymptotic}.
It stipulates that $\inprob$ and $\outprob$ are on the same order
(but their difference can be vanishingly small). This is the regime
where the clustering problem is hard, and it is the regime we focus
on. The assumption arises from a technical step in our proof, and
it is currently not clear to us whether this assumption is necessary.
We would like to point out that in Section 5.1 of \cite{gao2017achieving},
a weaker minimax upper bound is obtained with this assumption dropped.
\end{rem}
Letting $\num\to\infty$ in Theorem \ref{thm:SDP_error}, we immediately
obtain the following asymptotic result.
\begin{cor}[Upper bound, asymptotic]
\label{cor:SDP_error_asymp}For any constants $c_{0},c_{1}\in(0,1)$,
the following holds for Model~\ref{mdl:Z2}, Model~\ref{mdl:CBM},
and Model~\ref{mdl:SBM} with $0<c_{0}\inprob\le\outprob<\inprob\le1-c_{1}$.
If $\num\snr\to\infty$, then with probability $1-o(1)$,
\[
\begin{aligned}\frac{1}{\num}\norm[\Yhat-\Ystar]1 & \leq\Big\lfloor\num\exp\left[-\big(1-o(1)\big)\proxnum\snr\right]\Big\rfloor,\\
\misrate(\LabelSDP,\LabelStar) & \leq\exp\left[-\big(1-o(1)\big)\proxnum\snr\right],
\end{aligned}
\qquad\forall\Yhat\in\betterset(\A).
\]
\end{cor}
Comparing the upper bound in Corollary~\ref{cor:SDP_error_asymp}
with the minimax lower bound in Theorem~\ref{thm:minimax_lower},
we see that the SDP achieves the optimal error rate, up to a second-order
$o(1)$ term in the exponent.\footnote{We note that Theorem~\ref{thm:minimax_lower} bounds the error in
expectation and holds for a parameter space containing $\LabelStar$
with slightly unequal-sized clusters. The results in Theorem~\ref{thm:SDP_error}
and Corollary~\ref{cor:SDP_error_asymp} are high-probability bounds.
Extending these upper bounds to the setting of slightly unequal-sized
clusters is possible, albeit tedious; we leave this to future work.} Moreover, Theorem \ref{thm:SDP_error} provides an explicit, non-asymptotic
upper bound for the $o(1)$ term in the exponent. This bound, taking
the form of $O(1/\sqrt{\num\snr})$, yields second-order characterization
of various recovery thresholds and is strong enough to provide non-trivial
guarantees in the sparse graph regime\textemdash these points are
discussed in Section \ref{sec:consequence} to follow. We do not expect
this $O(1/\sqrt{\num\snr})$ bound to be information-theoretic optimal,
for reasons discussed in Section~\ref{sec:highlight}.

As a passing note, the above error upper bounds also apply to the
MLE solution $\LabelMLE$, since the optimality of $\LabelMLE$ to
the program (\ref{eq:MLE}) implies that $(\LabelMLE)(\LabelMLE)\t\in\betterset(\Adj)$.
 In fact, our proof of the upper bounds involves showing that the
SDP solutions closely approximate the MLE; we elaborate on this point
in Section~\ref{sec:proof_sketch}.

\subsection{Robustness under Semirandom Models\label{sec:robustness}}

Our next result shows that the error rate of the SDP is unaffected
by passing to the semirandom model. Recall the definition in Equation~(\ref{eq:def_larger_objval_set}),
so  $\betterset(\AdjSR)$ is the sublevel set of the SDP (\ref{eq:SBM_SDP})
with $\AdjSR$ as the input.
\begin{thm}[Semirandom SBM]
\emph{\label{thm:SBM_semirandom}} Suppose that $\AdjSR$ is generated
according to Model \ref{mdl:semirandom}. The conclusions of Theorem
\ref{thm:SDP_error} and Corollary \ref{cor:SDP_error_asymp} continue
to hold for the program~(\ref{eq:SBM_SDP}) with $\betterset(\Adj)$
replaced by $\betterset(\AdjSR)$.
\end{thm}

\begin{proof}
This theorem admits a short proof, thanks to the validity of Theorem
\ref{thm:SDP_error} for any $\Yhat\in\betterset(\Adj)$. Recall that
$\AdjSR$ is obtained by monotonically modifying the matrix $\Adj$
generated from Model~\ref{mdl:SBM} (binary symmetric SBM). Let $\Yhat$
be an arbitrary element of $\betterset(\AdjSR)$. By definition of
$\AdjSR$ and the feasibility of $\Yhat$, we have for all $i,j\in[\num]$
\[
\begin{cases}
\adj_{ij}^{\text{SR}}\ge\adj_{ij},\:\yhat_{ij}-\ystar_{ij}\le0, & \text{if }\ystar_{ij}=1,\\
\adj_{ij}^{\text{SR}}\le\adj_{ij},\:\yhat_{ij}-\ystar_{ij}\ge0, & \text{if }\ystar_{ij}=-1.
\end{cases}
\]
The fact $\Yhat$ has objective value no worse than $\Ystar$ under
$\Adj^{\text{SR}}$, together with the above inequalities, implies
that $0\le\langle\Adj^{\text{SR}},\Yhat-\Ystar\rangle\le\langle\Adj,\Yhat-\Ystar\rangle$.
This further implies that $\Yhat\in\betterset(\Adj)$. Therefore,
invoking Theorem \ref{thm:SDP_error} gives the desired result.
\end{proof}
\begin{rem}
As mentioned in Remark~\ref{rem:semirandom_Z2_CBM}, one may define
semirandom versions of Models~\ref{mdl:Z2} (Z2) and~\ref{mdl:CBM}
(CBM). It is easy to see that the proof above applies to these models
without change. Therefore, the SDP approach is also robust under semirandom
Z2 and CBM.
\end{rem}
As an immediate consequence of Theorem~\ref{thm:SBM_semirandom},
we obtain error bounds for a generalization of the standard SBM (Model~\ref{mdl:SBM})
with heterogeneous edge probabilities, where 
\[
\adj_{ij}\overset{\text{i.i.d.}}{\sim}\begin{cases}
\Bern(\inprob_{ij}) & \text{if }\labelstar_{i}\labelstar_{j}=1,\\
\Bern(\outprob_{ij}) & \text{if }\labelstar_{i}\labelstar_{j}=-1,
\end{cases}
\]
with $0\le\outprob_{ij}\le\outprob<\inprob\le\inprob_{ij}\le1$. 
\begin{cor}[Heterogeneous SBM]
\emph{\label{cor:SBM_heterogeneous}} Under the above generalization
of Model~\ref{mdl:SBM}, the conclusions of Theorem \ref{thm:SDP_error}
and Corollary \ref{cor:SDP_error_asymp} continue to hold for the
program~(\ref{eq:SBM_SDP}).
\end{cor}

\begin{proof}
The corollary follows from the same coupling argument as in \cite[Appendix V]{fei2019exponential},
which shows that the Heterogeneous SBM can be reduced to the semirandom
Model \ref{mdl:semirandom}. 
\end{proof}
The results above show that SDP is insensitive to monotone modification
and heterogeneous probabilities. We emphasize that such robustness
is by no means automatic. With non-uniformity in the probabilities,
the likelihood function no longer has a known, rigid form, a property
heavily utilized in many algorithms. The monotone adversary can similarly
alter the graph structure by creating hotspots and short cycles. Even
worse, the adversary is allowed to make changes \emph{after} observing
the realized graph,\footnote{We therefore strengthen the robustness results in the previous work
\cite{fei2019exponential}, which does not allow such adaptivity.} thus producing unspecified dependency among all edges in the observed
data and leading to major obstacles for existing analysis of iterative
algorithms. 

We would like to mention that the work in \cite{moitra2016robust}
shows that the semirandom model makes weak recovery strictly harder.
While not contradicting their results technically, the fact that our
error bounds remain unaffected under this model does demand a closer
look. We note that our bounds are optimal only up to a second-order
term in the exponent and consequently do not attain the optimal weak
recovery limit. Also, our robustness results on error rates are tied
to a specific form of SDP analysis (using the sublevel set $\betterset(\Adj)$).
In comparison, for exact recovery SDP is robust \emph{by design }to
the semirandom model, as is well recognized in past work \cite{feige2001semirandom,chen2012sparseclustering,hajek2016achieving}.

\subsection{Consequences \label{sec:consequence}}

Theorem~\ref{thm:SDP_error} and Corollary~\ref{cor:SDP_error_asymp}
imply sharp sufficient conditions for several types of recovery:
\begin{itemize}
\item \textbf{Exact recovery:} Whenever $\proxnum\snr\ge(1+\delta)\log\num$
for any constant $\delta>0$, we have $\norm[\Yhat-\Ystar]1=0$ by
Corollary~\ref{cor:SDP_error_asymp} (note the floor operation therein)
and hence SDP achieves exact recovery by itself without any rounding/post-processing
steps.
\item \textbf{Second-order refinement:} Using the non-asymptotic Theorem~\ref{thm:SDP_error},
we can obtain the following refinement of the above result: exact
recovery provided that $\frac{\proxnum\snr}{\log\num}\ge1+\frac{C_{1}}{\sqrt{\log\num}}+\frac{C_{2}}{\log\num}$
for some constants $C_{1},C_{2}>0$. 
\item \textbf{Weak recovery:} When $\proxnum\snr\ge C$ for a sufficiently
large constant $C$, Theorem~\ref{thm:SDP_error} ensures that $\misrate(\LabelSDP,\LabelStar)<\frac{1}{2}$
and hence SDP achieves weak recovery.
\item \textbf{Sparse regime: }Theorem~\ref{thm:SDP_error}\textbf{ }ensures
that SDP achieves an arbitrarily small constant error when $\num\snr$
is a sufficiently large but finite constant. This corresponds to the
sparse graph regime with constant expected degrees, namely $\obsprob,\inprob,\outprob=\Theta(1/\num)$
in CBM and SBM. Many results on minimax rates require $\num\snr$,
and hence the degrees, to diverge (e.g.,$\ $\cite{gao2017achieving,zhang2017theoretical}). 
\end{itemize}
Moreover, these conditions remain sufficient under the semirandom
model. Below we specialize the above results to each of the three
models.

\subsubsection{$\mathbb{Z}_{2}$ Synchronization}

Recall that $\snr:=\frac{1}{2\std^{2}}$ and $\proxnum=\num$ under
Model~\ref{mdl:Z2}. Consequently, SDP achieves exact recovery if
$\std^{2}\le\frac{\num}{2\log\num+C\!\sqrt{\log\num}}.$ This is a
refinement of the best existing threshold $\std^{2}\le\frac{\num}{(2+\delta)\log\num}$
in \cite{abbe2017entrywise,bandeira2015convex}. 

We also have weak recovery by SDP if $\std^{2}\le\frac{\num}{C}$,
which matches, up to constants, the optimal threshold $\std^{2}<\num$
established in \cite{perry2016optimality,lelarge2017fundamental}.

\subsubsection{Censored Block Model}

Recall that $\snr:=\obsprob\big(\!\sqrt{1-\flipprob}-\sqrt{\flipprob}\big)^{2}$
and $\proxnum=\num$ under Model~\ref{mdl:CBM}. Consequently, SDP
achieves exact recovery if $\frac{\num}{\log\num}\snr\ge1+\frac{C}{\sqrt{\log\num}}.$
This result is a second-order improvement over the threshold $\frac{\num}{\log\num}\snr\ge1+\delta$
for SDP established in the work \cite{hajek2016extension}. The same
work also proves that exact recovery is impossible if $\frac{\num}{\log\num}\snr<1-\delta$. 

Noting that $\snr\asymp\obsprob(1-2\flipprob)^{2}$ (cf.~Fact~\ref{fact:CBM_renyi_equivalence}),
we also have weak recovery by SDP if $\num\obsprob(1-2\flipprob)^{2}\ge C$,
which matches, up to constants, the optimal threshold $\num\obsprob(1-2\flipprob)^{2}>1$
proved in \cite{heimlicher2012label,lelarge2015reconstruction,saade2015spectral}. 

\subsubsection{Stochastic Block Model}

Recall that $\snr:=-2\log\big[\sqrt{\inprob\outprob}+\sqrt{(1-\inprob)(1-\outprob)}\big]$
and $\proxnum=\frac{\num}{2}$ under Model~\ref{mdl:SBM}, and note
the equivalence $\snr=(1+o(1))(\!\sqrt{\inprob}-\sqrt{\outprob})^{2}$
valid for $0<\outprob\asymp\inprob=o(1)$. Consequently, SDP achieves
exact recovery if $\num(\!\sqrt{\inprob}-\sqrt{\outprob})^{2}\ge(2+\delta)\log\num$,
recovering the result established in \cite{hajek2016achieving,bandeira2015convex}. 

We also have the following refinement: exact recovery provided that
$\frac{\num\snr}{\log\num}\ge2+\frac{C_{1}}{\sqrt{\log\num}}+\frac{C_{2}}{\log\num}$.
This result is comparable to the sufficient condition $\frac{\num(\!\sqrt{\inprob}-\sqrt{\outprob})^{2}}{\log\num}\ge2+\frac{C}{\sqrt{\log\num}}+\omega\left(\frac{1}{\log\num}\right)$
for SDP established in \cite{hajek2016achieving}, whereas the necessary
and sufficient condition for the optimal estimator (MLE) is $\frac{\num(\!\sqrt{\inprob}-\sqrt{\outprob})^{2}}{\log\num}\ge2-\frac{\log\log\num}{\log\num}+\omega\left(\frac{1}{\log\num}\right)$
\cite{abbe2016exact,mossel2016bisection}.  

Finally, noting that $\snr\asymp(\inprob-\outprob)^{2}/\inprob$ (cf.$\ $Fact
\ref{fact:SBM_renyi_equivalence}), we have weak recovery by SDP if
$\num(\inprob-\outprob)^{2}/\inprob\ge C$. This condition matches,
up to constants, the so-called Kesten-Stigum (KS) threshold $\num(\inprob-\outprob)^{2}/(\inprob+\outprob)>2$,
which is optimal \cite{massoulie2014ramanujan,abbe2015multiple,mossel2013proof,mossel2015reconstruction}.

\subsection{Comparison with existing results}

In this section we focus on partial recovery under the binary symmetric
SBM (Model~\ref{mdl:SBM}), and compare with the existing work that
derives sharp error rate bounds achievable by polynomial-time algorithms.
To be clear, the algorithms considered in this line of work are very
different from ours. In particular, most existing results require
a good enough initial estimate of the true clusters. Obtaining such
an initial solution (typically using spectral clustering) is itself
a non-trivial task.

Using neighbor voting and variational inference algorithms, the work
in \cite{gao2017achieving,zhang2017theoretical}  obtains an error
bound of the same form as our Corollary~\ref{cor:SDP_error_asymp},
though they do not provide non-asymptotic results as in our Theorem~\ref{thm:SDP_error}.
 The work in \cite{yun2014accurate} considers a spectral algorithm
and proves the error bound $\misrate(\LabelHat,\LabelStar)\le\exp\left[-(1-\delta)(\sqrt{\inprob}-\sqrt{\outprob})^{2}\cdot\num/2\right]$
for any constant $\delta>0$ if $\num\inprob\to\infty$. Recalling
$\snr=(1+o(1))(\!\sqrt{\inprob}-\sqrt{\outprob})^{2}$, we find that
our Corollary~\ref{cor:SDP_error_asymp} is better as we allow the
$\delta$ term to vanish. The recent work in \cite{abbe2017entrywise}
uses a novel perturbation analysis to show that a very simple spectral
algorithm achieves the error bound in Corollary~\ref{cor:SDP_error_asymp}
under the assumption $\frac{\num}{\log\num}(\!\sqrt{\inprob}-\sqrt{\outprob})^{2}\ge\delta'$
for any constant $\delta'>0$; their assumption excludes the sparse
regime with $\inprob,\outprob=o\left(\frac{\log\num}{\num}\right)$
and is stronger than our assumption $\num\snr\to\infty$ in Corollary~\ref{cor:SDP_error_asymp}.
Compared to the above works, another strength of our results is that
we provide an explicit bound for the second-order term in the exponent;
we know of few error rate results (with the exception discussed below)
that offer this level of accuracy.

Concurrently to our work, the paper \cite{zhou2018nonasymptotic}
establishes a tight non-asymptotic error bound for an EM-type algorithm.
Translated to our notation, their bound takes the form 
\[
\misrate(\LabelHat,\LabelStar)\le\exp\left[-\left(1+\frac{2}{\num\snr}\log(\num\inprob)\right)\frac{\num\snr}{2}\right],
\]
which is valid under the assumption $\num\snr\gtrsim\sqrt{\num\inprob}\gtrsim1$.
Their assumption is order-wise more restrictive than that in our
Theorem~\ref{thm:SDP_error}, but their error bound has a better
second-order term in the exponent. We do note that their algorithm
is fairly technical: it requires data partition and the leave-one-out
tricks to ensure independence, degree truncation to regularize spectral
clustering, and blackbox solvers for $K$-means and matching problems.
In comparison, the SDP approach is much simpler conceptually.

Finally, we emphasize that we also provide robustness guarantees under
the monotone semirandom model and non-uniform edge probabilities.
In comparison, it is unclear if comparable robustness results can
be established for the algorithms above, as these algorithms and their
analyses make substantial use of the properties of the standard SBM,
particularly the complete independence among edges and the specific
form of the likelihood function. 

\section{Proof Outline \label{sec:proof_sketch}}

In this section we outline our proofs of the lower and upper bounds.
In the process we provide insights on how the error rate  $e^{-\proxnum\snr}$
arises and why SDP achieves it. 

\subsection{Proof outline of Theorem~\ref{thm:minimax_lower}\label{sec:proof_sketch_lower}}

The intuition behind the lower bound is relatively easy to describe.
To illustrate the idea, take as an example the Z2 model, where $\adj_{ij}\overset{\text{i.i.d.}}{\sim}N(1,\std^{2})$,
and assume that $\labelstar_{i}=1,\forall i$. It is not hard to see
that the error fraction $\misrate(\LabelHat,\LabelStar)$ for any
estimator $\LabelHat$ is lower bounded by the probability of recovering
the label $\labelstar_{1}$ for the first node given true labels of
the other nodes; see the Lemma~\ref{lem:global_to_local} for the
precise argument. For this one-node problem, the optimal Bayes estimate
of $\labelstar_{1}$ is given by the sign of the majority vote $\sum\nolimits _{j=1}^{\num}\adj_{1j}$:
\begin{equation}
\labelhat_{1}=\argmax_{\labels\in\{\pm1\}}\left\{ \labels\cdot\sum\nolimits _{j=1}^{\num}\adj_{1j}\right\} =\sign\left(\sum\nolimits _{j=1}^{\num}\adj_{1j}\right);\label{eq:vote_single}
\end{equation}
see Lemma~\ref{lem:lower_bound_local_risk}.  It follows that the
error probability of recovering $\labelstar_{1}=1$ is
\[
\P\Big\{\labelhat_{1}\neq1\Big\}=\P\left\{ \sum\nolimits _{j=1}^{\num}\adj_{1j}<0\right\} =\exp\Big[-(1+o(1))\cdot\num I(0)\Big],
\]
where the last step can be justified in general by the large deviation
theory, with $I(x)\coloneqq\sup_{t>0}[tx-\log\E e^{t(-\adj_{11})}]$
being the rate function (see, e.g., Cramer's Theorem \cite[Theorem 2.2.3]{dembo2010LDP}).
In our setting, a direct calculation suffices, as is done in Lemma~\ref{lem:local_risk_minimax}.
The error exponent 
\begin{equation}
I(0)=-\inf_{t>0}\left\{ \log\E e^{t(-\adj_{11})}\right\} =\snr\label{eq:rate_function}
\end{equation}
is precisely our SNR measure, a quantity we will encounter again in
proving the upper bound.

The above intuition remains valid for CBM and SBM, though the specific
forms of the majority voting procedure and the rate $\snr$ vary.
The complete proof is given in Section~\ref{sec:proof_minimax_lower_bound}.

\subsection{Proof outline of Theorem \ref{thm:SDP_error}\label{sec:proof_sketch_upper}}

To prove the upper bound for SDP,  we proceed in three steps:

\textbf{Step 1: }As mentioned in Section \ref{sec:intro}, we construct
a diagonal matrix $\D$ with $D_{ii}=\labelstar_{i}\sum_{j}\adj_{ij}\labelstar_{j}$,
which takes the same form as the ``dual certificate'' used in previous
work. The construction of $\D$ allows us to establish the \emph{basic
inequality}:
\[
0\le\left\langle -\D,\PTperp(\Yhat)\right\rangle +\left\langle \Adj-\E\Adj,\PTperp(\Yhat)\right\rangle ,\quad\text{for any \ensuremath{\Yhat\in\betterset}(\ensuremath{\Adj})};
\]
see the proof of Lemma \ref{lem:basic_inequality} for the details
of this critical step. Here $\PTperp$ is an appropriate projection
operator that satisfies $\Tr\big[\PTperp(\Yhat)\big]=\frac{1}{\num}\norm[\Yhat-\Ystar]1$,
thus exposing the $\ell_{1}$ error of $\Yhat$ that we seek to control.

\textbf{Step 2: }We proceed by showing that the second term $S_{2}:=\langle\Adj-\E\Adj,\PTperp(\Yhat)\rangle$
in the basic inequality is negligible compared to the first term $\langle-\D,\PTperp(\Yhat)\rangle$;
see Proposition~\ref{prop:S2} for a quantitative version of this
claim, whose proof involves certain trimming argument in the case
of CBM and SBM. Dropping $S_{2}$ from the basic inequality hence
yields
\begin{equation}
0\le\langle-\D,\PTperp(\Yhat)\rangle=\sum\nolimits _{i=1}^{\num}(-D_{ii})b_{i},\label{eq:basic_informal}
\end{equation}
where $b_{i}:=\big(\PTperp(\Yhat)\big)_{ii}$ satisfies $\sum_{i\in[\num]}b_{i}=m:=\frac{1}{\num}\norm[\Yhat-\Ystar]1$. 

\textbf{Step 3: }The $b_{i}$'s take fractional values in general,
but must be bounded in $[0,4]$ (cf.~\ref{fact:PTperp}). We use
this fact to upper bound the RHS of (\ref{eq:basic_informal}) by
its worst-case value, hence obtaining 
\begin{equation}
0\le\sum\nolimits _{i=1}^{\num}(-D_{ii})b_{i}\lesssim\max_{\substack{\calM\subseteq[\num]\\
\left|\calM\right|=m
}
}\sum\nolimits _{i\in\calM}(-D_{ii}).\label{eq:order_stat}
\end{equation}
 This argument is reminiscent of the \textquotedblleft order statistics\textquotedblright{}
analysis in \cite{fei2019exponential}, though here we provide a more
fine-grained bound; see Section~\ref{sec:S1+S2} for details. The
rest of this step, done in Lemma~\ref{lem:subset_sum_ineq}, establishes
a probabilistic bound for the RHS of (\ref{eq:order_stat}) and ultimately
gives rise to the error exponent $-\snr$. To illustrate the idea,
we again consider Z2 with $\labelstar_{i}\equiv1,$ in which case
$D_{ii}=\sum_{j\in[\num]}\adj_{ij}$. For a fixed set $\calM$ with
$\left|\calM\right|=m$, the RHS of (\ref{eq:order_stat}) can be
controlled using the Chernoff bound:
\begin{align*}
\P\left\{ \sum\nolimits _{i\in\calM}\left(\sum\nolimits _{j=1}^{\num}(-\adj_{ij})\right)>0\right\}  & \le\inf_{t>0}\E\exp\left[t\sum\nolimits _{i\in\calM}\left(\sum\nolimits _{j=1}^{\num}(-\adj_{ij})\right)\right]\\
 & =\left(\inf_{t>0}\E\exp(-t\adj_{11})\right)^{\num m}=e^{-nm\snr},
\end{align*}
where the last two steps follow from independence and the expression
(\ref{eq:rate_function}) for $\snr$. By a union bound over all ${\num \choose m}\approx\left(\frac{\num}{m}\right)^{m}$
such $\calM$'s, we obtain
\[
\P\left\{ \max_{\substack{\calM\subseteq[\num]\\
\left|\calM\right|=m
}
}\sum\nolimits _{i\in\calM}\left(\sum\nolimits _{j=1}^{\num}(-\adj_{ij})\right)>0\right\} \le\left(\frac{\num}{m}\right)^{m}\cdot e^{-nm\snr}=\left(e^{\log(\num/m)-\num\snr}\right)^{m}.
\]
If $\log(\num/m)-\num\snr<0$, then RHS above is $\ll1$ and thus
with high probability the negation of (\ref{eq:order_stat}) holds,
a contradiction. We therefore must have $\log(\num/m)-\num\snr\ge0$,
which implies the desired error bound $\frac{m}{\num}\le e^{-\num\snr}$.
The second-order term in the exponent comes from a more accurate calculation
for Steps 2 and 3. \\

The previous arguments are closely connected to our proof for the
lower bound outlined above. Note that the MLE of the entire vector
$\LabelStar=\onevec$ is given by the ``joint majority voting''
procedure
\[
\LabelMLE=\argmax_{\Label\in\{\pm1\}^{\num}}\left\{ \sum\nolimits _{i=1}^{\num}\labels_{i}\cdot\left(\sum\nolimits _{j=1}^{\num}\adj_{ij}\right)\right\} ;
\]
one should compare this equation with the ``single-node majority
voting'' in (\ref{eq:vote_single}). The maximality of the above
$\LabelMLE$ over $\LabelStar=\onevec$, as well as the fact that
$1-\labelmle_{i}\in\{0,2\}$, implies that 
\[
0\le\sum\nolimits _{i=1}^{\num}\left(1-\labelmle_{i}\right)\cdot\left(\sum\nolimits _{j=1}^{\num}(-\adj_{ij})\right)\lesssim\max_{\substack{\calM\subseteq[\num]\\
\left|\calM\right|=m
}
}\sum_{i\in\calM}(-D_{ii})
\]
if we set $m=\frac{1}{2}\sum_{i=1}^{n}(1-\labelmle_{i})=\num\misrate(\LabelMLE,\LabelStar)$.
Note that this inequality is the same as (\ref{eq:order_stat}), so
following the arguments above shows that the MLE satisfies the same
error bound $\frac{m}{\num}\le e^{-\num\snr}$. We therefore see that
the SDP solution closely approximates the MLE in the above precise
sense, and both of them achieve the Bayes rate. The form of the rate
$e^{-\num\snr}$ arises from a majority voting mechanism, in the proofs
for both lower and upper bounds.

Again, the above intuition remains valid for CBM and SBM, though the
calculation of the rate~$\snr$ varies. The details of the proof
are given in Appendices~\ref{sec:proof_SDP_error_rate} and~\ref{sec:proof_cluster_error_rate}.

\section{Discussion\label{sec:discuss}}

In this paper, we analyze the error rates of the SDP relaxation approach
for clustering under several random graph models, namely Z2, CBM and
the binary symmetric SBM, via a unified framework. We show that SDP
achieves an exponentially-decaying error with a sharp exponent, matching
the minimax lower bound for all three models. We also show that these
results continue to hold under monotone semirandom models, demonstrating
the robustness of SDP.

Immediate future directions include extensions to problems with multiple
and unbalanced clusters, as well as to closely related models such
as weighted SBM. It is also of interest to see if better estimates
of the second order term can be obtained, and if there is a fundamental
tradeoff between statistical optimality and robustness. More broadly,
it would be interesting to explore the applications of the techniques
in this paper in analyzing SDP relaxations for other discrete problems.

\section*{Acknowledgement}

Y. Fei and Y. Chen were partially supported by the National Science
Foundation CRII award 1657420 and grant 1704828.

\appendix
\appendixpage

\section{Preliminaries\label{sec:preliminary}}

In this section we record several notations and facts that are useful
for subsequent proofs. 

We first define a random variable $\var$ that encapsulates the distributions
of the three models:
\begin{itemize}
\item For Model \ref{mdl:Z2} (Z2), let $\var\sim N(1,\std^{2})$. 
\item For Model \ref{mdl:CBM} (CBM), let $\var$ have probability mass
function $\obsprob(1-\flipprob)\cdot\delta_{1}+\obsprob\flipprob\cdot\delta_{-1}+(1-\alpha)\cdot\delta_{0}$,
where $\delta_{a}$ denotes the Dirac delta function centered at $a$.
 
\item For Model \ref{mdl:SBM} (SBM), let $\var=Y-Z$, where $Y\sim\Bern(\inprob)$,
$Z\sim\Bern(\outprob)$, and $Y,Z$ are independent.
\end{itemize}
It can be seen that under Model \ref{mdl:Z2} or \ref{mdl:CBM}, we
have $\adj_{ij}\sim\labelstar_{i}\labelstar_{j}\var$ (here $\sim$
means equality in distribution); under Model~\ref{mdl:SBM}, we have
$\adj_{ii'}-\adj_{ij}\sim\var$ if $\labelstar_{i}=\labelstar_{i'}=-\labelstar_{j}$. 

Let $t^{*}$ be the minimizer of the moment generating function $t\mapsto\E e^{-t\var}$,
which has the explicit expression

\begin{equation}
t^{*}\coloneqq\begin{cases}
\frac{1}{\std^{2}}, & \text{for Model \ref{mdl:Z2}},\\
\frac{1}{2}\log\frac{1-\flipprob}{\flipprob}, & \text{for Model \ref{mdl:CBM}},\\
\frac{1}{2}\log\frac{\inprob(1-\outprob)}{\outprob(1-\inprob)}, & \text{for Model \ref{mdl:SBM}}.
\end{cases}\label{eq:def_tstar}
\end{equation}
Note that $t^{*}>0$. We later verify that $\E e^{-t\var}\approx e^{-\snr}$
for all three models (see Facts~\ref{fact:Z2_magic_identity}, \ref{fact:CBM_magic_identity}
and~\ref{fact:SBM_magic_identity}). Also define the quantity 
\begin{equation}
\tune\coloneqq\begin{cases}
0 & \text{for Model \ref{mdl:Z2} and \ref{mdl:CBM}},\\
\frac{1}{2t^{*}}\log\frac{1-\outprob}{1-\inprob}, & \text{for Model \ref{mdl:SBM}},
\end{cases}\label{eq:def_tune_param}
\end{equation}
which plays a role only in Model~\ref{mdl:SBM} (SBM).

Finally, for Model \ref{mdl:CBM}, we let $\inprob\coloneqq\obsprob(1-\flipprob)$
and $\outprob\coloneqq\obsprob\flipprob$; this notation is chosen
to bring out the similarity between Models \ref{mdl:CBM} and \ref{mdl:SBM}.\\

We record several simple estimates for the above quantities $t^{*}$
and $\tune$ as well as the SNR measure $\snr$ defined in~(\ref{eq:snr}).
The proofs are given in Sections~\ref{sec:proof_CBM_estimate} and~\ref{sec:proof_SBM_estimate}
to follow.
\begin{fact}
\label{fact:CBM_estimate}Under Model \ref{mdl:CBM} with the notation
$\inprob\coloneqq\obsprob(1-\flipprob)$ and $\outprob\coloneqq\obsprob\flipprob$,
if $0<\outprob\le\inprob\le1$, then
\begin{enumerate}[label={(\alph*)},ref={\thefact(\alph*)}]
\item \label{fact:CBM_tstar_lesssim_(p-q)/p}$t^{*}\le\frac{1-\flipprob}{2\flipprob}\cdot\frac{\inprob-\outprob}{\inprob}$. 
\item \label{fact:CBM_renyi_equivalence}$\snr\in\left[\frac{(\inprob-\outprob)^{2}}{4\inprob},\frac{(\inprob-\outprob)^{2}}{\inprob}\right]$.
\end{enumerate}
\end{fact}
\begin{fact}
\label{fact:SBM_estimate}Under Model \ref{mdl:SBM}, if $0<\outprob<\inprob<1$,
then the following hold.
\begin{enumerate}[label={(\alph*)},ref={\thefact(\alph*)}]
\item \label{fact:SBM_tune_param_vs_p_q}$\tune\in(\outprob,\inprob)$.
\item \label{fact:SBM_renyi_equivalence}If in addition $p\le1-c$ for some
constant $c\in(0,1)$, then $\snr\asymp\frac{(\inprob-\outprob)^{2}}{\inprob}$.
\item \label{fact:SBM_tstar_lesssim_(p-q)/p}If in addition $\inprob\le1-c$
and $\outprob\ge c_{0}\inprob$ for some constants $c\in(0,1)$, $c_{0}\in(0,1)$,
then $t^{*}\lesssim\frac{\inprob-\outprob}{\inprob}$. 
\end{enumerate}
\end{fact}

\subsection{Proof of Fact \ref{fact:CBM_estimate} \label{sec:proof_CBM_estimate}}

Recall the shorthands $\inprob\coloneqq\obsprob(1-\flipprob)$ and
$\outprob\coloneqq\obsprob\flipprob$ introduced for Model \ref{mdl:CBM}.
For part (a) of the fact, by definition of $t^{*}$ in Equation (\ref{eq:def_tstar}),
we have 
\[
t^{*}=\frac{1}{2}\log\left(1+\frac{\inprob-\outprob}{\outprob}\right)\overset{(i)}{\le}\frac{\inprob-\outprob}{2\outprob}\overset{(ii)}{=}\frac{1-\flipprob}{2\flipprob}\cdot\frac{\inprob-\outprob}{\inprob},
\]
where step $(i)$ holds since the fact that $1+x\le e^{x}$ for $x\in\real$
implies $\log(1+x)\le x$ for $x>-1$, and step $(ii)$ holds by the
fact that $\outprob=\frac{\flipprob}{1-\flipprob}\inprob$. 

For part (b), recalling the definition of $\snr$ in Equation (\ref{eq:snr}),
we have 
\[
\snr=\frac{\left(\sqrt{\inprob}-\sqrt{\outprob}\right)^{2}\left(\sqrt{\inprob}+\sqrt{\outprob}\right)^{2}}{\left(\sqrt{\inprob}+\sqrt{\outprob}\right)^{2}}=\frac{\left(\inprob-\outprob\right)^{2}}{\inprob+\outprob+2\sqrt{\inprob\outprob}}.
\]
Some algebra shows that $\snr\le\frac{\left(\inprob-\outprob\right)^{2}}{\inprob}$
and $\snr\ge\frac{\left(\inprob-\outprob\right)^{2}}{\inprob+\inprob+2\inprob}=\frac{\left(\inprob-\outprob\right)^{2}}{4\inprob}$.

\subsection{Proof of Fact \ref{fact:SBM_estimate} \label{sec:proof_SBM_estimate}}

For part (a), recalling the definition of $\tune$ in Equation~(\ref{eq:def_tune_param}),
we obtain by direct calculation the identity
\[
\inprob-\tune=\left[\log\frac{\inprob(1-\outprob)}{\outprob(1-\inprob)}\right]^{-1}\left[\inprob\log\frac{\inprob}{\outprob}+(1-\inprob)\log\frac{1-\inprob}{1-\outprob}\right].
\]
The quantity inside the second bracket on the RHS is positive, as
it is the KL divergence between $\Bern(\inprob)$ and $\Bern(\outprob)$
with $\inprob\neq\outprob$. We also have $\log\frac{\inprob(1-\outprob)}{\outprob(1-\inprob)}>0$
since $0<\outprob<\inprob<1$. It follows that $\inprob-\tune>0$
as claimed. A similar argument shows that $\tune-\outprob>0$. 

Part (b) is a partial result of \cite[Lemma B.1]{zhang2016minimax}. 

For part (c), recall the definition $t^{*}:=\frac{1}{2}\left(\log\frac{\inprob}{\outprob}+\log\frac{1-\outprob}{1-\inprob}\right)$
in Eq.~(\ref{eq:def_tstar}). We consider two cases. If $\frac{\inprob}{\outprob}\ge\frac{1-\outprob}{1-\inprob}$,
we have 
\[
t^{*}\le\log\frac{\inprob}{\outprob}\overset{(i)}{\le}\frac{\inprob}{\outprob}-1\overset{(ii)}{\le}\frac{\inprob-\outprob}{c_{0}\inprob},
\]
where step $(i)$ holds since $\log(x)\le x-1,\forall x>0$, and step
$(ii)$ holds by assumption $c_{0}\inprob\le\outprob$.  If $\frac{\inprob}{\outprob}\le\frac{1-\outprob}{1-\inprob}$,
we have
\[
t^{*}\le\log\frac{1-\outprob}{1-\inprob}\overset{}{\le}\frac{1-\outprob}{1-\inprob}-1\overset{(i)}{\le}\frac{\inprob-\outprob}{c}\le\frac{\inprob-\outprob}{c\inprob},
\]
where step $(i)$ holds by the assumption that $\inprob\le1-c$. In
both cases, we have $t^{*}\lesssim\frac{\inprob-\outprob}{\inprob}$
as claimed.

\section{Proof of Theorem \ref{thm:minimax_lower}\label{sec:proof_minimax_lower_bound}}

In this section we prove Theorem \ref{thm:minimax_lower} under Models
\ref{mdl:Z2} and \ref{mdl:CBM}, following a similar strategy as
in the proof of \cite[Theorem 1.1]{zhang2016minimax}. We make use
of the definitions and facts given in Section~\ref{sec:preliminary}.

For simplicity, in the sequel we write $\paramset\equiv\paramset(\num):=\left\{ \pm1\right\} ^{\num}$.
Let $\prior$ be the uniform prior over all the elements in $\paramset$.
Define the global Bayesian risk
\[
\bayrisk_{\prior}\big(\paramset,\LabelHat\big)\coloneqq\frac{1}{\left|\paramset\right|}\sum_{\Label\in\paramset}\E_{\Label}\misrate(\LabelHat,\Label),
\]
and the local Bayesian risk for the first node
\[
\bayrisk_{\prior}\big(\paramset,\labelhat(1)\big)\coloneqq\frac{1}{\left|\paramset\right|}\sum_{\Label\in\paramset}\E_{\Label}\misrate\big(\labelhat(1),\labels(1)\big).
\]
In the above, the quantity $\misrate(\labelhat(1),\labels(1))$ denotes
the loss on the first node, defined as 
\[
\misrate\big(\labelhat(1),\labels(1)\big)\coloneqq\frac{1}{\left|\calS_{\Label}(\LabelHat)\right|}\sum_{\Label'\in\calS_{\Label}(\LabelHat)}\indic\left\{ \labels'(1)\neq\labels(1)\right\} ,
\]
where $\calS_{\Label}(\LabelHat)\coloneqq\left\{ g\LabelHat:g\in\{\pm1\},\frac{1}{\num}\sum_{i\in[\num]}\indic\{g\labelhat_{i}\neq\labels_{i}\}=\misrate(\LabelHat,\Label)\right\} $.
The following lemma shows that these risks are equal.
\begin{lem}
\label{lem:global_to_local}Under Models \ref{mdl:Z2} and \ref{mdl:CBM},
we have 
\[
\inf_{\LabelHat}\bayrisk_{\prior}\big(\paramset,\LabelHat\big)=\inf_{\LabelHat}\bayrisk_{\prior}\big(\paramset,\labelhat(1)\big).
\]
\end{lem}
\begin{proof}
The lemma essentially follows from the symmetry/exchangeability property
of Models~\ref{mdl:Z2} (Z2) and \ref{mdl:CBM} (CBM). Rigorous proof
of this intuitive result is however quite technical, as the definition
of clustering error involves a global sign flipping. Fortunately,
most of the work has been done in~\cite{zhang2016minimax}. In particular,
note that the parameter space $\paramset$ is closed under permutation
in the sense that for any label vector $\Label\in\paramset$ and any
permutation $\pi$ on $\left[\num\right]$, the new label vector $\Label'$
defined by $\labels'(i)\coloneqq\labels(\pi^{-1}(i))$ also belongs
to $\paramset$. It can also be seen that both Models~\ref{mdl:Z2}
(Z2) and~\ref{mdl:CBM} (CBM) are homogeneous, i.e., the distribution
of each $\adj_{ij}$ is uniquely determined by the sign of $\labelstar_{i}\labelstar_{j}$.
Consequently, for Model \ref{mdl:CBM} (CBM) this lemma immediately
follows from Lemma 2.1 in \cite{zhang2016minimax}, as its proof applies
without change. For Model \ref{mdl:Z2} (Z2) in which the distribution
of $\Adj$ is continuous, we note that the proof of Lemma 2.1 in \cite{zhang2016minimax}
continues to hold when summations therein are replaced by appropriate
integrations. 
\end{proof}
With the above lemma, it suffices to lower bound the local Bayes risk.
This task can be further reduced to computing the tail probability
of a certain sum of independent copies of the random variable $\var$
defined in Section~\ref{sec:preliminary}. This is done in the following
lemma. 
\begin{lem}
\label{lem:lower_bound_local_risk}Let $\prior$ be the uniform prior
over all elements in $\paramset$. Under Models \ref{mdl:Z2} and
\ref{mdl:CBM}, we have 
\[
\bayrisk_{\prior}\big(\paramset,\labelhat(1)\big)\ge\P\left(\sum_{i\in\left[\num-1\right]}Z_{i}\ge0\right),
\]
where $\left\{ Z_{i}\right\} $ are i.i.d.~copies of $-\var$. 
\end{lem}
This lemma is analogous to Lemma 5.1 in \cite{zhang2016minimax}.
We provide the proof in Section \ref{sec:proof_lower_bound_local_risk}.

Finally, the lemma below provides an explicit lower bound of the above
tail probability in terms of the SNR measure $\snr$. 
\begin{lem}
\label{lem:local_risk_minimax} Let $\left\{ Z_{i}\right\} $ be
i.i.d.~copies of $-\var$. For Model~\ref{mdl:Z2} and Model~\ref{mdl:CBM}
with  $\snr=o(1)$, if $\num\snr\to\infty$, then there exists $\xi=o(1)$
such that 
\[
\P\left(\frac{1}{\num-1}\sum_{i\in\left[\num-1\right]}Z_{i}\ge0\right)\ge\exp\left[-\left(1+\xi\right)(\num-1)\snr\right].
\]
\end{lem}
This lemma is analogous to Lemma 5.2 in \cite{zhang2016minimax}.
We provide the proof in Section~\ref{sec:proof_Z2_local_risk_minimax}
for Model~\ref{mdl:Z2} (Z2) and in Section~\ref{sec:proof_CBM_local_risk_minimax}
for Model~\ref{mdl:CBM} (CBM).\\

We are now ready to prove Theorem~\ref{thm:minimax_lower}. Note
that 
\[
\inf_{\LabelHat}\sup_{\Label\in\paramset}\E_{\Label}\misrate(\LabelHat,\Label)\ge\inf_{\LabelHat}\bayrisk_{\prior}(\paramset,\LabelHat),
\]
since the Bayes risk lower bounds the minimax risk. To complete the
proof, we continue the above inequality by successively invoking Lemmas~\ref{lem:global_to_local},
\ref{lem:lower_bound_local_risk} and~\ref{lem:local_risk_minimax}.

\subsection{Proof of Lemma \ref{lem:lower_bound_local_risk}\label{sec:proof_lower_bound_local_risk}}

Recall that $\bayrisk_{\prior}(\paramset,\labelhat(1))$ is defined
as
\[
\bayrisk_{\prior}\big(\paramset,\labelhat(1)\big)\coloneqq\frac{1}{\left|\paramset\right|}\sum_{\Label\in\paramset}\E_{\Label}\misrate\big(\labels(1),\labelhat(1)\big).
\]
For each $\Label_{0}\in\paramset$, we generate a new assignment $\Label[\Label_{0}]$
based on $\Label_{0}$ by setting $\labels[\labels_{0}](1)\coloneqq-\labels_{0}(1)$
and $\labels[\labels_{0}](i)\coloneqq\labels_{0}(i)$ for all $i\in\left[\num\right]\setminus\left\{ 1\right\} $.
It can be seen that $\Label[\Label_{0}]\in\paramset$ and the Hamming
distance between $\Label_{0}$ and $\Label[\Label_{0}]$ is 1. In
addition, for any $\Label_{1},\Label_{2}\in\paramset$ with $\Label_{1}\ne\Label_{2}$,
we have $\Label[\Label_{1}]\ne\Label[\Label_{2}]$. This bijection
implies that $\left\{ \Label[\Label_{0}]:\Label_{0}\in\paramset\right\} =\paramset$.
Consequently, continuing from the last displayed equation we obtain
\[
\bayrisk_{\prior}(\paramset,\labelhat(1))=\frac{1}{\left|\paramset\right|}\sum_{\Label_{0}\in\paramset}\frac{1}{2}\Big(\E_{\Label_{0}}\misrate\big(\labelhat(1),\labels_{0}(1)\big)+\E_{\Label[\Label_{0}]}\misrate\big(\labelhat(1),\labels[\labels_{0}](1)\big)\Big),
\]
whence
\begin{align}
\inf_{\LabelHat}\bayrisk_{\prior}(\paramset,\labelhat(1)) & \ge\frac{1}{\left|\paramset\right|}\sum_{\Label_{0}\in\paramset}\frac{1}{2}\inf_{\LabelHat}\Big(\E_{\Label_{0}}\misrate\big(\labelhat(1),\labels_{0}(1)\big)+\E_{\Label[\Label_{0}]}\misrate\big(\labelhat(1),\labels[\labels_{0}](1)\big)\Big).\label{eq:decompose_local_risk}
\end{align}

We proceed to compute the infimum above for a given $\Label_{0}\in\paramset$.
Let $\Labeltilde$ be the Bayes estimator that attains the infimum.
Since $\Label_{0}$ and $\Label[\Label_{0}]$ only differ at the first
node, we must have $\labeltilde(i)=\labels[\labels_{0}](i)=\labels_{0}(i)$
for all $i\in\left[\num\right]\backslash\left\{ 1\right\} $, and
either $\labeltilde(1)=\labels_{0}(1)$ or $\labeltilde(1)=\labels[\labels_{0}](1)$.
Now the problem is reduced to a test between two distributions $\P_{\Label_{0}}$
and $\P_{\Label\left[\Label_{0}\right]}$.  Since the prior $\prior$
is uniform, $\labeltilde(1)$ is given by the likelihood ratio test
$\frac{\P_{\Label_{0}}(\Adj)}{\P_{\Label\left[\Label_{0}\right]}(\Adj)}\gtrless1$.
The following lemma gives the explicit form of this test. Here we
let $J_{0}\coloneqq\left\{ u\in\left[\num\right]\backslash\left\{ 1\right\} :\labels{}_{0}(u)=\labels{}_{0}(1)\right\} $
and $J_{1}\coloneqq\left\{ u\in\left[\num\right]:\labels{}_{0}(u)=\labels[\labels{}_{0}](1)\right\} $.\footnote{In Lemma \ref{lem:bayes_estimator} and its proof, we adopt the convention
that $\sum_{u\in J}f(u)=0$ and $\prod_{u\in J}f(u)=1$ if $J=\emptyset$. }
\begin{lem}
\label{lem:bayes_estimator}Let $\Label_{0}$ and $\Label[\Label_{0}]$
be defined as above. Under Models~\ref{mdl:Z2} and~\ref{mdl:CBM},
we have that 
\[
\labeltilde(1)=\begin{cases}
\labels_{0}(1), & \text{if }\sum_{u\in J_{0}}\adj_{1u}\ge\sum_{u\in J_{1}}\adj_{1u},\\
\labels[\labels{}_{0}](1), & \text{otherwise.}
\end{cases}
\]
\end{lem}
The lemma follows from a routine calculation of the likelihood. We
give proof in Section~\ref{sec:proof_Z2_bayes_estimator} for Model
\ref{mdl:Z2} (Z2) and in Section~\ref{sec:proof_CBM_bayes_estimator}
for Model \ref{mdl:CBM} (CBM). 

From Lemma \ref{lem:bayes_estimator} we have
\begin{align*}
\E_{\Label_{0}}\misrate\big(\labeltilde(1),\labels_{0}(1)\big) & =\P_{\Label_{0}}\bigg(\sum_{u\in J_{0}}\adj_{1u}<\sum_{u\in J_{1}}\adj_{1u}\bigg),\qquad\text{and}\\
\E_{\Label[\Label_{0}]}\misrate\big(\labeltilde(1),\labels[\labels{}_{0}](1)\big) & =\P_{\Label\left[\Label_{0}\right]}\bigg(\sum_{u\in J_{0}}\adj_{1u}\ge\sum_{u\in J_{1}}\adj_{1u}\bigg).
\end{align*}
Recalling the distribution of $\left\{ \adj_{1u}\right\} $, the definition
of $\var$ and that $Z_{i}\overset{\text{i.i.d.}}{\sim}-\var$, we
see that both probabilities above equal $\P\big(\sum_{i\in\left[\num-1\right]}Z_{i}\ge0\big)$.
Combining with the bound~(\ref{eq:decompose_local_risk}), we obtain
the desired inequality $\inf_{\LabelHat}\bayrisk_{\prior}(\paramset,\labelhat(1))\ge\P\big(\sum_{i\in\left[\num-1\right]}Z_{i}\ge0\big).$

\subsubsection{Proof of Lemma \ref{lem:bayes_estimator} for Model \ref{mdl:Z2}
(Z2) \label{sec:proof_Z2_bayes_estimator}}

Since $\Label_{0}$ and $\Label[\Label_{0}]$ only differ at the first
node, the likelihood ratio $\frac{\P_{\Label_{0}}(\Adj)}{\P_{\Label\left[\Label_{0}\right]}(\Adj)}$
only depends on the first row and column of $\Adj$. In particular,
recalling that $\{A_{1u}\}$ are Gaussian under Model \ref{mdl:Z2},
we have
\begin{align*}
\frac{\P_{\Label_{0}}(\Adj)}{\P_{\Label\left[\Label_{0}\right]}(\Adj)} & =\frac{\prod_{u\in J_{0}}\exp\left[-\left(\adj_{1u}-1\right)^{2}/(2\std^{2})\right]\times\prod_{u\in J_{1}}\exp\left[-\left(\adj_{1u}+1\right)^{2}/(2\std^{2})\right]}{\prod_{u\in J_{0}}\exp\left[-\left(\adj_{1u}+1\right)^{2}/(2\std^{2})\right]\times\prod_{u\in J_{1}}\exp\left[-\left(\adj_{1u}-1\right)^{2}/(2\std^{2})\right]}\\
 & =\frac{\exp\left[-\left(\sum_{u\in J_{0}}\left(\adj_{1u}-1\right)^{2}+\sum_{u\in J_{1}}\left(\adj_{1u}+1\right)^{2}\right)/(2\std^{2})\right]}{\exp\left[-\left(\sum_{u\in J_{0}}\left(\adj_{1u}+1\right)^{2}+\sum_{u\in J_{1}}\left(\adj_{1u}-1\right)^{2}\right)/(2\std^{2})\right]}.
\end{align*}
Some algebra shows that 
\[
\frac{\P_{\Label_{0}}(\Adj)}{\P_{\Label\left[\Label_{0}\right]}(\Adj)}\gtrless1\quad\Longleftrightarrow\quad\sum_{u\in J_{0}}\adj_{1u}\gtrless\sum_{u\in J_{1}}\adj_{1u}.
\]
The result follows from the fact that the likelihood ratio test is
Bayes-optimal for binary hypotheses under a uniform prior. 

\subsubsection{Proof of Lemma \ref{lem:bayes_estimator} for Model \ref{mdl:CBM}
(CBM) \label{sec:proof_CBM_bayes_estimator}}

Similarly to the previous section, we recall the distribution of $\{A_{1u}\}$
under Model \ref{mdl:CBM} to obtain
\begin{align*}
\frac{\P_{\Label_{0}}(\Adj)}{\P_{\Label\left[\Label_{0}\right]}(\Adj)} & =\frac{\prod_{u\in J_{0}}[\obsprob(1-\flipprob)]^{\indic\left\{ \adj_{1u}=1\right\} }[\obsprob\flipprob]^{\indic\left\{ \adj_{1u}=-1\right\} }\times\prod_{u\in J_{1}}[\obsprob(1-\flipprob)]^{\indic\left\{ \adj_{1u}=-1\right\} }[\obsprob\flipprob]^{\indic\left\{ \adj_{1u}=1\right\} }}{\prod_{u\in J_{0}}[\obsprob(1-\flipprob)]^{\indic\left\{ \adj_{1u}=-1\right\} }[\obsprob\flipprob]^{\indic\left\{ \adj_{1u}=1\right\} }\times\prod_{u\in J_{1}}[\obsprob(1-\flipprob)]^{\indic\left\{ \adj_{1u}=1\right\} }[\obsprob\flipprob]^{\indic\left\{ \adj_{1u}=-1\right\} }}.
\end{align*}
Since $\flipprob\in(0,\frac{1}{2})$, some algebra shows that 
\[
\frac{\P_{\Label_{0}}(\Adj)}{\P_{\Label\left[\Label_{0}\right]}(\Adj)}\gtrless1\quad\Longleftrightarrow\quad\sum_{u\in J_{0}}\adj_{1u}\gtrless\sum_{u\in J_{1}}\adj_{1u}.
\]
The result follows from the fact that the likelihood ratio test is
Bayes-optimal for binary hypotheses under a uniform prior.

\subsection{Proof of Lemma \ref{lem:local_risk_minimax} for Model \ref{mdl:Z2}
(Z2) \label{sec:proof_Z2_local_risk_minimax}}

Let $n'\coloneqq\num-1$, $p(z)$ be the pdf of $Z_{1}$, and $M(t)$
be the moment generating function of $Z_{1}$. Since $Z_{1}\sim-\var\sim N(-1,\std^{2})$,
we can compute $M(t)=\exp(-t+\frac{1}{2}t^{2}\std^{2}).$ Recalling
$t^{*}=\frac{1}{\std^{2}}$ as defined in Equation (\ref{eq:def_tstar})
and the definition of $\snr$ in Equation (\ref{eq:snr}), we obtain
\begin{align*}
M\left(t^{*}\right) & =\exp\left(-\frac{1}{2\std^{2}}\right)=\exp\left(-\snr\right).
\end{align*}
Let $\delta\coloneqq\left(2\num\snr\right)^{-\frac{1}{4}}$, $S_{n'}\coloneqq\sum_{i\in\left[n'\right]}Z_{i}$
and $S_{n'}(\z)\coloneqq\sum_{i\in\left[n'\right]}z_{i}$. We have
\begin{align*}
\P\left(S_{n'}\ge0\right) & \ge\int_{\left\{ \z:S_{n'}(\z)\in[0,n'\delta]\right\} }\prod_{i\in\left[n'\right]}p(z_{i})\mathrm{d}\z\\
 & \ge\frac{\left(M\left(t^{*}\right)\right)^{n'}}{\exp\left(n't^{*}\delta\right)}\int_{\left\{ \z:S_{n'}(\z)\in[0,n'\delta]\right\} }\prod_{i\in\left[n'\right]}\frac{\exp\left(t^{*}z_{i}\right)p(z_{i})}{M\left(t^{*}\right)}\mathrm{d}\z,
\end{align*}
where the last step holds since $\exp\left(n't^{*}\delta\right)\ge\exp\left(t^{*}\sum_{i\in\left[n'\right]}z_{i}\right)=\prod_{i\in\left[n'\right]}\exp\left(t^{*}z_{i}\right)$
given that $\sum_{i\in\left[n'\right]}z_{i}=S_{n'}(\z)\le n'\delta$.
Let $q(w)\coloneqq\frac{\exp\left(t^{*}w\right)p(w)}{M\left(t^{*}\right)}$
and we have 
\[
\P\left(S_{n'}\ge0\right)\ge\exp\left(-n'\snr\right)\exp\left(-n't^{*}\delta\right)\int_{\left\{ \z:S_{n'}(\z)\in[0,n'\delta]\right\} }\prod_{i\in\left[n'\right]}q(z_{i})\mathrm{d}\z.
\]
Note that $q(w)$ is a pdf since $\int_{w}q(w)\mathrm{d}w=1$ and
$q(w)\ge0$ for any $w$. Let $W_{1},W_{2},\ldots,W_{n'}$ be i.i.d.$\ $random
variables with pdf $q(w)$. We have 
\begin{equation}
\P\left(S_{n'}\ge0\right)\ge\exp\left(-n't^{*}\delta\right)\P\bigg(\frac{1}{n'}\sum_{i\in\left[n'\right]}W_{i}\in[0,\delta]\bigg)\exp\left(-n'\snr\right)\eqqcolon Q_{1}Q_{2}\exp\left(-n'\snr\right).\label{eq:Z2_P(S_n' >=00003D 0) lower bound}
\end{equation}

\subsubsection{Controlling $Q_{1}$}

It can be seen that $t^{*}\delta=2\snr\cdot\left(2\num\snr\right)^{-\frac{1}{4}}$
by the definitions of $t^{*}$ and $\delta$. Therefore, for some
constant $C'>0$ we have 
\[
Q_{1}\ge\exp\left[-C'\left(\frac{1}{\num\snr}\right)^{\frac{1}{4}}n'\snr\right].
\]

\subsubsection{Controlling $Q_{2}$}

Recall that $p(w)$ is the pdf for $N(-1,\std^{2}$). A closer look
at $q(w)$ yields 
\[
q(w)=\exp\left(\frac{w}{\std^{2}}\right)\exp\left(\frac{1}{2\std^{2}}\right)\frac{1}{\sqrt{2\pi\std^{2}}}\exp\left(-\frac{\left(w+1\right)^{2}}{2\std^{2}}\right)=\frac{1}{\sqrt{2\pi\std^{2}}}\exp\left(-\frac{w^{2}}{2\std^{2}}\right).
\]
Therefore, $q(w)$ is the pdf for $N(0,\std^{2})$. Define $V\coloneqq\Var\left(\frac{1}{n'}\sum_{i\in\left[n'\right]}W_{i}\right)$
and we have $V=\frac{1}{n'}\Var\left(W_{1}\right)=\frac{\std^{2}}{n'}=\frac{1}{2n'\snr}$.

Recall that $\delta\coloneqq\left(2\num\snr\right)^{-\frac{1}{4}}$
and $\num'\coloneqq\num-1$. Using Chebyshev's inequality we have
\[
\P\left(\left|\frac{1}{n'}\sum_{i\in\left[n'\right]}W_{i}\right|>\delta\right)\le\frac{V}{\delta^{2}}\le\frac{C}{\sqrt{\num\snr}},
\]
for some constant $C>0$, where the second step holds by $\num\asymp\num'$.
Therefore, there exist some constants $\consts>0$ and $c'\in(0,1)$
depending only on $\consts$ such that 
\[
Q_{2}=\P\left(\frac{1}{n'}\sum_{i\in\left[n'\right]}W_{i}\in[0,\delta]\right)=\frac{1}{2}\left(1-\P\left(\left|\frac{1}{n'}\sum_{i\in\left[n'\right]}W_{i}\right|>\delta\right)\right)\ge c'
\]
when $\num\snr\ge\consts$ (implied by our assumption $\num\snr\to\infty$).

\subsubsection{Putting together}

Returning to Equation (\ref{eq:Z2_P(S_n' >=00003D 0) lower bound}),
we have
\begin{align*}
\P\left(S_{n'}\ge0\right) & \ge c'\cdot\exp\left[-C'\left(\frac{1}{\num\snr}\right)^{\frac{1}{4}}n'\snr\right]\cdot\exp\left[-\num'\snr\right]\\
 & =\exp\left[-\left(1+C'\left(\frac{1}{\num\snr}\right)^{\frac{1}{4}}+\frac{1}{\num'\snr}\log\frac{1}{c'}\right)\num'\snr\right].
\end{align*}
The desired inequality follows by taking $\xi\coloneqq C'\left(\frac{1}{\num\snr}\right)^{\frac{1}{4}}+\frac{1}{\num'\snr}\log\frac{1}{c'}$
and noting that $\xi=o(1)$ under our assumption $\num\snr\to\infty$.

\subsection{Proof of Lemma \ref{lem:local_risk_minimax} for Model \ref{mdl:CBM}
(CBM) \label{sec:proof_CBM_local_risk_minimax}}

Let $n'\coloneqq\num-1$, $p(z)$ be the probability mass function
of $Z_{1}$, and $M(t)$ be the moment generating function of $Z_{1}$.
Recall that $\inprob\coloneqq\obsprob(1-\flipprob)$ and $\outprob\coloneqq\obsprob\flipprob$.
Since $Z_{1}\sim-\var$, we can compute $M(t)=(1-\obsprob)+\inprob e^{-t}+\outprob e^{t}.$
Noting that $\obsprob=\inprob+\outprob$ and recalling $t^{*}=\frac{1}{2}\log\frac{\inprob}{\outprob}$
as defined in Equation (\ref{eq:def_tstar}), we obtain 
\[
M\left(t^{*}\right)=(1-\obsprob)+2\sqrt{\inprob\outprob}=1-\snr.
\]
Let $\delta\coloneqq V^{\frac{1}{4}}\snr^{\frac{1}{2}}\left(t^{*}\right)^{-\frac{1}{2}}$,
$S_{n'}\coloneqq\sum_{i\in\left[n'\right]}Z_{i}$ and $S_{n'}(\z)\coloneqq\sum_{i\in\left[n'\right]}z_{i}$.
We have 
\begin{align*}
\P\left(S_{n'}\ge0\right) & \ge\sum_{\left\{ \z:S_{n'}(\z)\in[0,n'\delta]\right\} }\prod_{i\in\left[n'\right]}p(z_{i})\\
 & \ge\frac{\left(M\left(t^{*}\right)\right)^{n'}}{\exp\left(n't^{*}\delta\right)}\sum_{\left\{ \z:S_{n'}(\z)\in[0,n'\delta]\right\} }\prod_{i\in\left[n'\right]}\frac{\exp\left(t^{*}z_{i}\right)p(z_{i})}{M\left(t^{*}\right)},
\end{align*}
where the last step holds since $\exp\left(n't^{*}\delta\right)\ge\exp\left(t^{*}\sum_{i\in\left[n'\right]}z_{i}\right)=\prod_{i\in\left[n'\right]}\exp\left(t^{*}z_{i}\right)$
given that $\sum_{i\in\left[n'\right]}z_{i}=S_{n'}(\z)\le n'\delta$.
Let $q(w)\coloneqq\frac{\exp\left(t^{*}w\right)p(w)}{M\left(t^{*}\right)}$
for $w\in\{-1,0,1\}$ and we have 
\[
\P\left(S_{n'}\ge0\right)\ge\exp\left(n'\log(1-\snr)\right)\exp\left(-n't^{*}\delta\right)\sum_{\left\{ \z:S_{n'}(\z)\in[0,n'\delta]\right\} }\prod_{i\in\left[n'\right]}q(z_{i}).
\]
Noting that $q(w)$ is a pmf,  we let $W_{1},W_{2},\ldots,W_{n'}$
be i.i.d.$\ $random variables with pmf $q(w)$. We have 
\begin{equation}
\P\left(S_{n'}\ge0\right)\ge\exp\left(-n't^{*}\delta\right)\P\left(\frac{1}{n'}\sum_{i\in\left[n'\right]}W_{i}\in[0,\delta]\right)\exp\left(n'\log(1-\snr)\right)\eqqcolon Q_{1}Q_{2}Q_{3}.\label{eq:CBM_P(S_n' >=00003D 0) lower bound}
\end{equation}

\subsubsection{Controlling $Q_{2}$}

A closer look at $q(w)$ yields 
\[
q(w)=\begin{cases}
\frac{1}{M\left(t^{*}\right)}\sqrt{\inprob\outprob}, & \text{if }w=1\text{ or }w=-1,\\
\frac{1}{M\left(t^{*}\right)}(1-\obsprob), & \text{if }w=0,
\end{cases}
\]
whence $\Var(W_{1})=\frac{2}{M\left(t^{*}\right)}\sqrt{\inprob\outprob}.$
Define $V\coloneqq\Var\left(\frac{1}{n'}\sum_{i\in\left[n'\right]}W_{i}\right)$
and we have $V=\frac{1}{n'}\Var\left(W_{1}\right)=\frac{2\sqrt{\inprob\outprob}}{n'M\left(t^{*}\right)}$.
We need the following estimates.
\begin{lem}
\label{fact:CBM_V_s_tstar}If $\flipprob\in(0,\frac{1}{2})$ is a
constant and $0<\outprob<\inprob\le1-c$ for some constant $c\in(0,1)$,
then there exist constants $C,C_{1}>0$ such that 
\[
V\le\frac{4\inprob}{c\num},\qquad\qquad\frac{V\left(t^{*}\right)^{2}}{\snr^{2}}\le\frac{C}{\num\snr},\qquad\qquad t^{*}\sqrt{V}\le C_{1}\sqrt{\frac{\snr}{\num}}.
\]
\end{lem}
\begin{proof}
By Fact \ref{fact:CBM_renyi_equivalence}, we have $\snr\le\frac{(\inprob-\outprob)^{2}}{\inprob}\le\inprob\le1-c$
and therefore $M\left(t^{*}\right)\ge c$. This implies 
\[
V\le\frac{4\sqrt{\inprob\outprob}}{nM\left(t^{*}\right)}\le\frac{4\inprob}{c\num},
\]
where the first step holds by $n'=\num-1\ge\frac{1}{2}\num$ for $\num\ge2$.
Furthermore, there exist some constants $C,C',C''>0$ such that 
\[
\frac{V\left(t^{*}\right)^{2}}{\snr^{2}}\le C''\frac{\frac{4\inprob}{c\num}\left(\frac{\inprob-\outprob}{\inprob}\right)^{2}}{\left(\frac{(\inprob-\outprob)^{2}}{\inprob}\right)^{2}}=C'\frac{\inprob}{\num(\inprob-\outprob)^{2}}\le\frac{C}{\num\snr},
\]
where the first step holds by Facts \ref{fact:CBM_tstar_lesssim_(p-q)/p}
and \ref{fact:CBM_renyi_equivalence}, and the last step holds by
Fact \ref{fact:CBM_renyi_equivalence}. Finally, we have 
\[
t^{*}\sqrt{V}\le C_{0}\frac{\inprob-\outprob}{\inprob}\sqrt{\frac{\inprob}{\num}}\le C_{1}\sqrt{\frac{\snr}{\num}}
\]
for some constants $C_{0},C_{1}>0$, where the first step holds by
Fact~\ref{fact:CBM_tstar_lesssim_(p-q)/p} and the last step holds
by Fact~\ref{fact:CBM_renyi_equivalence}.
\end{proof}
We return to controlling $Q_{2}$. Recalling that $\delta\coloneqq V^{\frac{1}{4}}\snr^{\frac{1}{2}}\left(t^{*}\right)^{-\frac{1}{2}}$
and using Chebyshev's inequality, we have 
\[
\P\left(\left|\frac{1}{n'}\sum_{i\in\left[n'\right]}W_{i}\right|>\delta\right)\le\frac{V}{\delta^{2}}\le\frac{C}{\sqrt{\num\snr}},
\]
for some constant $C>0$, where the second step holds since $\frac{V}{\delta^{2}}=\frac{t^{*}\sqrt{V}}{\snr}\lesssim\frac{1}{\sqrt{\num\snr}}$
by Lemma~\ref{fact:CBM_V_s_tstar}. Therefore, there exist some constants
$\consts>0$ and $c'\in(0,1)$ that only depends on $\consts$ such
that 
\[
Q_{2}=\P\left(\frac{1}{n'}\sum_{i\in\left[n'\right]}W_{i}\in[0,\delta]\right)=\frac{1}{2}\left(1-\P\left(\left|\frac{1}{n'}\sum_{i\in\left[n'\right]}W_{i}\right|>\delta\right)\right)\ge c'
\]
when $\num\snr\ge\consts$ (implied by our assumption $\num\snr\to\infty$).

\subsubsection{Controlling $Q_{1}$}

The last two inequalities of Lemma~\ref{fact:CBM_V_s_tstar} implies
that $t^{*}\delta=\sqrt{\snr t^{*}\sqrt{V}}\lesssim\sqrt{\snr\sqrt{\frac{\snr}{\num}}}\lesssim\snr\left(\frac{1}{\num\snr}\right)^{\frac{1}{4}}$.
Therefore, for some constant $C'>0$ we have 
\[
Q_{1}\ge\exp\left[-C'\left(\frac{1}{\num\snr}\right)^{\frac{1}{4}}n'\snr\right].
\]

\subsubsection{Controlling $Q_{3}$}

By Taylor's theorem, we have $\log(1-\snr)=-\snr-\frac{\snr^{2}}{2(1-u)^{2}}$
for some $u\in[0,\snr]$. This implies that when $\snr\le1-c_{0}$
for some constant $c_{0}\in(0,1)$, we have $\log(1-\snr)\ge-\snr-C_{2}\snr^{2}$
for some constant $C_{2}\in(0,1)$ that only depends on $c_{0}$.
It follows that
\[
Q_{3}\ge\exp\left[-(1+C_{2}\snr)\num'\snr\right].
\]

\subsubsection{Putting together}

Returning to Equation (\ref{eq:CBM_P(S_n' >=00003D 0) lower bound}),
we obtain
\begin{align*}
\P\left(S_{n'}\ge0\right) & \ge c'\exp\left[-(1+C_{2}\snr)\num'\snr\right]\exp\left[-C'\left(\frac{1}{\num\snr}\right)^{\frac{1}{4}}n'\snr\right]\\
 & =\exp\left[-\left(1+C_{2}\snr+C'\left(\frac{1}{\num\snr}\right)^{\frac{1}{4}}+\frac{1}{\num'\snr}\log\frac{1}{c'}\right)\num'\snr\right].
\end{align*}
The desired inequality follows by taking $\xi\coloneqq C_{2}\snr+C'\left(\frac{1}{\num\snr}\right)^{\frac{1}{4}}+\frac{1}{\num'\snr}\log\frac{1}{c'}$
and noting that $\xi=o(1)$ under our assumptions $\snr=o(1)$ and
$\num\snr\to\infty$.

\section{Proof of the first inequality in Theorem \ref{thm:SDP_error}\label{sec:proof_SDP_error_rate}}

Here we prove the first inequality in Theorem \ref{thm:SDP_error}.
The proof of the second inequality is given in Section \ref{sec:proof_cluster_error_rate}.

\subsection{Preliminaries\label{sec:SDP_proof_prelim}}

Recall the definitions given in Section~\ref{sec:preliminary}. We
introduce some additional notations. For a matrix $\M$, we let $\norm[\M]F\coloneqq\sqrt{\sum_{i,j}M_{ij}^{2}}$
denote its Frobenius norm, $\opnorm{\M}$ its spectral norm (the maximum
singular value), and $\norm[\M]{\infty}\coloneqq\max_{i,j}\left|M_{ij}\right|$
its entrywise $\ell_{\infty}$ norm. With another matrix $\G$ of
the same shape as $\M$, we use $\M\geq\G$ to mean that $M_{ij}\geq G_{ij}$
for all $i,j$. 

Let $\U\coloneqq\frac{1}{\sqrt{\num}}\LabelStar$; it can be seen
that $\U\U\t=\frac{1}{\num}\Ystar$ and in particular $\U$ is a singular
vector of $\Ystar$. Define the projections $\PT(\M)\coloneqq\mathbf{U}\mathbf{U}^{\top}\M+\M\mathbf{U}\mathbf{U}^{\top}-\mathbf{U}\mathbf{U}^{\top}\M\mathbf{U}\mathbf{U}^{\top}$
and $\PTperp(\M)\coloneqq\M-\PT(\M)$ for any $\M\in\mathbb{R}^{\num\times\num}$.
Recall that $\Adj$ is the observed matrix from Model \ref{mdl:Z2},
\ref{mdl:CBM} or \ref{mdl:SBM}. For any $\Yhat$ in the sublevel
set $\betterset(\A)$, we introduce the shorthand $\error\coloneqq\norm[\Yhat-\Ystar]1$
for the $\ell_{1}$ error we aim to bound. Define the shifted adjacency
matrix $\Adj^{0}\coloneqq\Adj-\tune\OneMat$, where $\tune$ is defined
in Equation (\ref{eq:def_tune_param}), and the centered adjacency
matrix (or noise matrix) $\Noise\coloneqq\Adj-\E\Adj=\Adj^{0}-\E\Adj^{0}$.
Crucial to our analysis is a ``dual certificate'' $\D$, which is
an $\num\times\num$ diagonal matrix with diagonal matrices
\[
D_{ii}\coloneqq\sum_{j\in\left[\num\right]}\adj_{ij}^{0}\ystar_{ij}=\labelstar_{i}\sum_{j\in\left[\num\right]}\adj_{ij}^{0}\labelstar_{j},\qquad\text{for each \ensuremath{i\in[\num].}}
\]
See \cite{bandeira2015convex} for how $\D$ arises as a candidate
solution to the dual program of the SDPs~(\ref{eq:CBM_Z2_SDP}) and~(\ref{eq:SBM_SDP}),
though we do not rely on the explicit form of the dual program in
the subsequent proof. 

Let us record some facts about any feasible solution $\Y$ to program
(\ref{eq:CBM_Z2_SDP}) or (\ref{eq:SBM_SDP}).
\begin{fact}
\label{fact:pos_neg_entries_in_Y-Ystar}For any $\Y$ feasible to
program (\ref{eq:CBM_Z2_SDP}) or (\ref{eq:SBM_SDP}), we have
\[
\begin{cases}
Y_{ij}-\ystar_{ij}=0, & \text{if }i=j,\\
Y_{ij}-\ystar_{ij}\le0, & \text{if }\ystar_{ij}=1,\\
Y_{ij}-\ystar_{ij}\ge0, & \text{if }\ystar_{ij}=-1.
\end{cases}
\]
\end{fact}
\begin{proof}
Since $Y_{ii}=1$ for $i\in[\num]$ and $\Y\succeq0$, we have $\norm[\Y]{\infty}=1$.
The result follows from the fact that $\Ystar\in\{\pm1\}^{\num\times\num}$.
\end{proof}
\begin{fact}
\label{fact:PTperp}For any $\Y$ feasible to program (\ref{eq:CBM_Z2_SDP})
or (\ref{eq:SBM_SDP}), the following hold.
\begin{enumerate}[label={(\alph*)},ref={\thefact(\alph*)}]
\item  \label{fact:Tr(PTperp(Yhat))}$\PTperp\left(\Y\right)\succeq0$
and $\Tr\left[\PTperp\left(\Y\right)\right]=\frac{\error}{\num}$.
\item \label{fact:Yhp_infty}$\norm[\PTperp(\Y)]{\infty}\leq4$.
\end{enumerate}
\end{fact}
\begin{proof}
For part (a), note that $\PTperp\left(\Y\right)=(\IdMat-\U\U\t)\Y(\IdMat-\U\U\t)$,
which is positive semidefinite since $\Y\succeq0$ by feasibility
to program (\ref{eq:CBM_Z2_SDP}) or (\ref{eq:SBM_SDP}). We also
have
\begin{align*}
\Tr\left[\PTperp\left(\Y\right)\right] & \overset{(i)}{=}\Tr\left[\left(\I-\U\U\t\right)\left(\Y-\Ystar\right)\right]\\
 & \overset{(ii)}{=}\Tr\left[\left(-\U\U\t\right)\left(\Y-\Ystar\right)\right]\\
 & =\frac{1}{\num}\Tr\left[\left(-\Ystar\right)\left(\Y-\Ystar\right)\right]=\frac{\error}{\num},
\end{align*}
where step $(i)$ holds since trace is invariant under cyclic permutations
and the matrix $\I-\U\U\t$ is idempotent, and step $(ii)$ holds
since $\ystar_{ii}-Y_{ii}=0$ for $i\in[\num]$.

For part (b), the definition of $\PTperp(\cdot)$ and direct calculation
give
\[
\norm[\PTperp\left(\Y\right)]{\infty}\le\norm[\Y]{\infty}+\norm[\U\U\t\Y]{\infty}+\norm[\Y\U\U\t]{\infty}+\norm[\U\U\t\Y\U\U\t]{\infty}\le4,
\]
where the last step holds because for all $(i,j)$, $(\U\U\t)_{ij}=\frac{1}{\num}\ystar_{ij}\in[-\frac{1}{\num},\frac{1}{\num}]$
and $Y_{ij}\in[-1,1]$.
\end{proof}

We now proceed to the proof of the first inequality in Theorem \ref{thm:SDP_error}.
Following the strategy outlined in Section~\ref{sec:proof_sketch},
we perform the proof in three steps.

\subsection{Step 1: Basic inequality \label{sec:basic_inequality}}

As our first step, we establish the following critical basic inequality.
\begin{lem}
\label{lem:basic_inequality} Any $\Yhat\in\betterset(\Adj)$ satisfies
the inequality 
\[
0\le\left\langle -\D,\PTperp(\Yhat)\right\rangle +\left\langle \Noise,\PTperp(\Yhat)\right\rangle .
\]
\end{lem}
\noindent We prove this lemma in Section \ref{sec:proof_basic_inequality}.

With the basic inequality, we can reduce the problem of bounding the
error $\frac{1}{\num}\error=\frac{1}{\num}\norm[\Yhat-\Ystar]1$ to
that of studying the two random matrices $\D$ (the dual certificate)
and $\Noise$ (the noise matrix). In particular, recall that the matrix
$\PTperp(\Yhat)$ satisfies $\Tr\big[\PTperp(\Yhat)\big]=\frac{1}{\num}\error$
and the other properties in Fact~\ref{fact:PTperp}. The rest of
the proof relies only on these properties of $\PTperp(\Yhat)$, and
it suffices to study how matrices with such properties interact with
$\D$ and $\Noise$.

Henceforth we use $S_{1}\coloneqq\left\langle -\D,\PTperp(\Yhat)\right\rangle $
and $S_{2}\coloneqq\left\langle \Noise,\PTperp(\Yhat)\right\rangle $
to denote the two terms on the RHS of the basic inequality. As the
next two steps of the proof, we first control $S_{2}$, and then derive
the desired exponential error rate by analyzing the sum $S_{1}+S_{2}$.

\subsection{Step 2: Controlling $S_{2}$\label{sec:S2}}

The proposition below provides a bound on $S_{2}$.
\begin{prop}
\label{prop:S2}Under the conditions of Theorem \ref{thm:SDP_error},
with probability at least $1-\frac{4}{\sqrt{\num}}$, at least one
of the following inequalities holds: 
\begin{align*}
\frac{\error}{\num} & \le\left\lfloor \num\exp\left[-\left(1-\conste\sqrt{\frac{1}{\num\snr}}\right)\proxnum\snr\right]\right\rfloor ,\\
S_{2} & \le\constStwo\frac{\error}{\num}\sqrt{\frac{1}{\num\snr}}\frac{\proxnum\snr}{t^{*}},
\end{align*}
where $\conste>0$ and $\constStwo>0$ are numeric constants.  
\end{prop}
\noindent The proof of this proposition is model-dependent, and is
given in Sections~\ref{sec:proof_prop_Z2_S2}, \ref{sec:proof_CBM_prop_S2}
and \ref{sec:proof_SBM_prop_S2} for Models \ref{mdl:Z2}, \ref{mdl:CBM}
and \ref{mdl:SBM}, respectively.

While technical in its form, Proposition \ref{prop:S2} has a simple
interpretation: either the desired exponential error bound already
holds, or $S_{2}$ is bounded by a small quantity that eventually
dictates the second order term in the error exponent. To proceed,
we may assume that the second bound in Proposition~\ref{prop:S2}
holds. Plugging this bound into the basic inequality in Lemma \ref{lem:basic_inequality},
we obtain that
\begin{equation}
0\le\underbrace{\left\langle -\D,\PTperp(\Yhat)\right\rangle }_{S_{1}}+\constStwo\frac{\error}{\num}\cdot\frac{1}{t^{*}}\sqrt{\frac{1}{\num\snr}}\proxnum\snr.\label{eq:gamma_bound}
\end{equation}

\subsection{Step 3: Analyzing $S_{1}+S_{2}$\label{sec:S1+S2}}

If $\error=0$ then we are done, so we assume $\error>0$ in the sequel.
To show that $\error$ decays exponentially in $\snr$, we need a
simple pilot bound on $\error$ that is polynomial in $\snr$. 
\begin{lem}
\emph{\label{lem:pilot_bound} }Suppose that $\std>0$ for Model~\ref{mdl:Z2},
$\num\obsprob\ge1$ for Model \ref{mdl:CBM}, and $\num\inprob\ge1$
for Model \ref{mdl:SBM}. There exists a constant $\constpilot>0$
such that with probability at least $1-2(e/2)^{-2\num}$,
\[
\error\leq\constpilot\sqrt{\frac{\num^{3}}{\snr}}.
\]
\end{lem}
\noindent The proof is model-dependent, and is given in Sections~\ref{sec:proof_Z2_pilot_bound},
\ref{sec:proof_CBM_pilot_bound} and \ref{sec:proof_SBM_pilot_bound}
for Models \ref{mdl:Z2}, \ref{mdl:CBM} and~\ref{mdl:SBM}, respectively.
We verify that the premise of Lemma~\ref{lem:pilot_bound} is satisfied:
for Model~\ref{mdl:Z2} this is clear; for Models \ref{mdl:CBM}
and \ref{mdl:SBM}, we have $\inprob\ge\frac{(\inprob-\outprob)^{2}}{\inprob}\gtrsim\snr$
thanks to Facts \ref{fact:CBM_renyi_equivalence} and \ref{fact:SBM_renyi_equivalence}
(recall $\inprob:=\obsprob(1-\flipprob)$ in Model \ref{mdl:CBM}),
and $\num\snr\ge\consts$ for some large enough $\consts>0$ under
the premise of Theorem~\ref{thm:SDP_error}. \\

Now recall that the positive semidefinite matrix $\PTperp(\Yhat)$
has non-negative diagonal entries and satisfies $\Tr\left[\PTperp(\Yhat)\right]=\frac{\error}{\num}$
(Fact \ref{fact:Tr(PTperp(Yhat))}). Define the (non-negative) numbers
\[
b_{i}\coloneqq\left(\PTperp\left(\Yhat\right)\right)_{ii},\qquad b_{\max}=4,\qquad\beta\coloneqq\frac{1}{b_{\max}}\sum_{i\in\left[\num\right]}b_{i}=\frac{\error}{b_{\max}\num}
\]
and the random variables $X_{i}\coloneqq-D_{ii}$, which is the $i$-th
diagonal entry of the dual certificate $\D$ defined in Section~\ref{sec:SDP_proof_prelim}.
With the above notations, we have 
\begin{equation}
S_{1}:=\left\langle -\D,\PTperp(\Yhat)\right\rangle =\sum_{i\in[\num]}X_{i}b_{i}=b_{\max}\sum_{i\in\left[\num\right]}X_{i}\left(\frac{b_{i}}{b_{\max}}\right),\label{eq:S1_as_sum}
\end{equation}
where $\frac{b_{i}}{b_{\max}}\in[0,1]$ by Fact \ref{fact:Yhp_infty}. 

To proceed, we shall employ a technique that is reminiscent of the
``order statistics argument'' in \cite{fei2018hidden,fei2019exponential}.
Let $X_{(1)}\ge X_{(2)}\ge\cdots\ge X_{(\num)}$ be the order statistics
of $\left\{ X_{i}\right\} $. Let $C$ be a constant to be chosen
later. For ease of presentation, we define shorthands $\eta\coloneqq C\sqrt{\frac{1}{\num\snr}}$
,
\begin{align*}
\varrho\left(m,c\right) & \coloneqq\frac{1}{t^{*}}\left[(1+\eta)\log\left(\frac{\num e}{m}\right)-(1-c\eta)\proxnum\snr\right],\qquad\text{and}\\
\vartheta\left(c\right) & \coloneqq\exp\left[-(1-c\eta)\proxnum\snr\right].
\end{align*}
Below we consider two cases: $\beta>1$ and $0<\beta\le1$, where
we recall that $\beta\coloneqq\sum_{i\in\left[\num\right]}\left(\frac{b_{i}}{b_{\max}}\right)=\frac{\error}{b_{\max}\num}$.

\subsubsection*{Case 1: $\beta>1$.}

In this case, the expression (\ref{eq:S1_as_sum}) implies that
\[
S_{1}\le b_{\max}\left[\sum_{i\in\left[\left\lfloor \beta\right\rfloor \right]}X_{(i)}+\left(\beta-\left\lfloor \beta\right\rfloor \right)X_{(\left\lceil \beta\right\rceil )}\right].
\]
Combining with Equation (\ref{eq:gamma_bound}), we obtain that
\begin{align*}
0 & \le S_{1}+\constStwo b_{\max}\beta\frac{1}{t^{*}}\sqrt{\frac{1}{\num\snr}}\proxnum\snr\\
 & \le b_{\max}\left[\sum_{i\in\left[\left\lfloor \beta\right\rfloor \right]}\left(X_{(i)}+\constStwo\frac{1}{t^{*}}\sqrt{\frac{1}{\num\snr}}\proxnum\snr\right)+\left(\beta-\left\lfloor \beta\right\rfloor \right)\left(X_{(\left\lceil \beta\right\rceil )}+\constStwo\frac{1}{t^{*}}\sqrt{\frac{1}{\num\snr}}\proxnum\snr\right)\right].
\end{align*}
When $\beta$ is not an integer, the residual term above involving
$\left(\beta-\left\lfloor \beta\right\rfloor \right)$ is cumbersome.
Fortunately, the following simple lemma (proved in Section \ref{sec:proof_lem_make_gamma_integer})
allows us to take the integer part of $\beta$.
\begin{lem}
\label{lem:make_gamma_integer} Suppose that $\beta\in[1,\num]$,
and $\phi_{1}\ge\phi_{2}\ge\ldots\ge\phi_{\num}$ are $\num$ fixed
numbers. Define $V(u)\coloneqq\sum_{i\in[\left\lfloor u\right\rfloor ]}\phi_{i}+(u-\left\lfloor u\right\rfloor )\phi_{\left\lceil u\right\rceil }$.
If $0\le V(\beta)$, then $0\le V(\beta_{0})$ for any $\beta_{0}\in[1,\beta]$.
\end{lem}
Letting $\beta_{0}\coloneqq\left\lfloor \beta\right\rfloor $ and
invoking Lemma \ref{lem:make_gamma_integer}, we deduce from the
last displayed equation that
\begin{align*}
0 & \le b_{\max}\sum_{i\in\left[\beta_{0}\right]}\left(X_{(i)}+\constStwo\frac{1}{t^{*}}\sqrt{\frac{1}{\num\snr}}\proxnum\snr\right)\\
 & =b_{\max}\sum_{i\in\left[\beta_{0}\right]}X_{(i)}+b_{\max}\beta_{0}\constStwo\frac{1}{t^{*}}\sqrt{\frac{1}{\num\snr}}\proxnum\snr\\
 & \le b_{\max}\cdot\max_{\substack{\calM\subset\left[\num\right]\\
\left|\calM\right|=\beta_{0}
}
}\left\{ \sum_{i\in\calM}X_{i}\right\} +b_{\max}\beta_{0}\constStwo\frac{1}{t^{*}}\sqrt{\frac{1}{\num\snr}}\proxnum\snr.
\end{align*}
 The following lemma, proved in Section \ref{sec:proof_subset_sum_ineq},
provides a tight bound on the first sum above.  
\begin{lem}
\label{lem:subset_sum_ineq}Let $C$ be any constant satisfying $C\ge2\sqrt{2}$.
Let $\eta:=C\sqrt{\frac{1}{\num\snr}}$ and $M$ be any positive number
satisfying $1\le M\le\max\left\{ 1,C\sqrt{\frac{\num}{\snr}}\right\} $.
If $\num\snr\ge\consts$ for some constant $\consts\ge4$, then we
have 
\begin{align*}
\max_{\substack{\calM\subset\left[\num\right]\\
\left|\calM\right|=m
}
}\left\{ \sum_{i\in\calM}X_{i}\right\}  & \le\frac{1}{t^{*}}\left((1+\eta)m\log\left(\frac{\num e}{m}\right)-(1-2\eta)m\proxnum\snr\right),\quad\forall m=1,2,\ldots,\left\lfloor M\right\rfloor 
\end{align*}
with probability at least $1-3\exp\left(-\sqrt{\log\num}\right)$. 
\end{lem}
Set $C=\consteta:=\frac{\constpilot}{b_{\max}}$ and $\eta=C\sqrt{\frac{1}{\num\snr}}$.
Note that Lemma \ref{lem:pilot_bound} ensures that $\beta_{0}\le\beta\le\consteta\sqrt{\frac{\num}{\snr}}$
with high probability.\footnote{We assume that $\consteta\sqrt{\frac{\num}{\snr}}\ge1$ here; if this
is not true, we can skip to the proof under case $0<\beta\le1$ that
is presented later. } Therefore, applying Lemma \ref{lem:subset_sum_ineq} with $M=\consteta\sqrt{\frac{\num}{\snr}}$
and the above $C$, we obtain that with high probability,
\begin{align*}
0 & \le\beta_{0}\cdot\varrho\left(\beta_{0},2\right)+\constStwo\frac{1}{t^{*}}\beta_{0}\sqrt{\frac{1}{\num\snr}}\proxnum\snr\\
 & =\beta_{0}\cdot\varrho\left(\beta_{0},2\right)+\beta_{0}\frac{1}{t^{*}}\frac{\constStwo}{\consteta}\eta\proxnum\snr\\
 & =\beta_{0}\cdot\varrho\left(\beta_{0},2+\frac{\constStwo}{\consteta}\right),
\end{align*}
which implies $0\le\varrho\left(\beta_{0},2+\frac{\constStwo}{\consteta}\right)$.
Rearranging this inequality using the definition of $\varrho$, we
obtain 
\[
\beta_{0}\le e\num\cdot\exp\left[-\frac{1-(2+\constStwo/\consteta)\eta}{1+\eta}\proxnum\snr\right].
\]
To simplify the last RHS, we use the following elementary lemma.
\begin{lem}
\label{lem:explicit_const_helper}If $0\le\eta\le\frac{1}{C_{0}+1}$
for some $C_{0}>0$, then $\frac{1-C_{0}\eta}{1+\eta}\ge1-(C_{0}+1)\eta\ge0$.
\end{lem}
\begin{proof}
We have $\frac{1-C_{0}\eta}{1+\eta}=1-\frac{(1+C_{0})\eta}{1+\eta}\overset{(a)}{\ge}1-(1+C_{0})\eta\overset{(b)}{\ge}0,$
where step $(a)$ holds since $\eta\ge0$ and step $(b)$ holds since
$\eta\le\frac{1}{C_{0}+1}$.
\end{proof}
The premise of Theorem~\ref{thm:SDP_error}, i.e., $\num\snr\ge\consts$
for $\consts$ sufficiently large, implies that $\eta\le\frac{1}{C_{0}+1}$
for $C_{0}:=2+\frac{\constStwo}{\consteta}$. Applying the above lemma
gives $\frac{1-C_{0}\eta}{1+\eta}\ge1-(C_{0}+1)\eta$. Combining with
the last displayed equation, we obtain that
\begin{align*}
\beta_{0} & \le e\num\cdot\vartheta\Big(1+C_{0}\Big)\implies\beta_{0}\le\left\lfloor e\num\cdot\vartheta\Big(1+C_{0}\Big)\right\rfloor 
\end{align*}
since $\beta_{0}$ is an integer. Because $\frac{\error}{\num}=b_{\max}\cdot\beta\le b_{\max}\cdot2\beta_{0}$,
it follows that
\[
\frac{\error}{\num}\le2b_{\max}\left\lfloor e\num\cdot\vartheta\Big(1+C_{0}\Big)\right\rfloor \le\left\lfloor 2b_{\max}e\num\cdot\vartheta\Big(1+C_{0}\Big)\right\rfloor ,
\]
where in the last step we use the fact that $c\left\lfloor x\right\rfloor \le\left\lfloor cx\right\rfloor $
for any real number $x\ge0$ and integer $c>0$.\footnote{Proof: we have $\left\lfloor cx\right\rfloor =\left\lfloor c\left\lfloor x\right\rfloor +(cx-c\left\lfloor x\right\rfloor )\right\rfloor \ge\left\lfloor c\left\lfloor x\right\rfloor \right\rfloor =c\left\lfloor x\right\rfloor $
by noting that $cx-c\left\lfloor x\right\rfloor \ge0$ and $c\left\lfloor x\right\rfloor $
is an integer.} As long as $\num\snr\ge1$, we have $\frac{1}{\num\snr}\le\sqrt{\frac{1}{\num\snr}}$
and hence
\[
\frac{\error}{\num}\le\left\lfloor \num\cdot\exp\Big(\log(2b_{\max}e)\Big)\cdot\vartheta\Big(1+C_{0}\Big)\right\rfloor \le\left\lfloor \num\cdot\vartheta\left(1+C_{0}+\frac{1}{\consteta}\log(2b_{\max}e)\right)\right\rfloor .
\]
Recalling the definition of $\vartheta$, we see that we have proved
the first inequality in Theorem~\ref{thm:SDP_error}. 

\subsubsection*{Case 2: $0<\beta\le1$.}

In this case, continuing from the expression (\ref{eq:S1_as_sum}),
we have 
\[
S_{1}\le b_{\max}\beta\cdot X_{(1)}\le b_{\max}\beta\cdot\varrho\left(1,2\right),
\]
where in the last step we apply Lemma \ref{lem:subset_sum_ineq} with
$m=M=1$, $C=2\sqrt{2}$ and $\eta=C\sqrt{\frac{1}{\num\snr}}$. Combining
with Equation (\ref{eq:gamma_bound}), we obtain that
\begin{align*}
0 & \le b_{\max}\beta\cdot\varrho\left(1,2\right)+\constStwo\frac{\error}{\num}\frac{1}{t^{*}}\sqrt{\frac{1}{\num\snr}}\proxnum\snr\\
 & =b_{\max}\beta\cdot\varrho\left(1,2\right)+b_{\max}\beta\frac{1}{t^{*}}\frac{\constStwo}{2\sqrt{2}}\eta\proxnum\snr\\
 & =b_{\max}\beta\cdot\varrho\left(1,2+\frac{\constStwo}{2\sqrt{2}}\right),
\end{align*}
which implies $\varrho\left(1,2+\frac{\constStwo}{2\sqrt{2}}\right)\ge0$.
Rearranging this inequality using the definition of $\varrho$, we
obtain 
\[
1\le e\num\cdot\exp\left[-\frac{1-(2+\constStwo/2\sqrt{2})\eta}{1+\eta}\proxnum\snr\right].
\]
Applying Lemma \ref{lem:explicit_const_helper} with $C_{0}=2+\frac{\constStwo}{2\sqrt{2}}$
gives $\frac{1-C_{0}\eta}{1+\eta}\ge1-(C_{0}+1)\eta$. Combining with
the last displayed equation, we obtain
\[
1\le e\num\cdot\vartheta\left(C_{0}+1\right)\implies1\le\left\lfloor e\num\cdot\vartheta\left(C_{0}+1\right)\right\rfloor .
\]
But we have $\frac{\error}{b_{\max}\num}=\beta\le1$ by the case assumption.
It follows that 
\[
\frac{\error}{\num}\le b_{\max}\left\lfloor e\num\cdot\vartheta\left(C_{0}+1\right)\right\rfloor \le\left\lfloor b_{\max}e\num\cdot\vartheta\left(C_{0}+1\right)\right\rfloor ,
\]
where in the last step we use the fact that $c\left\lfloor x\right\rfloor \le\left\lfloor cx\right\rfloor $
for any real $x\ge0$ and integer $c>0$. As long as $\num\snr\ge1$,
we have  $\frac{1}{\num\snr}\le\sqrt{\frac{1}{\num\snr}}$ and hence
\[
\frac{\error}{\num}\le\left\lfloor \num\cdot\exp\Big(\log(b_{\max}e)\Big)\cdot\vartheta\Big(1+C_{0}\Big)\right\rfloor \le\left\lfloor \num\cdot\vartheta\left(1+C_{0}+\frac{1}{2\sqrt{2}}\log(b_{\max}e)\right)\right\rfloor .
\]
Recalling the definition of $\vartheta$, we see that we have proved
the first inequality in Theorem~\ref{thm:SDP_error}.

\section{Proofs of Technical Lemmas in Section~\ref{sec:proof_SDP_error_rate}}

\subsection{Proof of Lemma \ref{lem:basic_inequality}\label{sec:proof_basic_inequality}}

Recall the matrices $\Noise$ and $\D$ defined in Section \ref{sec:SDP_proof_prelim}.
Let $\d:=(D_{11},\ldots,D_{\num\num})\t\in\real^{\num}$ be the vector
of diagonal entries of $\D$. Note that $\D\LabelStar=\LabelStar\circ\d=\LabelStar\circ\left(\LabelStar\circ(\Adj^{0}\LabelStar)\right)=\Adj^{0}\LabelStar,$
where $\circ$ denotes element-wise product. Therefore, we have the
identity
\[
\D\Ystar=\D\LabelStar\left(\LabelStar\right)\t=\Adj^{0}\LabelStar\left(\LabelStar\right)\t=\Adj^{0}\Ystar.
\]

To prove the basic inequality, let us fix an arbitrary $\Yhat\in\betterset(\A)$
and observe that $0\le\left\langle \Adj,\Yhat-\Ystar\right\rangle $.
On the other hand, we have 
\begin{align*}
\left\langle \Adj,\Yhat-\Ystar\right\rangle  & =\left\langle \Adj^{0}-\D,\Yhat-\Ystar\right\rangle 
\end{align*}
thanks to the following facts: ($i$) $\Yhat-\Ystar$ has zero diagonal
and $\D$ is a diagonal matrix; ($ii$) for Models~\ref{mdl:Z2}
and~\ref{mdl:CBM} we have $\Adj^{0}=\Adj$; ($iii$) for Model~\ref{mdl:SBM}
we have $\Adj^{0}=\Adj-\tune\OneMat$ but the program (\ref{eq:SBM_SDP})
used for this model ensures that $\left\langle \OneMat,\Yhat\right\rangle =\left\langle \OneMat,\Ystar\right\rangle =0$.
Using the equality $\D\Ystar=\Adj^{0}\Ystar$ proved above, we obtain
that
\begin{align*}
\left\langle \Adj,\Yhat-\Ystar\right\rangle = & \left\langle \Adj^{0}-\D,\Yhat\right\rangle =\left\langle \Adj^{0}-\D,\PTperp(\Yhat)\right\rangle +\left\langle \Adj^{0}-\D,\PT(\Yhat)\right\rangle .
\end{align*}
By definition of $\PT$, we can write $\PT(\Yhat)=\LabelStar\u\t+\v\left(\LabelStar\right)\t$
for some $\u,\v\in\real^{\num}$, hence 
\[
\left\langle \Adj^{0}-\D,\PT(\Yhat)\right\rangle =\left\langle \Adj^{0}-\D,\LabelStar\u\t+\v\left(\LabelStar\right)\t\right\rangle =0
\]
because $\D\LabelStar=\Adj^{0}\LabelStar$. It follows that
\begin{align*}
\left\langle \Adj,\Yhat-\Ystar\right\rangle  & =\left\langle \Adj^{0}-\D,\PTperp(\Yhat)\right\rangle =\left\langle (\E\Adj-\tune\OneMat)+\Noise-\D,\PTperp(\Yhat)\right\rangle .
\end{align*}
We shall prove later that
\begin{equation}
\left\langle \E\Adj-\tune\OneMat,\PTperp(\Yhat)\right\rangle =0.\label{eq:SBM_EA-lambdaJ}
\end{equation}
Taking this identity as given, we obtain $0\le\left\langle \Adj,\Yhat-\Ystar\right\rangle \le\left\langle \Noise-\D,\PTperp(\Yhat)\right\rangle $,
thereby completing the proof of Lemma \ref{lem:basic_inequality}.
\begin{proof}[Proof of inequality (\ref{eq:SBM_EA-lambdaJ})]
 Under Models \ref{mdl:Z2} and \ref{mdl:CBM}, we have $\E\Adj=c\Ystar$
for some scalar $c$ as well as  $\tune=0$ as chosen in (\ref{eq:def_tune_param}),
hence 
\[
\left\langle \E\Adj-\tune\OneMat,\PTperp(\Yhat)\right\rangle =c\left\langle \PTperp(\Ystar),\Yhat\right\rangle =0
\]
as desired. Under Model \ref{mdl:SBM}, we have the equalities
\begin{align*}
\left\langle \E\Adj-\tune\OneMat,\PTperp\left(\Y\right)\right\rangle  & \overset{(a)}{=}\left\langle \PTperp\left(\frac{\inprob-\outprob}{2}\Ystar+\frac{\inprob+\outprob}{2}\OneMat-\tune\OneMat\right),\Yhat\right\rangle \\
 & \overset{(b)}{=}\left\langle \frac{\inprob+\outprob}{2}\OneMat-\tune\OneMat,\Yhat\right\rangle \overset{(c)}{=}0,
\end{align*}
where step $(a)$ holds since $\E\Adj=\frac{\inprob-\outprob}{2}\Ystar+\frac{\inprob+\outprob}{2}\OneMat$
and the projection $\PTperp$ is self-adjoint, step $(b)$ holds since
$\PTperp(\Ystar)=0$ and $\PTperp(\OneMat)=\OneMat$ (because $(\IdMat-\U\U\t)\OneMat=(\IdMat-\num^{-1}\Ystar)\OneMat=\OneMat$),
and step $(c)$ holds since $\left\langle \OneMat,\Yhat\right\rangle =0$
by feasibility of the matrix $\Yhat\in\betterset(\A)$ to the program~(\ref{eq:SBM_SDP}).
\end{proof}

\subsection{Proof of Proposition \ref{prop:S2} for Model \ref{mdl:Z2} (Z2)
\label{sec:proof_prop_Z2_S2}}

We shall make use of the following matrix inequality: 
\begin{equation}
\left\langle \G,\M\right\rangle \le\opnorm{\G}\Tr(\M),\quad\forall\G,\forall\M\succeq0,\label{eq:holder}
\end{equation}
which is due to the fact that for $\M\succeq0$, $\Tr(\M)$ equals
the sum of its singular values. We also need the following spectral
norm bound, which is from a direct application of~\cite[Theorem 2.11]{davidson2001local}. 
\begin{lem}
\label{lm:Z2_spectral_bound_W} We have $\opnorm{\Noise}\leq(2+\sqrt{2})\sqrt{\std^{2}\num}$
with probability at least $1-e^{-\num/2}$. 
\end{lem}
Turning to bounding $S_{2}$, we have with probability at least $1-e^{-\num/2}$,
\begin{align*}
S_{2}:=\left\langle \Noise,\PTperp(\Yhat)\right\rangle  & \overset{(a)}{\le}\opnorm{\Noise}\cdot\Tr\left[\PTperp(\Yhat)\right]\\
 & \overset{(b)}{\le}4\std\sqrt{\num}\cdot\frac{\error}{\num}\\
 & \overset{(c)}{=}4\sqrt{2}\frac{\error}{\num}\sqrt{\frac{1}{\num\snr}}\frac{\num\snr}{t^{*}},
\end{align*}
where step $(a)$ follows from noting that $\PTperp(\Yhat)\succeq0$
by Fact~\ref{fact:Tr(PTperp(Yhat))} and applying inequality~(\ref{eq:holder}),
step $(b)$ holds by Lemma~\ref{lm:Z2_spectral_bound_W} and Fact~\ref{fact:Tr(PTperp(Yhat))},
and step $(c)$ holds by definitions of $\snr$ in Equation~(\ref{eq:snr})
and $t^{*}$ in Equation~(\ref{eq:def_tstar}). Setting $\constStwo=4\sqrt{2}$
completes the proof of Proposition~\ref{prop:S2} for Z2.

\subsection{Proof of Proposition \ref{prop:S2} for Mode \ref{mdl:CBM} (CBM)
\label{sec:proof_CBM_prop_S2}}

Recall that $S_{2}:=\left\langle \PTperp(\Yhat),\Noise\right\rangle .$
We control the right hand side by splitting it into two parts, one
involving a trimmed version of $\W$ and the other the residual. This
technique is similar to those in \cite{fei2019exponential,zhang2017theoretical},
but here we use it for a different model.

\paragraph{Trimming.}

We consider an equivalent way of generating $\Adj$ under Mode \ref{mdl:CBM}
(CBM). Define a symmetric random matrix $\G\in\left\{ \pm1\right\} ^{\num\times\num}$
such that $G_{ii}=0$ for $i\in\left[\num\right]$ and $\left\{ G_{ij}:i<j\right\} $
are generated independently as
\[
G_{ij}=\begin{cases}
\ystar_{ij}, & \text{w.p.\ }1-\flipprob,\\
-\ystar_{ij}, & \text{w.p.\ }\flipprob.
\end{cases}
\]
The observed matrix $\Adj$ from Model \ref{mdl:CBM} can be equivalently
generated by
\[
\adj_{ij}=\begin{cases}
G_{ij}, & \text{w.p.\ }\obsprob,\\
0 & \text{w.p.\ }1-\obsprob,
\end{cases}\qquad\text{independently for \ensuremath{i<j}},
\]
with $\adj_{ii}=0$ for $i\in[\num]$.\footnote{As mentioned in Section~\ref{sec:setup_model}, it is inconsequential
to change the diagonal entries of $\A$.} We introduce a few additional notations. For a vector $\v\in\real^{\num}$,
we let $\norm[\v]0$ denote the number of nonzero entries in $\v$.
For a matrix $\M\in\real^{\num\times\num}$, we let $\M^{\text{up}}$
be obtained from $\M$ by zeroing out its lower triangular entries,
$\M_{i,:}$ and $\M_{:,j}$ be the $i$-th row and $j$-th column
of $\M$ respectively, and we define the trimmed matrix $\widetilde{\M}\coloneqq\left(M_{ij}\indic\{\norm[\M_{i,:}]0\le40\obsprob\num,\norm[\M_{:,j}]0\le40\obsprob\num\}\right)_{i,j\in[\num]}$.

With the above notations, we note that $\Adj^{\text{up}}$ and $\G^{\text{up}}$
are both matrices with independent entries. We first record a series
of lemmas that are useful for our proof to follow.
\begin{lem}
\label{lm:CBM_trimmed_spectral_bound}For some absolute constant
$C>0$, we have 
\[
\P\left(\opnorm{\widetilde{\Adj^{\text{up}}}-\obsprob\G^{\text{up}}}\ge C\sqrt{\obsprob\num}\right)\le\frac{1}{\num^{3}}.
\]
\end{lem}
\begin{proof}
Note that $\norm[\G^{\text{up}}]{\infty}\le1$ surely. Applying \cite[Lemma 3.2]{keshavan2009matrix}
with $m$ and $\epsilon$ therein set to $\num$ and $\obsprob\num$,
we obtain $\P(\opnorm{\widetilde{\Adj^{\text{up}}}-\obsprob\G^{\text{up}}}\ge C\sqrt{\obsprob\num}\mid\G^{\text{up}})\le\frac{1}{\num^{3}}$
for any $\G^{\text{up}}$.\footnote{\cite[Lemma 3.2]{keshavan2009matrix} involves trimming rows and columns
that contain more than $2\obsprob\num$ nonzero entries. A closer
inspection of their proof reveals that their bound still applies to
our setting, albeit with a possibly larger constant $C$. } Integrating out $\G^{\text{up}}$ yields the result.  
\end{proof}
\begin{lem}
\label{lem:subgaussian_matrix_spectral_bound}Let $\M\in\real^{\num\times\num}$
be a random matrix whose entries $M_{ij}$ are independent mean-zero
random variables with $\left|M_{ij}\right|\le C'$ for some constant
$C'\ge0$. Then there exists a constant $C>0$ such that with probability
at least $1-2e^{-\num}$, we have $\opnorm{\M}\le C\sqrt{\num}.$
\end{lem}
\begin{proof}
Such a result is standard. For example, it follows as a corollary
of Theorem 4.4.5 in \cite{vershynin2018HDP} with $m=\num$, $t=\sqrt{\num}$
and $K\le C''$ for some constant $C'$ since $\norm[\M]{\infty}\le C'$. 
\end{proof}
\begin{lem}
\label{lm:untrimmed_L1_bound}Let $\M\in\{0,1\}^{\num\times\num}$
be a binary matrix with $M_{ii}=0$ for all $i\in[\num]$, and $\{M_{ij}\}_{i,j\in[\num]}$
being independent Bernoulli random variables. Let $\inprob'\coloneqq\max_{ij}\E M_{ij}$.
Define $\calT\coloneqq\{i\in[\num]:\sum_{j}M_{ij}\ge40\num\inprob'\}$,
$Z_{i}\coloneqq\sum_{j}\left|M_{ij}-\E M_{ij}\right|\indic\{i\in\calT\}$
and $Z'_{i}\coloneqq\sum_{j}M_{ij}\indic\{i\in\calT\}$. If $\inprob'\ge\frac{C}{\num}$
for a sufficiently large positive constant $C$, then with probability
at least $1-\frac{1}{\sqrt{\num}}$, we have 
\[
\sum_{i}Z_{i}\le2\sum_{i}Z'_{i}\le40\num^{2}\inprob'e^{-5\num\inprob'}
\]
\end{lem}
\begin{proof}
Define the event $B:=\left\{ \sum_{i}Z_{i}>40\num^{2}\inprob'e^{-5\num\inprob'}\right\} $.
First consider the case where $\num\inprob'\ge\frac{1}{10}\log\num$.
Applying Lemma C.5 in \cite{zhang2017theoretical} with $\inprob$
therein set to $2\inprob'$,\footnote{Inspecting their proof, we see that their bound holds without change
for matrices with independent entries. } we obtain that $\P\left\{ B\right\} \le e^{-10\num\inprob'}$. Under
the case $\num\inprob'\ge\frac{1}{10}\log\num$, this probability
is at most $e^{-\log\num}\le\frac{1}{\sqrt{\num}}$ as claimed. Next
consider the case where $C\le\num\inprob'<\frac{1}{10}\log\num$.
Since $\sum_{j}M_{ij}\indic\{i\in\calT\}\ge40\num\inprob'\indic\{i\in\calT\}$
by definition of $\calT$, we have 
\begin{align*}
\sum_{i}Z_{i} & \le\sum_{i}\sum_{j}M_{ij}\indic\{i\in\calT\}+\num\inprob'\sum_{i}\indic\{i\in\calT\}\\
 & \le2\sum_{i}\sum_{j}M_{ij}\indic\{i\in\calT\}=2\sum_{i}Z'_{i}.
\end{align*}
Set $\varepsilon=20\num\inprob'e^{-5\num\inprob'}$; note that $\varepsilon\in(0,1/2]$
and $21\num\inprob'+2\log\varepsilon^{-1}\le40\num\inprob'$ since
$\inprob'\ge\frac{C}{\num}$. Applying Lemma 8.1 in \cite{rebrova2016local}
with the above $\varepsilon$,\footnote{Inspecting their proof, we see that their bound holds without change
when the means of the Bernoulli are upper bounded by the same $\inprob'$.} we obtain that 
\begin{align*}
\P\left\{ \sum_{i}Z'_{i}>20\num^{2}\inprob'e^{-5\num\inprob'}\right\}  & \le\exp\left(-10\num^{2}\inprob'e^{-5\num\inprob'}\right)\\
 & \le\exp\left(-10C\num e^{-\frac{1}{2}\log\num}\right)\\
 & =\exp\left(-10C\sqrt{\num}\right)\le\frac{1}{\sqrt{\num}}.
\end{align*}
Combining the last two displayed equations proves that $\P\left\{ B\right\} \le\frac{1}{\sqrt{\num}}$
as claimed.
\end{proof}
We are now ready to bound $S_{2}$. Observe that 
\begin{align*}
S_{2} & =2\left\langle \PTperp(\Yhat),\Noise^{\text{up}}\right\rangle \\
 & =2\left\langle \PTperp(\Yhat),\left(\widetilde{\Adj^{\text{up}}}-\E\left[\Adj^{\text{up}}\mid\G\right]\right)+\left(\E\left[\Adj^{\text{up}}\mid\G\right]-\E\Adj^{\text{up}}\right)+\left(\Adj^{\text{up}}-\widetilde{\Adj^{\text{up}}}\right)\right\rangle \\
 & =2\left\langle \PTperp(\Yhat),\left(\widetilde{\Adj^{\text{up}}}-\obsprob\G^{\text{up}}\right)+\left(\obsprob\G^{\text{up}}-\obsprob\E\G^{\text{up}}\right)+\left(\Adj^{\text{up}}-\widetilde{\Adj^{\text{up}}}\right)\right\rangle \\
 & \overset{(a)}{\le}2\Tr\left[\PTperp(\Yhat)\right]\cdot\opnorm{\widetilde{\Adj^{\text{up}}}-\obsprob\G^{\text{up}}}+2\obsprob\Tr\left[\PTperp(\Yhat)\right]\cdot\opnorm{\G^{\text{up}}-\E\G^{\text{up}}}+2\norm[\PTperp(\Yhat)]{\infty}\norm[\Adj^{\text{up}}-\widetilde{\Adj^{\text{up}}}]1\\
 & \overset{(b)}{\le}2\frac{\error}{\num}\opnorm{\widetilde{\Adj^{\text{up}}}-\obsprob\G^{\text{up}}}+2\obsprob\frac{\error}{\num}\opnorm{\G^{\text{up}}-\E\G^{\text{up}}}+8\norm[\Adj^{\text{up}}-\widetilde{\Adj^{\text{up}}}]1,
\end{align*}
where step $(a)$ follows noting that $\PTperp(\Yhat)\succeq0$ (by
Fact \ref{fact:Tr(PTperp(Yhat))}) and applying inequality~(\ref{eq:holder}),
and step $(b)$ holds by Fact \ref{fact:Tr(PTperp(Yhat))} and Fact
\ref{fact:Yhp_infty}. 

We then apply Lemma \ref{lm:CBM_trimmed_spectral_bound} to bound
$\opnorm{\widetilde{\Adj^{\text{up}}}-\obsprob\G^{\text{up}}}$, Lemma
\ref{lem:subgaussian_matrix_spectral_bound} to bound $\opnorm{\G^{\text{up}}-\E\G^{\text{up}}}$,
and Lemma \ref{lm:untrimmed_L1_bound} to $\Adj^{\text{up}}$ and
$(\Adj^{\text{up}})^{\top}$ to bound $\norm[\Adj^{\text{up}}-\widetilde{\Adj^{\text{up}}}]1$
(we do so by setting $M_{ij}=|\adj_{ij}^{\text{up}}|$ for $i,j\in[\num]$
and noting that $\{|\adj_{ij}^{\text{up}}|\}$ are independent Bernoulli
random variables with means $u_{ij}\le\obsprob$). Note that the assumption
$\obsprob\ge\frac{C}{\num}$ of Lemma \ref{lm:untrimmed_L1_bound}
is satisfied by the assumption of this proposition that $\num\snr\ge\consts$
for some large enough $\consts>0$ (since Fact \ref{fact:CBM_renyi_equivalence}
implies $\snr\lesssim\frac{(\inprob-\outprob)^{2}}{\inprob}\le\inprob\le\obsprob$).
It follows that with probability at least $1-\frac{4}{\sqrt{\num}}$,
there holds 
\begin{align*}
S_{2} & \le C_{0}\frac{\error}{\num}\sqrt{\num\obsprob}+C_{1}\num^{2}\obsprob e^{-5\num\obsprob}\eqqcolon C_{0}Q_{1}+C_{1}Q_{2}
\end{align*}
for some constants $C_{0},C_{1}>0$.  It remains to bound $Q_{1}$
and $Q_{2}$ above. 

For $Q_{1}$, Fact \ref{fact:CBM_renyi_equivalence} implies that
$\sqrt{\snr}\le C'\frac{\inprob-\outprob}{\sqrt{\inprob}}$ for some
constant $C'>0$ and therefore
\[
Q_{1}=\error\frac{\inprob-\outprob}{\inprob-\outprob}\sqrt{\frac{\inprob}{\num}}\le\frac{1}{C'}\error(\inprob-\outprob)\sqrt{\frac{1}{\num\snr}}.
\]
Bounding $Q_{2}$ involves some elementary manipulation.

\paragraph{Controlling $Q_{2}$.}

We record an elementary inequality.
\begin{lem}
\label{lm:CBM_simple_ineq}There exists a constant $\consts\ge1$
such that if $\num\snr\ge\consts$, then $\inprob e^{-5\inprob\num}\leq(\inprob-\outprob)e^{-5\num\snr/2}.$
\end{lem}
\begin{proof}
Note that $\inprob\num\ge(\inprob-\outprob)\num\geq\frac{(\inprob-\outprob)^{2}}{\inprob}\num\ge\frac{1}{C'}\num\snr\geq\frac{1}{C'}\consts$
for some constant $C'>0$ by Fact~\ref{fact:CBM_renyi_equivalence}.
As long as $\consts$ is sufficiently large, we have $\frac{\inprob\num}{2}\leq e^{5\inprob\num/2}$.
These inequalities imply that 
\[
\frac{\inprob}{\inprob-\outprob}\leq\frac{\inprob\num}{2}\leq e^{5\inprob\num/2}\leq e^{5(2\inprob-\snr)\num/2}.
\]
Multiplying both sides by $(\inprob-\outprob)e^{-5\inprob\num}$ yields
the claimed inequality.
\end{proof}
Equipped with the above bound, we are ready to bound $Q_{2}$. Let
$\kappa\coloneqq e^{-(1-\xi)\num\snr}$ where $\xi\coloneqq\conste\sqrt{\frac{1}{\num\snr}}$
for some constant $\conste>0$ such that $\left\lfloor \num\kappa\right\rfloor >0$.
If $\frac{\error}{\num}=0$ or $\frac{\error}{\num}\le\left\lfloor \num\kappa\right\rfloor $,
then the first inequality in Proposition \ref{prop:S2} holds and
we are done. It remains to consider the case $\frac{\error}{\num}>\left\lfloor \num\kappa\right\rfloor >0$.
We have that $\left\lfloor \num\kappa\right\rfloor $ is a positive
integer and $\error>\num\left\lfloor \num\kappa\right\rfloor \ge\frac{1}{2}\num^{2}\kappa$.
Hence, 
\begin{align*}
Q_{2}=\obsprob\num^{2}e^{-5\obsprob\num} & \overset{(a)}{\le}2\inprob\num^{2}e^{-5\inprob\num}\\
 & \overset{(b)}{\le}2(\inprob-\outprob)\num^{2}e^{-5\num\snr/2}\\
 & \le2(\inprob-\outprob)e^{-\num\snr}\cdot\num^{2}e^{-\num\snr}\\
 & \overset{(c)}{\le}2(\inprob-\outprob)e^{-\num\snr}\cdot\num^{2}\kappa\\
 & \le4(\inprob-\outprob)e^{-\num\snr}\cdot\error,
\end{align*}
where step $(a)$ holds since $\obsprob=\frac{1}{1-\flipprob}\inprob$
and $\flipprob\in[0,\frac{1}{2}]$ imply $\obsprob\in[\inprob,2\inprob]$,
step $(b)$ holds by Lemma~\ref{lm:CBM_simple_ineq}, and step $(c)$
holds by definition of $\kappa$. Choosing $\consts>0$ large enough
so that $e^{-\num\snr/2}\le\sqrt{\frac{1}{\num\snr}}$, we have $Q_{2}\le4\error(\inprob-\outprob)\sqrt{\frac{1}{\num\snr}}$.

\paragraph{Putting together.}

Combining the above bounds for $Q_{1}$ and $Q_{2}$, we obtain that
\[
S_{2}\le C_{2}\error(\inprob-\outprob)\sqrt{\frac{1}{\num\snr}}=C_{2}\frac{\error}{\num}\sqrt{\frac{1}{\num\snr}}\num(\inprob-\outprob)
\]
for some constant $C_{2}>0$. Under the assumption $0<\outprob\le\inprob\le1$,
we have $\inprob-\outprob\le C'\frac{\snr}{t^{*}}\cdot\frac{1-\flipprob}{\flipprob}$
for some constant $C'>0$ by Facts~\ref{fact:CBM_tstar_lesssim_(p-q)/p}
and~\ref{fact:CBM_renyi_equivalence}. This completes the proof of
Proposition \ref{prop:S2} for CBM.

\subsection{Proof of Proposition \ref{prop:S2} for Model \ref{mdl:SBM} (SBM)
\label{sec:proof_SBM_prop_S2}}

Similarly to the proof for Model \ref{mdl:CBM}, we control the RHS
of $S_{2}=\left\langle \PTperp(\Yhat),\Noise\right\rangle $ by splitting
it into two parts, one involving a trimmed version of $\W$ and the
other the residual. This technique is similar to those in \cite{fei2019exponential,zhang2017theoretical},
but here we provide somewhat tighter bounds.

\paragraph{Trimming. }

We record a technical lemma concerning a trimmed Bernoulli matrix
\begin{lem}
\label{lm:SBM_trimmed_spectral_bound}Suppose $\M\in\real^{\num\times\num}$
is a random matrix with zero on the diagonal and independent entries
$\{M_{ij}\}$ with the following distribution: $M_{ij}=1-p_{ij}$
with probability $p_{ij}$, and $M_{ij}=-p_{ij}$ with probability
$1-p_{ij}$. Let $\inprob'\coloneqq\max_{ij}\inprob_{ij}$, and let
$\widetilde{\M}$ be the matrix obtained from $\M$ by zeroing out
all the rows and columns having more than $40\num\inprob'$ positive
entries. Then there exists some constant $C>0$ such that with probability
at least $1-\frac{1}{\num^{2}}$, 
\[
\opnorm{\widetilde{\M}}\le C\sqrt{\num\inprob'}
\]
\end{lem}
\begin{proof}
The claim follows from \cite[Lemma 9]{fei2019exponential} with $\sigma^{2}$
therein set to $\inprob'$. 
\end{proof}
Let $\Noise^{\text{up}}$ be obtained from $\Noise$ by zeroing out
its lower triangular entries. To bound $S_{2}$, we observe that 
\begin{align*}
S_{2} & =2\left\langle \PTperp(\Yhat),\Noise^{\text{up}}\right\rangle \\
 & =2\left\langle \PTperp(\Yhat),\widetilde{\Noise^{\text{up}}}\right\rangle +2\left\langle \PTperp(\Yhat),\Noise^{\text{up}}-\widetilde{\Noise^{\text{up}}}\right\rangle \\
 & \overset{(a)}{\le}2\Tr\left[\PTperp(\Yhat)\right]\cdot\opnorm{\widetilde{\Noise^{\text{up}}}}+2\norm[\PTperp(\Yhat)]{\infty}\norm[\Noise^{\text{up}}-\widetilde{\Noise^{\text{up}}}]1\\
 & \overset{(b)}{\le}2\frac{\error}{\num}\opnorm{\widetilde{\Noise^{\text{up}}}}+8\norm[\Noise^{\text{up}}-\widetilde{\Noise^{\text{up}}}]1,
\end{align*}
where step $(a)$ follows from noting that $\PTperp(\Yhat)\succeq0$
(by Fact \ref{fact:Tr(PTperp(Yhat))}) and applying inequality~(\ref{eq:holder}),
and step $(b)$ holds by Facts~\ref{fact:Tr(PTperp(Yhat))} and~\ref{fact:Yhp_infty}.
We then apply Lemma \ref{lm:SBM_trimmed_spectral_bound} to $\widetilde{\Noise^{\text{up}}}$
to bound $\opnorm{\widetilde{\Noise^{\text{up}}}}$, and apply Lemma
\ref{lm:untrimmed_L1_bound} to $\Noise^{\text{up}}$ and $(\Noise^{\text{up}})^{\top}$
to bound $\norm[\Noise^{\text{up}}-\widetilde{\Noise^{\text{up}}}]1$.\footnote{Here, we assume that $\Adj$ has zero diagonal and therefore $\Noise$
also has zero diagonal. This assumption is inconsequential to our
proof as mentioned in Section \ref{sec:setup_model}.} Note that the assumption $\inprob'\ge\frac{C}{\num}$ of Lemma \ref{lm:untrimmed_L1_bound}
is satisfied by the assumption of this proposition that $\num\snr\ge\consts$
for some large enough $\consts>0$ (since Fact \ref{fact:SBM_renyi_equivalence}
implies $\snr\lesssim\frac{(\inprob-\outprob)^{2}}{\inprob}\le\inprob$).
We conclude that with probability at least $1-\frac{1}{\num^{2}}-\frac{2}{\sqrt{\num}}$,
\begin{align*}
S_{2} & \le C_{0}\frac{\error}{\num}\sqrt{\num\inprob}+C_{1}\num^{2}\inprob e^{-5\num\inprob}\eqqcolon C_{0}Q_{1}+C_{1}Q_{2}
\end{align*}
for some constants $C_{0},C_{1}>0$. It remains to control $Q_{1}$
and $Q_{2}$ above. 

For $Q_{1},$ note that Fact \ref{fact:SBM_renyi_equivalence} implies
$\sqrt{\snr}\le C'\frac{\inprob-\outprob}{\sqrt{\inprob}}$ for some
constant $C'>0$ and therefore 
\[
Q_{1}=\error\frac{\inprob-\outprob}{\inprob-\outprob}\sqrt{\frac{\inprob}{\num}}\le\frac{1}{C'}\error(\inprob-\outprob)\sqrt{\frac{1}{\num\snr}}.
\]
Bounding $Q_{2}$ involves some elementary manipulation.

\paragraph{Controlling $Q_{2}$.}

We record an elementary inequality.
\begin{lem}
\label{lm:SBM_simple_ineq}There exists a constant $\consts\ge1$
such that if $\num\snr\ge\consts$, then $\inprob e^{-5\inprob\num/2}\leq(\inprob-\outprob)e^{-5\num\snr/4}.$
\end{lem}
\begin{proof}
Note that $\inprob\num\ge(\inprob-\outprob)\num\geq\frac{(\inprob-\outprob)^{2}}{\inprob}\num\ge\frac{1}{C'}\num\snr\geq\frac{1}{C'}\consts$
for some constant $C'>0$ by Fact~\ref{fact:SBM_renyi_equivalence}.
As long as $\consts$ is sufficiently large, we have $\frac{\inprob\num}{2}\leq e^{5\inprob\num/4}$.
These inequalities imply that 
\[
\frac{\inprob}{\inprob-\outprob}\leq\frac{\inprob\num}{2}\leq e^{5\inprob\num/4}\leq e^{5(2\inprob-\snr)\num/4}.
\]
Multiplying both sides by $(\inprob-\outprob)e^{-5\inprob\num/2}$
yields the claimed inequality.
\end{proof}
Equipped with the above bound, we are ready to bound $Q_{2}$. Let
$\kappa\coloneqq e^{-(1-\xi)\num\snr/2}$ where $\xi\coloneqq\conste\sqrt{\frac{1}{\num\snr}}$
for some constant $\conste>0$ such that $\left\lfloor \num\kappa\right\rfloor >0$.
 If $\frac{\error}{\num}=0$ or $\frac{\error}{\num}\le\left\lfloor \num\kappa\right\rfloor $,
then the first inequality in Proposition \ref{prop:S2} holds and
we are done. It remains to consider the case $\frac{\error}{\num}>\left\lfloor \num\kappa\right\rfloor >0$.
We have that $\left\lfloor \num\kappa\right\rfloor $ is a positive
integer and $\error>\num\left\lfloor \num\kappa\right\rfloor \ge\frac{1}{2}\num^{2}\kappa$.
We therefore have
\begin{align*}
Q_{2}\le\inprob\num^{2}e^{-5\inprob\num/2} & \overset{(a)}{\le}(\inprob-\outprob)\num^{2}e^{-5\num\snr/4}\\
 & \le(\inprob-\outprob)e^{-\num\snr/2}\cdot\num^{2}e^{-(1-\xi)\num\snr/2}\\
 & \le2(\inprob-\outprob)e^{-\num\snr/2}\cdot\error\\
 & \overset{(b)}{\le}2\error(\inprob-\outprob)\sqrt{\frac{1}{\num\snr}},
\end{align*}
where step $(a)$ is due to Lemma~\ref{lm:SBM_simple_ineq}, and
step $(b)$ holds because $\num\snr\ge\consts$ for $\consts$ sufficiently
large. 

\paragraph{Putting together.}

Combining the above bounds for $Q_{1}$ and $Q_{2}$, we obtain that
\[
S_{2}\le C_{2}\error(\inprob-\outprob)\sqrt{\frac{1}{\num\snr}}
\]
for some constant $C_{2}>0$. Under the assumption $0<c_{0}\inprob\le\outprob<\inprob\le1-c_{1}$,
we have $\inprob-\outprob\le\frac{C'\snr}{t^{*}}$ for a constant
$C'>0$ by Facts~\ref{fact:SBM_tstar_lesssim_(p-q)/p} and \ref{fact:SBM_renyi_equivalence}.
This completes the proof of Proposition \ref{prop:S2} for SBM.

\subsection{Proof of Lemma \ref{lem:pilot_bound}}

In this section, we establish the pilot bound in Lemma \ref{lem:pilot_bound}
under each of the three models.

\subsubsection{Proof of Lemma \ref{lem:pilot_bound} for Model \ref{mdl:Z2} (Z2)
\label{sec:proof_Z2_pilot_bound}}

Since $\Yhat\in\betterset(\A)$, we have 
\[
0\le\left\langle \Yhat-\Ystar,\Adj\right\rangle =\left\langle \Yhat-\Ystar,\E\Adj\right\rangle +\left\langle \Yhat-\Ystar,\Adj-\E\Adj\right\rangle .
\]
By Fact \ref{fact:pos_neg_entries_in_Y-Ystar} with $\Y=\Yhat$ and
the fact that $\E\Adj=\Ystar$, we have $\left\langle \Yhat-\Ystar,\E\Adj\right\rangle =-\error.$
Combining, we have the bound $\error\le\left\langle \Yhat-\Ystar,\Adj-\E\Adj\right\rangle .$
We proceed by controlling the RHS as 
\[
\error\le\Tr(\Yhat)\cdot\opnorm{\Noise}+\Tr(\Ystar)\cdot\opnorm{\Noise}=2\num\opnorm{\Noise},
\]
which follows inequality~(\ref{eq:holder}) applied to positive semidefinite
matrices $\Yhat$ and $\Ystar$ satisfying $\Tr(\Yhat)=\Tr(\Ystar)=\num$.
Applying the spectral norm bound in Lemma \ref{lm:Z2_spectral_bound_W},
we obtain that with probability at least $1-e^{-\num/2}$, 
\[
\error\le8\num\sqrt{\std^{2}\num}=4\sqrt{2}\sqrt{\frac{\num^{3}}{\snr}},
\]
where the last step follows from the definition of $\snr$ in Equation
(\ref{eq:snr}). The proof is completed. 

\subsubsection{Proof of Lemma \ref{lem:pilot_bound} for Model \ref{mdl:CBM} (CBM)
\label{sec:proof_CBM_pilot_bound}}

Recall that we have introduced the shorthands  $\inprob\coloneqq\obsprob(1-\flipprob)$
and $\outprob\coloneqq\obsprob\flipprob$. Since $\Yhat\in\betterset(\A)$,
we have 
\[
0\le\left\langle \Yhat-\Ystar,\Adj\right\rangle =\left\langle \Yhat-\Ystar,\E\Adj\right\rangle +\left\langle \Yhat-\Ystar,\Adj-\E\Adj\right\rangle .
\]
By Fact \ref{fact:pos_neg_entries_in_Y-Ystar} and the fact that $\E\Adj=(\inprob-\outprob)\Ystar$,
we have $\left\langle \Yhat-\Ystar,\E\Adj\right\rangle =-(\inprob-\outprob)\error.$
Combining, we have the bound $\error\le\frac{1}{\inprob-\outprob}\left\langle \Yhat-\Ystar,\Adj-\E\Adj\right\rangle .$
To control the RHS, we compute 
\[
\left\langle \Yhat-\Ystar,\Adj-\E\Adj\right\rangle \leq2\sup_{\mathbf{Y}\succeq0,\textmd{diag}(\mathbf{Y})\le\onevec}\left|\left\langle \mathbf{Y},\Adj-\E\Adj\right\rangle \right|.
\]
Grothendieck's inequality~\cite{Grothendieck,Lindenstrauss} guarantees
that 
\[
\sup_{\mathbf{Y}\succeq0,\textmd{diag}(\mathbf{Y})\le\onevec}\left|\left\langle \mathbf{Y},\Adj-\E\Adj\right\rangle \right|\leq K_{G}\norm[\Adj-\E\Adj]{\infty\to1}
\]
where $K_{G}$ denotes Grothendieck's constant ($0<K_{G}\leq1.783$)
and 
\[
\norm[\Adj-\E\Adj]{\infty\to1}:=\sup_{\mathbf{x}:\|\mathbf{x}\|_{\infty}\leq1}\norm[(\Adj-\E\Adj)\mathbf{x}]1=\sup_{\mathbf{y},\mathbf{z}\in\{\pm1\}^{\num}}\left|\mathbf{y}^{\top}(\Adj-\E\Adj)\mathbf{z}\right|.
\]

Set $v^{2}\coloneqq\sum_{1\leq i\le j\leq\num}\Var(\adj_{ij})$ and
note that $\left|\adj_{ij}-\E\adj_{ij}\right|\le2$ for $i,j\in[\num]$.
For each pair of fixed vectors $\mathbf{y},\mathbf{z}\in\{\pm1\}^{\num}$,
the Bernstein inequality ensures that for each number $t\geq0$, 
\[
\P\left\{ \left|\mathbf{y}^{\top}(\Adj-\E\Adj)\mathbf{z}\right|>t\right\} \leq2\exp\left\{ -\frac{t^{2}}{2v^{2}+4t/3}\right\} .
\]
Setting $t=\sqrt{16\num v^{2}}+\frac{8}{3}\num$ gives
\[
\P\left\{ \left|\mathbf{y}^{\top}(\Adj-\E\Adj)\mathbf{z}\right|>\sqrt{16\num v^{2}}+\frac{8}{3}\num\right\} \leq2e^{-2\num}.
\]
Applying the union bound and using the fact that $v^{2}\leq\obsprob\frac{\num^{2}+\num}{2}=\frac{\inprob}{1-\flipprob}\cdot\frac{\num^{2}+\num}{2}$,
we obtain that with probability at least $1-2^{2\num}\cdot2e^{-2\num}=1-2(e/2)^{-2\num}$,
\[
\norm[\Adj-\E\Adj]{\infty\to1}\le2\sqrt{2\frac{\inprob}{1-\flipprob}(\num^{3}+\num^{2})}+\frac{8}{3}\num.
\]

Combining pieces, we conclude that with probability at least $1-2(e/2)^{-2\num}$,
\[
\left\langle \Yhat-\Ystar,\Adj-\E\Adj\right\rangle \leq8\sqrt{2\frac{\inprob}{1-\flipprob}(\num^{3}+\num^{2})}+\frac{32}{3}\num;
\]
whence 
\begin{align*}
\error & \leq\frac{1}{\inprob-\outprob}\left(8\sqrt{2\frac{\inprob}{1-\flipprob}(\num^{3}+\num^{2})}+\frac{32}{3}\num\right)\\
 & \overset{(a)}{\le}\frac{45}{\inprob-\outprob}\sqrt{\frac{\inprob}{1-\flipprob}\num^{3}}\overset{(b)}{\le}\frac{45}{C'}\sqrt{\frac{\num^{3}}{\snr(1-\flipprob)}},
\end{align*}
for some constant $C'>0$, where step $(a)$ holds by our assumption
$\inprob:=\obsprob(1-\flipprob)\ge\frac{1-\flipprob}{\num}$, and
step $(b)$ follows from Fact \ref{fact:CBM_renyi_equivalence}. The
proof is completed in view of the assumption of Model \ref{mdl:CBM}
that $\flipprob$ is a constant. 

\subsubsection{Proof of Lemma \ref{lem:pilot_bound} for Model \ref{mdl:SBM} (SBM)
\label{sec:proof_SBM_pilot_bound}}

The proof follows similar arguments as those in Section~\ref{sec:proof_CBM_pilot_bound}.
Since $\Yhat\in\betterset(\A)$, we have 
\begin{align*}
0\le\left\langle \Yhat-\Ystar,\Adj\right\rangle  & \overset{(a)}{=}\left\langle \Yhat-\Ystar,\Adj-\frac{\inprob+\outprob}{2}\OneMat\right\rangle \\
 & =\left\langle \Yhat-\Ystar,\E\Adj-\frac{\inprob+\outprob}{2}\OneMat\right\rangle +\left\langle \Yhat-\Ystar,\Adj-\E\Adj\right\rangle 
\end{align*}
where step $(a)$ holds since $\left\langle \OneMat,\Yhat\right\rangle =\left\langle \OneMat,\Ystar\right\rangle $.
By Fact \ref{fact:pos_neg_entries_in_Y-Ystar} and the fact that $\E\Adj-\frac{\inprob+\outprob}{2}\OneMat=\frac{\inprob-\outprob}{2}\Ystar$,
we have
\[
\left\langle \Yhat-\Ystar,\E\Adj-\frac{\inprob+\outprob}{2}\OneMat\right\rangle =-\frac{\inprob-\outprob}{2}\error.
\]
Therefore, we have the bound $\error\le\frac{2}{\inprob-\outprob}\left\langle \Yhat-\Ystar,\Adj-\E\Adj\right\rangle .$
To control the RHS, we apply Grothendieck's inequality~\cite{Grothendieck,Lindenstrauss}
to obtain 
\[
\left\langle \Yhat-\Ystar,\Adj-\E\Adj\right\rangle \leq2\sup_{\mathbf{Y}\succeq0,\textmd{diag}(\mathbf{Y})\le\onevec}\left|\left\langle \mathbf{Y},\Adj-\E\Adj\right\rangle \right|\leq2K_{G}\norm[\Adj-\E\Adj]{\infty\to1},
\]
where $K_{G}$ is Grothendieck's constant ($0<K_{G}\leq1.783$) and
\[
\norm[\Adj-\E\Adj]{\infty\to1}:=\sup_{\mathbf{x}:\|\mathbf{x}\|_{\infty}\leq1}\norm[(\Adj-\E\Adj)\mathbf{x}]1=\sup_{\mathbf{y},\mathbf{z}\in\{\pm1\}^{\num}}\left|\mathbf{y}^{\top}(\Adj-\E\Adj)\mathbf{z}\right|.
\]

Set $v^{2}\coloneqq\sum_{1\leq i<j\leq\num}\Var(\adj_{ij})$. For
each pair of fixed vectors $\mathbf{y},\mathbf{z}\in\{\pm1\}^{\num}$,
the Bernstein inequality ensures that for each number $t\geq0$, 
\[
\P\left\{ \left|\mathbf{y}^{\top}(\Adj-\E\Adj)\mathbf{z}\right|>t\right\} \leq2\exp\left\{ -\frac{t^{2}}{2v^{2}+4t/3}\right\} .
\]
Setting $t=\sqrt{16\num v^{2}}+\frac{8}{3}\num$ gives
\[
\P\left\{ \left|\mathbf{y}^{\top}(\Adj-\E\Adj)\mathbf{z}\right|>\sqrt{16\num v^{2}}+\frac{8}{3}\num\right\} \leq2e^{-2\num}.
\]
Applying the union bound and using the fact that $v^{2}\leq\inprob(\num^{2}+\num)/2$,
we obtain that with probability at least $1-2^{2\num}\cdot2e^{-2\num}=1-2(e/2)^{-2\num}$,
\[
\norm[\Adj-\E\Adj]{\infty\to1}\le2\sqrt{2\inprob(\num^{3}+\num^{2})}+\frac{8}{3}\num.
\]

Combining pieces, we conclude that with probability at least $1-2(e/2)^{-2\num}$,
\begin{align*}
\error & \leq\frac{2}{\inprob-\outprob}\cdot2K_{G}\cdot\left(2\sqrt{2\inprob(\num^{3}+\num^{2})}+\frac{8}{3}\num\right)\\
 & \overset{(a)}{\le}\frac{45\sqrt{\inprob\num^{3}}}{\inprob-\outprob}\le\frac{45}{C'}\sqrt{\frac{\num^{3}}{\snr}},
\end{align*}
for some constant $C'>0$, where step $(a)$ holds by our assumption
$\inprob\ge\frac{1}{\num}$ and the last step follows from Fact \ref{fact:SBM_renyi_equivalence}.
The proof is completed.

\subsection{Proof of Lemma \ref{lem:make_gamma_integer} \label{sec:proof_lem_make_gamma_integer}}

If $\phi_{i}\ge0$ for all $i\in[\left\lceil \beta\right\rceil ]$,
then the result follows immediately. Now we assume that at least one
of $\{\phi_{i}\}$ is negative. Define $w\coloneqq\argmin\{i\in[\left\lceil \beta\right\rceil ]:\phi_{i}<0\}$
to be the smallest index of negative $\phi_{i}$. If $\beta_{0}\in[1,w-1]$,
we have $V(\beta_{0})\ge0$ since $\phi_{i}\ge0$, $\forall i\in[1,w-1]$.
If $\beta_{0}\in[w-1,\beta]$, we note that $V$ is decreasing on
$[w-1,\beta]$ since $\phi_{i}<0$, $\forall i\in[w,\left\lceil \beta\right\rceil ]$,
hence $V(\beta_{0})\ge V(\beta)\ge0$. The proof is completed. 

\subsection{Proof of Lemma \ref{lem:subset_sum_ineq} \label{sec:proof_subset_sum_ineq}}

Recall that $X_{i}:=-D_{ii}=-\labelstar_{i}\sum_{j\in\left[\num\right]}\adj_{ij}^{0}\labelstar_{j}$.
 For clarity of exposition, we define the shorthands 
\begin{align*}
L_{m} & \coloneqq\max_{\calM\subset\left[\num\right],\ \left|\calM\right|=m}\left[\sum_{i\in\calM}X_{i}\right],\qquad\text{for }m\in[\left\lfloor M\right\rfloor ],\\
L_{m,\calM} & \coloneqq\sum_{i\in\calM}X_{i},\qquad\text{for }\calM\subset[\num]\text{ with }\left|\calM\right|=m,\\
R_{m} & \coloneqq\frac{1}{t^{*}}\left((1+\eta)m\log\left(\frac{\num e}{m}\right)-(1-2\eta)m\proxnum\snr\right),\qquad\text{for }m\in[\left\lfloor M\right\rfloor ],\\
P_{m,\calM} & \coloneqq\P\left(L_{m,\calM}\ge R_{m}\right),\qquad\text{for }\calM\subset[\num]\text{ with }\left|\calM\right|=m,\\
P_{m} & \coloneqq\P\left(\exists\calM\subset[\num],\left|\calM\right|=m:L_{m,\calM}\ge R_{m}\right),\\
P & \coloneqq\P\left(\exists m\in[\left\lfloor M\right\rfloor ]:L_{m}\ge R_{m}\right).
\end{align*}
Our goal is to show that $P\le3\exp\left(-\sqrt{\log\num}\right)$.
We start the proof by controlling $P_{m,\calM}$ for a fixed $\calM\subset[\num]$
with $\left|\calM\right|=m$. 

\subsubsection{A closer look at $L_{m,\protect\calM}$}

 For fixed $m$ and $\calM$, the quantity $L_{m,\calM}$ is the
sum of $m\num$ random variables: $L_{m,\calM}=\sum_{j\in[m\num]}V_{j}$.
A technicality is that due to the symmetry of the matrix $\Adj$,
there may exist some $j\ne j'\in[m\num]$ such that $V_{j}$ and $V_{j'}$
identify the same random variable. Let us define a set to group together
all such random variables. We set 
\begin{align*}
\calJ & \coloneqq\left\{ j\in[m\num]:\exists j'\in[m\num]\backslash\{j\}\text{ s.t. }V_{j}=V_{j'}\right\} .
\end{align*}
We also define the complement of $\calJ$ as $\calJ{}^{\complement}\coloneqq[m\num]\backslash\calJ.$
Note that  
\[
m_{1}\coloneqq\left|\calJ{}^{\complement}\right|=m\num-m^{2}+m\qquad\text{and}\qquad m_{2}\coloneqq\frac{1}{2}\left|\calJ\right|=\frac{m(m-1)}{2}.
\]
It is not hard to see that $\{V_{j}:j\in\calJ^{\complement}\}$ and
half of $\{V_{j}:j\in\calJ\}$ are independent. Now we can write
\[
L_{m,\calM}=\sum_{j\in\calJ}V_{j}+\sum_{j\in\calJ{}^{\complement}}V_{j}.
\]

\subsubsection{Controlling $P_{m,\protect\calM}$}

Recall that $t^{*}$ defined in Equation (\ref{eq:def_tstar}) satisfies
$t^{*}>0$. Using the Chernoff bound, we have 
\begin{align*}
P_{m,\calM} & =\P\left(L_{m,\calM}\ge R_{m}\right)\\
 & =\P\left(\exp\left(t^{*}L_{m,\calM}\right)\ge\exp\left(t^{*}R_{m}\right)\right)\\
 & \le\exp\left(-t^{*}R_{m}\right)\cdot\left[\E\exp\left(t^{*}\sum_{j\in\calJ{}^{\complement}}V_{j}\right)\right]\cdot\left[\E\exp\left(t^{*}\sum_{j\in\calJ}V_{j}\right)\right]\\
 & \eqqcolon Q_{1}Q_{2}Q_{3}.
\end{align*}
It suffices to control $Q_{1}$, $Q_{2}$ and $Q_{3}$. By definition
of $R_{m}$, we have 
\[
Q_{1}=\exp\left[-(1+\eta)m\log\left(\frac{\num e}{m}\right)+(1-2\eta)m\proxnum\snr\right].
\]
As our main step, we show that the following bounds for $Q_{2}$ and
$Q_{3}$ hold for all three models.
\begin{lem}
\label{lem:subset_sum_Qs}Under the assumption in Lemma \ref{lem:subset_sum_ineq},
we have 
\[
Q_{2}\le\exp\left[-(1-\eta)m\proxnum\snr\right]\qquad\text{and}\qquad Q_{3}=1.
\]
\end{lem}
\noindent The proof of the lemma is model-dependent, and is given
in Sections~\ref{sec:proof_Z2_subset_sum_Qs}, \ref{sec:proof_CBM_subset_sum_Qs}
and~\ref{sec:proof_SBM_subset_sum_Qs} for Models \ref{mdl:Z2},
\ref{mdl:CBM} and~\ref{mdl:SBM}, respectively. Combining the above
bounds for $Q_{1},Q_{2}$ and $Q_{3}$, we obtain 
\begin{align*}
P_{m,\calM} & \le\exp\left[-(1+\eta)m\log\left(\frac{\num e}{m}\right)+(1-2\eta)m\proxnum\snr\right]\cdot\exp\left[-(1-\eta)m\proxnum\snr\right]\cdot1\\
 & =\exp\left[-(1+\eta)m\log\left(\frac{\num e}{m}\right)-\eta m\proxnum\snr\right].
\end{align*}

\subsubsection{Controlling $P_{m}$ and $P$}

Using the above bound on $P_{m,\calM}$ and applying the union bound,
we have
\begin{align*}
P_{m}\le\sum_{\calM\subset[\num]:\left|\calM\right|=m}P_{m,\calM} & \le\binom{\num}{m}\exp\left[-(1+\eta)m\log\left(\frac{\num e}{m}\right)-\eta m\proxnum\snr\right]\\
 & \overset{(a)}{\le}\exp\left[-\eta m\log\left(\frac{\num e}{m}\right)-\eta m\proxnum\snr\right]\\
 & \overset{(b)}{\le}\left\{ \exp\left[-C\sqrt{\frac{1}{\num\snr}}\log\left(\frac{\num e}{m}\right)-\frac{C}{2}\sqrt{\num\snr}\right]\right\} ^{m},
\end{align*}
where step $(a)$ holds due to $\binom{\num}{m}\le\left(\frac{e\num}{m}\right)^{m}$,
and step $(b)$ holds due to the definition that $\eta:=C\sqrt{\frac{1}{\num\snr}}$
and the fact that $\proxnum\ge\frac{\num}{2}$. We proceed by considering
two cases: ($i$) If $m\le\sqrt{\num}$, then 
\[
C\sqrt{\frac{1}{\num\snr}}\log\left(\frac{\num e}{m}\right)+\frac{C}{2}\sqrt{\num\snr}\ge\frac{C}{2}\sqrt{2\log\left(\frac{\num e}{m}\right)}\ge\sqrt{\log\num},
\]
where the last step holds since $C\ge2$. We hence have $P_{m}\le\left[\exp\left(-\sqrt{\log\num}\right)\right]^{m}<\frac{1}{2}$.
($ii$) If $m>\sqrt{\num}$, then 
\[
C\sqrt{\frac{1}{\num\snr}}\log\left(\frac{\num e}{m}\right)+\frac{C}{2}\sqrt{\num\snr}\ge\frac{C}{2}\sqrt{\num\snr}\ge\log10
\]
under the assumption in Lemma~\ref{lem:subset_sum_ineq} that $C\ge2\sqrt{2}$
and $\num\snr\ge\consts\ge4$. We hence have $P_{m}\le\left[\exp\left(-\log10\right)\right]^{m}=\frac{1}{10^{m}}.$
Combining the two cases and applying union bound, we conclude that
\begin{align*}
P & \le\sum_{1\le m\le\sqrt{\num}}P_{m}+\sum_{\sqrt{\num}<m\le\num}P_{m}\\
 & \le\sum_{1\le m<\infty}\left[\exp\left(-\sqrt{\log\num}\right)\right]^{m}+\num\cdot\frac{1}{10^{\sqrt{\num}}}\\
 & \le\frac{\exp\left(-\sqrt{\log\num}\right)}{1-\exp\left(-\sqrt{\log\num}\right)}+\frac{1}{\num}\\
 & \le2\exp\left(-\sqrt{\log\num}\right)+\exp\left(-\log\num\right)\le3\exp\left(-\sqrt{\log\num}\right)
\end{align*}
as desired. This completes the proof of Lemma \ref{lem:subset_sum_ineq}.

\subsubsection{Proof of Lemma \ref{lem:subset_sum_Qs} for Model \ref{mdl:Z2} (Z2)
\label{sec:proof_Z2_subset_sum_Qs}}

Recall the random variable $\var\sim N(1,\std^{2})$ defined in Section~\ref{sec:preliminary}.
We need the following fact.
\begin{fact}
\label{fact:Z2_magic_identity}Under Model \ref{mdl:Z2}, we have
the following identities 
\begin{align*}
\E e^{t^{*}(-\var)} & =e^{-\snr}\qquad\text{and}\qquad\E e^{2t^{*}(-\var)}=1.
\end{align*}
\end{fact}
\begin{proof}
Recall the definitions $\snr:=(2\std^{2})^{-1}$ and $t^{*}:=\frac{1}{\std^{2}}$
in Equations~(\ref{eq:snr}) and~(\ref{eq:def_tstar}). The results
follow from direct calculation:
\[
\E e^{t^{*}(-\var)}=\exp\left[-t^{*}+\frac{1}{2}\std^{2}\left(t^{*}\right)^{2}\right]=e^{-\snr},\qquad\text{and}\qquad\E e^{2t^{*}(-\var)}=\exp\left[-2t^{*}+2\std^{2}\left(t^{*}\right)^{2}\right]=1.
\]
\end{proof}
To proceed, note that each of $\{V_{j}:j\in[m\num]\}$ is distributed
as $-\var$.

\paragraph*{Controlling $Q_{2}$.}

We have 
\[
Q_{2}:=\E\exp\left[t^{*}\sum_{j\in\calJ^{\complement}}V_{j}\right]\overset{(a)}{=}\exp\left[-m_{1}\snr\right]=\exp\left[-(m\num-m^{2}+m)\snr\right]
\]
where step $(a)$ follows from Fact \ref{fact:Z2_magic_identity}.
If $1\ge C\sqrt{\frac{\num}{\snr}}$, then we must have $m=\left\lfloor M\right\rfloor =1$
and 
\[
Q_{2}=\exp\left[-m\num\snr\right]\le\exp\left[-(1-\eta)m\num\snr\right].
\]
If $1\le C\sqrt{\frac{\num}{\snr}}$, then we have $1\le M\le C\sqrt{\frac{\num}{\snr}}$
and 

\[
Q_{2}\le\exp\left[-(m\num-m^{2})\snr\right]\le\exp\left[-(1-\eta)m\num\snr\right]
\]
where the last step holds since $m\le M\le C\sqrt{\frac{\num}{\snr}}=\num\eta$.
Either way, we have the desired inequality. 

\paragraph*{Controlling $Q_{3}$.}

Fact \ref{fact:Z2_magic_identity} directly implies the desired equality:
\[
Q_{3}:=\E\exp\left[t^{*}\sum_{j\in\calJ}V_{j}\right]=\left(\E e^{2t^{*}V_{1}}\right)^{m_{2}}=1.
\]

\subsubsection{Proof of Lemma \ref{lem:subset_sum_Qs} for Model \ref{mdl:CBM}
(CBM) \label{sec:proof_CBM_subset_sum_Qs}}

Recall the definition of the random variable $\var$ in Section~\ref{sec:preliminary}.
We need the following fact, whose proof is deferred to the end of
this section.
\begin{fact}
\label{fact:CBM_magic_identity}Under Model \ref{mdl:CBM}, we have
the following identities 
\[
\E e^{t^{*}(-\var)}=1-\snr\qquad\text{and}\qquad\E e^{2t^{*}(-\var)}=1.
\]
\end{fact}
To proceed, note that each of $\{V_{j}:j\in[m\num]\}$ is distributed
as $-\var$.

\paragraph*{Controlling $Q_{2}$.}

We have
\begin{align*}
Q_{2}:=\E\exp\left[t^{*}\sum_{j\in\calJ^{\complement}}V_{j}\right] & \overset{(a)}{=}(1-\snr)^{m_{1}}\\
 & \overset{(b)}{=}\exp\left[(m\num-m^{2}+m)\log(1-\snr)\right]\\
 & \overset{(c)}{\le}\exp\left[-(m\num-m^{2}+m)\snr\right]\\
 & \overset{(d)}{\le}\exp\left[-(1-\eta)m\num\snr\right],
\end{align*}
where step $(a)$ follows from Fact \ref{fact:CBM_magic_identity},
step $(b)$ follows from the fact that $m_{1}=m\num-m^{2}+m$, step
$(c)$ holds since $\log(1-x)\le-x,\forall x<1$, and step $(d)$
holds since $m\le M\le C\sqrt{\frac{\num}{\snr}}=\num\eta$ when $1\le C\sqrt{\frac{\num}{\snr}}$,
or $m=\left\lfloor M\right\rfloor =1$ when $1\ge C\sqrt{\frac{\num}{\snr}}$.
 We thus obtain the desired bound on $Q_{2}$. 

\paragraph*{Controlling $Q_{3}$.}

Fact \ref{fact:CBM_magic_identity} directly implies the desired equality:
\[
Q_{3}:=\E\exp\left[t^{*}\sum_{j\in\calJ}V_{j}\right]=\left(\E e^{2t^{*}(-\var)}\right)^{m_{2}}=1.
\]

\begin{proof}[Proof of Fact~\ref{fact:CBM_magic_identity}]
 Recall the shorthands $\inprob\coloneqq\obsprob(1-\flipprob)$ and
$\outprob\coloneqq\obsprob\flipprob$ introduced for Model~\ref{mdl:CBM};
note that $\obsprob=\inprob+\outprob$. Also recall the definitions
$\snr:=(\sqrt{\obsprob(1-\flipprob)}-\sqrt{\obsprob\flipprob})^{2}=(\sqrt{\inprob}-\sqrt{\outprob})^{2}$
and $t^{*}:=\frac{1}{2}\log\frac{1-\flipprob}{\flipprob}$ in Equations~(\ref{eq:snr})
and (\ref{eq:def_tstar}). The results follow from direct calculation:
\begin{align*}
\E e^{t^{*}(-\var)} & =(1-\obsprob)+\inprob e^{-t^{*}}+\outprob e^{t^{*}}\\
 & =(1-\obsprob)+2\sqrt{\inprob\outprob}=1-\snr,
\end{align*}
and 
\begin{align*}
\E e^{2t^{*}(-\var)} & =(1-\obsprob)+\inprob e^{-2t^{*}}+\outprob e^{2t^{*}}\\
 & =(1-\obsprob)+\outprob+\inprob=1.
\end{align*}
\end{proof}

\subsubsection{Proof of Lemma \ref{lem:subset_sum_Qs} for Model \ref{mdl:SBM}
(SBM) \label{sec:proof_SBM_subset_sum_Qs}}

We record the following fact, whose proof is deferred to the end of
this section.
\begin{fact}
\label{fact:SBM_magic_identity}Let $Z\sim\Bern(\outprob)$ and $Y\sim\Bern(\inprob)$.
We have the following identities 
\begin{align*}
\E e^{t^{*}Z}\E e^{-t^{*}Y} & =e^{-\snr},\\
\left(\E e^{t^{*}Z}\right)^{\frac{1}{2}}\left(\E e^{-t^{*}Y}\right)^{-\frac{1}{2}}e^{-t^{*}\tune} & =1,\\
\E e^{2t^{*}Z}\E e^{-2t^{*}Y} & =1,\\
\left(\E e^{2t^{*}Z}\right)^{\frac{1}{2}}\left(\E e^{-2t^{*}Y}\right)^{-\frac{1}{2}}e^{-2t^{*}\tune} & =1.
\end{align*}
\end{fact}
Let $\{Z_{j}\},\{Z_{j}'\}\overset{\text{i.i.d.}}{\sim}\Bern(\outprob)$
and independently $\{Y_{j}\},\{Y_{j}'\}\overset{\text{i.i.d.}}{\sim}\Bern(\inprob)$.
Note that each of $\{V_{j}:j\in[m\num]\}$ is distributed as either
$Z_{1}-\tune$ or $-Y_{1}+\tune$. We define the quantities 
\begin{align*}
m_{\inprob} & \coloneqq\left|\left\{ j\in\calJ^{\complement}:V_{j}\sim-Y_{1}+\tune\right\} \right|,\\
m_{\outprob} & \coloneqq\left|\left\{ j\in\calJ^{\complement}:V_{j}\sim Z_{1}-\tune\right\} \right|,\\
m_{\inprob}' & \coloneqq\frac{1}{2}\left|\left\{ j\in\calJ:V_{j}\sim-Y_{1}+\tune\right\} \right|,\\
m_{\outprob}' & \coloneqq\frac{1}{2}\left|\left\{ j\in\calJ:V_{j}\sim Z_{1}-\tune\right\} \right|.
\end{align*}
Note that $m_{\inprob}+m_{\outprob}=\left|\calJ^{\complement}\right|=m_{1}:=m\num-m^{2}+m$
and $m_{\inprob}'+m_{\outprob}'=\frac{1}{2}\left|\calJ\right|.$

\paragraph{Controlling $Q_{2}$.}

Expanding the definition of $Q_{2}$, we have 
\begin{align*}
Q_{2} & =\E\exp\left(t^{*}\sum_{j\in\left[m_{\outprob}\right]}(Z_{j}-\tune)-t^{*}\sum_{j\in\left[m_{\inprob}\right]}(Y_{j}-\tune)\right)\\
 & =e^{-t^{*}\tune(m_{\outprob}-m_{\inprob})}\left(\E e^{t^{*}Z_{1}}\right)^{m_{\outprob}}\left(\E e^{-t^{*}Y_{1}}\right)^{m_{\inprob}}\\
 & =\left(\E e^{t^{*}Z_{1}}\E e^{-t^{*}Y_{1}}\right)^{\frac{1}{2}m_{\inprob}+\frac{1}{2}m_{\outprob}}\left(\left(\frac{\E e^{t^{*}Z_{1}}}{\E e^{-t^{*}Y_{1}}}\right)^{\frac{1}{2}}e^{-t^{*}\tune}\right)^{m_{\outprob}-m_{\inprob}}.
\end{align*}
By Fact \ref{fact:SBM_magic_identity}, we can continue to write 
\begin{align*}
Q_{2} & \le\exp\left(-\left(\frac{1}{2}m_{\inprob}+\frac{1}{2}m_{\outprob}\right)\snr\right)\\
 & \le\exp\left(-\frac{1}{2}(m\num-m^{2}+m)\snr\right)\\
 & \le\exp\left(-(1-\eta)\frac{m\num}{2}\snr\right),
\end{align*}
where the last step holds since $m\le M\le C\sqrt{\frac{\num}{\snr}}=\num\eta$
when $1\le C\sqrt{\frac{\num}{\snr}}$, or $m=\left\lfloor M\right\rfloor =1$
when $1\ge C\sqrt{\frac{\num}{\snr}}$. We thus obtain the desired
bound on $Q_{2}$. 

\paragraph{Controlling $Q_{3}$.}

Similar to controlling $Q_{2}$, we compute 
\begin{align*}
Q_{3} & =\E\exp\left(2t^{*}\sum_{j\in\left[m_{\outprob}'\right]}(Z_{j}'-\tune)-2t^{*}\sum_{j\in\left[m_{\inprob}'\right]}(Y_{j}'-\tune)\right)\\
 & =e^{-2t^{*}\tune(m_{\outprob}'-m_{\inprob}')}\left(\E e^{2t^{*}Z_{1}'}\right)^{m_{\outprob}'}\left(\E e^{-2t^{*}Y_{1}'}\right)^{m_{\inprob}'}\\
 & =\left(\E e^{2t^{*}Z_{1}'}\E e^{-2t^{*}Y_{1}'}\right)^{\frac{1}{2}m_{\inprob}'+\frac{1}{2}m_{\outprob}'}\left(\left(\frac{\E e^{2t^{*}Z_{1}'}}{\E e^{-2t^{*}Y_{1}'}}\right)^{\frac{1}{2}}e^{-2t^{*}\tune}\right)^{m_{\outprob}'-m_{\inprob}'}=1,
\end{align*}
where the last step holds due to Fact \ref{fact:SBM_magic_identity}. 
\begin{proof}[Proof of Fact~\ref{fact:SBM_magic_identity}]
Under  Model~\ref{mdl:SBM}, recall the definitions $\snr:=-2\log\left[\sqrt{\inprob\outprob}+\sqrt{(1-\inprob)(1-\outprob)}\right]$,
$t^{*}:=\frac{1}{2}\log\frac{\inprob(1-\outprob)}{\outprob(1-\inprob)}$
and $\tune:=\frac{1}{2t^{*}}\log\frac{1-\outprob}{1-\inprob}$ in
Equations~(\ref{eq:snr}), (\ref{eq:def_tstar}) and~(\ref{eq:def_tune_param}),
respectively. For the first equation, we compute
\begin{align*}
\E e^{t^{*}Z}\E e^{-t^{*}Y} & =\left(\outprob e^{t^{*}}+1-\outprob\right)\left(\inprob e^{-t^{*}}+1-\inprob\right)\\
 & =\inprob\outprob+(1-\inprob)(1-\outprob)+\outprob(1-\inprob)e^{t^{*}}+\inprob(1-\outprob)\inprob e^{-t^{*}}\\
 & =\inprob\outprob+(1-\inprob)(1-\outprob)+2\sqrt{\inprob\outprob(1-\inprob)(1-\outprob)}\\
 & =\left(\sqrt{\inprob\outprob}+\sqrt{(1-\inprob)(1-\outprob)}\right)^{2}\\
 & =e^{-\snr}.
\end{align*}
For the second equation, noting that $e^{2t^{*}\tune}=\frac{1-\outprob}{1-\inprob}$,
we compute
\begin{align*}
\frac{\E e^{t^{*}Z}}{\E e^{-t^{*}Y}}\cdot e^{-2t^{*}\tune} & =\frac{\outprob e^{t^{*}}+1-\outprob}{\inprob e^{-t^{*}}+1-\inprob}\cdot\frac{1-\inprob}{1-\outprob}\\
 & =\frac{\outprob\sqrt{\frac{\inprob(1-\outprob)}{\outprob(1-\inprob)}}+1-\outprob}{\inprob\sqrt{\frac{\outprob(1-\inprob)}{\inprob(1-\outprob)}}+1-\inprob}\cdot\frac{1-\inprob}{1-\outprob}=1
\end{align*}
and then take the square root of both sides. Finally, the remaining
two equations follow from $e^{2t^{*}\tune}=\frac{1-\outprob}{1-\inprob}$
and the identities
\begin{align*}
\E e^{2t^{*}Z} & =\outprob e^{2t^{*}}+1-\outprob=\outprob\frac{\inprob(1-\outprob)}{\outprob(1-\inprob)}+1-\outprob=\frac{1-\outprob}{1-\inprob},\\
\E e^{-2t^{*}Y} & =\inprob e^{-2t^{*}}+1-\inprob=\inprob\frac{\outprob(1-\inprob)}{\inprob(1-\outprob)}+1-\inprob=\frac{1-\inprob}{1-\outprob}.
\end{align*}
\end{proof}

\section{Proof of the second inequality in Theorem \ref{thm:SDP_error}\label{sec:proof_cluster_error_rate}}

Fix any $\Yhat\in\betterset(\Adj)$. Note that $\Yhat,\Ystar\in[-1,1]^{\num\times\num}$
by feasibility to the program (\ref{eq:CBM_Z2_SDP}) or (\ref{eq:SBM_SDP}).
It follows that 
\[
\norm[\Yhat-\Ystar]F^{2}\le\max_{i,j\in[\num]}\left\{ \left|\yhat_{ij}-\ystar_{ij}\right|\right\} \cdot\sum_{i,j\in[\num]}\left|\yhat_{ij}-\ystar_{ij}\right|=2\norm[\Yhat-\Ystar]1.
\]
Combining with the first inequality of Theorem \ref{thm:SDP_error},
we obtain 
\begin{align*}
\norm[\Yhat-\Ystar]F^{2} & \le\num^{2}\cdot2\exp\left[-\left(1-\conste\sqrt{\frac{1}{\num\snr}}\right)\proxnum\snr\right]=:n^{2}\cdot\varepsilon.
\end{align*}
Let $\hat{\v}$ be an eigenvector of $\Yhat$ corresponding to the
largest eigenvalue with $\norm[\hat{\v}]2^{2}=\num$. It can be seen
that the largest eigenvalue of $\Ystar$ is $\num$ with $\LabelStar$
being the corresponding eigenvector, and  that all the other eigenvalues
are 0. Because $\Yhat=\Ystar+(\Yhat-\Ystar)$ and $\norm[\Yhat-\Ystar]F\le\sqrt{\varepsilon}\num$,
Davis-Kahan theorem (see, e.g., \cite[Corollary 3]{vu2011singular})
implies that 
\[
\min_{g\in\{\pm1\}}\norm[g\hat{\u}-\u^{*}]2=2\left|\sin\left(\frac{\theta}{2}\right)\right|\le C\sqrt{\varepsilon}
\]
for some absolute constant $C>0$, where $\hat{\u}$ and $\u^{*}$
denote the unit-norm eigenvectors associated with the largest eigenvalues
of $\Yhat$ and $\Ystar$, respectively, and $\theta\in[0,\frac{\pi}{2}]$
denotes the angle between these two vectors. By definition $\hat{\v}=\sqrt{\num}\hat{\u}$
and $\LabelStar=\sqrt{\num}\u^{*}$, we obtain that 
\[
\min_{g\in\{\pm1\}}\norm[g\hat{\v}-\LabelStar]2^{2}\le C^{2}\varepsilon\num.
\]
We proceed by relating $\misrate(\LabelSDP,\LabelStar)$ to $\min_{g\in\{\pm1\}}\norm[g\v-\LabelStar]2^{2}$.
Without loss of generality, assume that the minimum is attained by
$g=1$. Since $\labelsdp_{i}=\sign(\hat{v}_{i})$ by definition, we
have the bound 
\begin{align*}
C^{2}\varepsilon\num\ge\norm[\hat{\v}-\LabelStar]2^{2} & \ge\sum_{i\in[\num]}(\hat{v}_{i}-\labelstar_{i})^{2}\indic\{\sign(\hat{v}_{i})\ne\labelstar_{i}\}\\
 & \ge\sum_{i\in[\num]}\indic\{\sign(\hat{v}_{i})\ne\labelstar_{i}\}\\
 & \ge\num\cdot\misrate(\LabelSDP,\LabelStar).
\end{align*}
We divide both sides of the above equation by $\num$, and note that
the constant $2C^{2}$ can be absorbed into $\conste$ under the assumption
that $n\snr\ge\consts$ for $\consts$ sufficiently large. The result
follows.

\bibliographystyle{plain}
\bibliography{references}

\end{document}